%% file: thesis.tex
\documentclass[tocnosub,noragright,centerchapter,12pt,mixcasechap]{uiucecethesis09}

\makeatletter

\usepackage{setspace}
\phdthesis

\author{Daan Michiels}
\title{Symplectic Foliations, Currents, and Local Lie Groupoids}
\department{Mathematics}
\degreeyear{2018}
\advisor{Rui Loja Fernandes}
\committee{Professor Susan Tolman, Chair \\
Professor Rui Loja Fernandes, Advisor\\
Professor Ely Kerman \\
Professor James Pascaleff}

\usepackage{microtype}
\usepackage{enumerate}
\usepackage{tikz}
\usepackage[utf8]{inputenc}
\usepackage[T2A,T1]{fontenc}
\usepackage{lmodern}
\usepackage{graphicx}
\usepackage{amsfonts, amsmath, amsthm, amssymb}
\usepackage{tensor}
\usepackage{bbm}
\usepackage{subcaption}
\usepackage{comment}
\usepackage{thmtools, thm-restate}
\usepackage{varioref}
\usepackage{hyperref}
\usepackage[noabbrev]{cleveref}
\usepackage{overpic}
\usepackage{listings}
\usepackage[main=english,russian]{babel}

\DeclareFixedFont{\ttb}{T1}{txtt}{bx}{n}{9} 
\DeclareFixedFont{\ttm}{T1}{txtt}{m}{n}{9}  
\usepackage{color}
\definecolor{lightgrey}{rgb}{0.7,0.7,0.7}
\definecolor{deepblue}{rgb}{0,0,0.5}
\definecolor{deepred}{rgb}{0.6,0,0}
\definecolor{deepgreen}{rgb}{0,0.5,0}
\newcommand\pythonstyle{\lstset{
language=Python,
basicstyle=\ttm,
numbers=left,
stepnumber=1,
numberstyle=\itshape\color{lightgrey},
commentstyle=\itshape\color{lightgrey},
otherkeywords={as, self, yield},
keywordstyle=\ttb\color{deepblue},
emph={__name__},
emphstyle=\ttb\color{deepred},
stringstyle=\color{deepgreen},
frame=tb,
showstringspaces=false
}}
\lstnewenvironment{python}[1][]
{
\pythonstyle
\lstset{#1}
}
{}

\newcommand{\Z}{\ensuremath{\mathbb{Z}}}
\newcommand{\R}{\ensuremath{\mathbb{R}}}
\renewcommand{\C}{\ensuremath{\mathbb{C}}}
\newcommand{\F}{\mathcal{F}}
\newcommand{\HH}{\mathcal{H}}
\usepackage[hmargin=1.2in, vmargin=1.2in]{geometry}
\newcommand{\Sphere}{\mathbb{S}}
\newcommand{\Torus}{\mathbb{T}}
\newcommand{\Disk}{\mathbb{D}}
\newcommand{\dR}{\text{dR}}
\newcommand{\cl}{\overline}
\newcommand{\ind}{\mathbbm{1}}
\newcommand{\Li}{\mathcal{L}}

\newcommand{\eps}{\varepsilon}
\newcommand{\G}{\mathcal{G}}
\newcommand{\g}{\mathfrak{g}}
\newcommand{\Cur}{\mathcal{C}}
\newcommand{\ZZ}{\mathcal{Z}}
\newcommand{\BB}{\mathcal{B}}
\newcommand{\V}{\mathcal{V}}
\renewcommand{\U}{\mathcal{U}}
\newcommand{\dd}[1]{\frac{\partial}{\partial{#1}}}
\newcommand{\timesst}{\tensor[_s]{\times}{_t}}

\newcommand{\tto}{\rightrightarrows}
\newcommand{\slot}{-}
\newcommand{\realization}[1]{{\lvert {#1}\rvert}}
\newcommand{\Nerve}{\mathcal{N}}
\newcommand{\AsCo}{\mathcal{AC}}
\newcommand{\Mon}{\tilde{\mathcal{N}}}
\newcommand{\Orbit}{\mathcal{O}}
\newcommand{\into}{\hookrightarrow}

\DeclareMathOperator{\id}{Id}
\DeclareMathOperator{\im}{Im}
\DeclareMathOperator{\spn}{span}
\DeclareMathOperator{\interior}{int}

\DeclareMathOperator{\Hol}{Hol}
\DeclareMathOperator{\Verts}{Vert}
\DeclareMathOperator{\Assoc}{Assoc}

\DeclareMathOperator{\ConvexHull}{ConvexHull}

\usetikzlibrary{matrix,arrows,calc,shapes,decorations}
\usetikzlibrary{patterns,decorations.pathreplacing}
\usetikzlibrary{decorations.markings}
\usetikzlibrary{arrows.meta}

\newtheorem{theorem}{Theorem}
\newtheorem{maintheorem}{Theorem}
\newtheorem{lemma}[theorem]{Lemma}
\newtheorem{corollary}[theorem]{Corollary}
\newtheorem{proposition}[theorem]{Proposition}
\newtheorem{conjecture}[theorem]{Conjecture}
\newtheorem*{conjecture*}{Conjecture}
\theoremstyle{definition}
\newtheorem*{example-plain}{Example}
\newtheorem*{remark}{Remark}
\newtheorem{definition}[theorem]{Definition}

\newcommand\xqed[1]{%
  \leavevmode\unskip\penalty9999 \hbox{}\nobreak\hfill
  \quad\hbox{#1}}
\newcommand\afterexample{\xqed{$\blacktriangleleft$}}
\newenvironment{example}[1]{%
    \begin{example-plain}#1}{%
    \afterexample\end{example-plain}%
}

\begin{document}

\maketitle
\parindent 1em%

\frontmatter

\begin{abstract}
    \input{chapters/abstract}

\end{abstract}

\begin{dedication}
    to my parents
\end{dedication}

\input{chapters/acknowledgements}

\tableofcontents

\mainmatter
\include{chapters/introduction}

\include{chapters/background}

\include{chapters/calibrations}
\include{chapters/currents}

\include{chapters/llg}

\include{chapters/associativity}

\include{chapters/classification}

\include{chapters/integrability}

\backmatter
\bibliographystyle{plain}
\bibliography{thesis}

\appendix
\include{chapters/appendix}

\end{document}

%% file: chapters/abstract.tex
This thesis treats two main topics: calibrated symplectic foliations, and local
Lie groupoids.
Calibrated symplectic foliations are one possible generalization of taut foliations
of 3-manifolds to higher dimensions. Their study has been popular in recent years,
and we collect several interesting results.
We then show how de Rham's theory of currents, and Sullivan's theory of structure currents,
can be applied in trying to understand the calibratability of symplectic foliations.

Our study of local Lie groupoids begins with their definition
and an exploration of some of their basic properties.
Next, three main results are obtained.
The first is the generalization of a theorem by Mal'cev. The original theorem
characterizes the local Lie groups that are part of a (global) Lie group. We
give the corresponding result for local Lie groupoids.
The second result is the generalization of a theorem by Olver which classifies
local Lie groups in terms of Lie groups. Our generalization classifies, in
terms of Lie groupoids, those local Lie groupoids that have integrable algebroids.
The third and final result demonstrates a relationship between the
associativity of a local Lie groupoid, and the integrability of its algebroid.
In a certain sense, the monodromy groups of a Lie algebroid manifest themselves
combinatorially in a local integration, as a lack of associativity.

%% file: chapters/acknowledgements.tex
\chapter*{Acknowledgements}

This thesis would not have been possible without the support of many people.
I wish to express my gratitude to all of them.

First and foremost, I would like to thank my advisor.
Rui, you have been a kind and wise mentor throughout my doctoral studies.
Your advice and your guidance have been invaluable, and your support has always been
much appreciated. Thank you.

I would also like to thank the members of the examination committee for taking the time
to read my work, and for their helpful questions and comments.

I would like to thank all other people at the mathematics department of the University of
Illinois at Urbana-Champaign.
All of you have contributed to this project in one way or another, be it administratively,
through teaching, or through one of countless other ways.

I would also like to thank Arno Kuijlaars, who was the first person to suggest I go to
the United States for my PhD, and Marius Crainic, who suggested I work with Rui.

In the five years I spent working toward this thesis, I have had the pleasure
of meeting many other mathematicians from all over the world. If I try to list
them all here, I will inevitably forget some. I am grateful for all
math-related and non-math-related conversations we had.

Doing mathematical research is no easy task, and I often found myself at the climbing
wall in an effort to clear my head. I would like to thank Matt for all the good times
we had climbing, and for his friendship.

I would like to thank my entire family for their love and support.
I am eternally grateful to my parents, in particular, for raising me
with incredible warmth and giving me all the opportunities anyone could ever want.
You are truly exceptional.

Finally, I would like to thank Annelies.
Living an ocean apart has not been easy, and I am
unable to count the hours we have spent on Skype.
I want to thank you for your support, and for everything you do for me.
Your love makes me a better person.

%% file: chapters/introduction.tex
\chapter*{Introduction}
\addcontentsline{toc}{chapter}{Introduction}

This thesis addresses two main topics: calibrations of symplectic foliations,
and the associativity of local Lie groupoids.
These topics are related to each other through the study of (symplectic) Lie groupoids.
On the one hand, symplectic foliations are precisely regular Poisson
structures, and Poisson manifolds are the infinitesimal counterpart to
symplectic Lie groupoids.
On the other hand, local Lie groupoids generalize Lie groupoids.

Let us sketch the broad context in which these topics are studied,
and give a brief summary of the main results in this thesis.
At the end of this introduction, we will outline the organization of this thesis.

\subsection*{Calibrated symplectic foliations}

A symplectic foliation is, roughly speaking, a foliation whose leaves are equipped
with a symplectic structure that varies smoothly from leaf to leaf.
Alternatively, one can consider a symplectic foliation as a regular Poisson structure.
One interesting problem about symplectic foliations is:
\begin{center}
\itshape Which manifolds admit a symplectic foliation of a given dimension?
\end{center}
For open manifolds, the answer is understood, in the sense that there is an $h$-principle
governing the existence of symplectic foliations \cite{fernandes-frejlich}.
For closed manifolds, however, this problem is very hard, and understanding is limited.
Indeed, in codimension 0 the problem reduces to the existence of symplectic structures on
closed manifolds, known to be a tough problem.
In \cite{bertelson}, an example is given of a foliation that does not admit a
symplectic structure despite all basic obstructions vanishing. The thesis
\cite{osornotorres} is dedicated to the study of symplectic foliations in
codimension 1.

Because the class of symplectic foliations is large, it is natural to focus attention
on a more restricted type of symplectic foliation.
One suitable class of structures is that of \emph{calibrated} symplectic foliations:
one requires the existence of a closed 2-form on the manifold that restricts to the given
leafwise symplectic form on the foliation.
(This condition known by different names in the literature. In
\cite{martineztorres}, these symplectic foliations
are called \emph{2-calibrated} to stress that the extra condition on the symplectic
foliation is the existence of a certain 2-form. Another common name is that of a
symplectic foliation \emph{with a closed extension}, to stress that the extra condition
on the symplectic foliation is the existence of a certain closed form.)
The question above then leads to the following problem:
\begin{center}
\itshape Which manifolds admit a calibrated symplectic foliation of a given dimension?
\end{center}
It is the class of calibrated symplectic foliations, and the existence problem
just formulated, that chapters \ref{chapter:calibrations} and
\ref{chapter:currents} will focus on.

An important motivation for studying calibrated symplectic foliations is
the hope that they can provide a suitable generalization of \emph{tautness} of
a foliation to higher dimensions.
Taut foliations are a special class of codimension-1 foliations that have been studied
extensively, particularly in dimension 3.
In dimension $m$, tautness of a foliation amounts to the existence of a closed $(m-1)$-form
whose restriction to each leaf is non-vanishing.
The existence of a taut foliation on a 3-manifold has consequences for the topology
of this manifold (see \cpageref{page:taut-in-dim3}).
While the study of taut foliations has been very successful on 3-manifolds, they
are not as useful in higher dimensions. In particular, the $h$-principle proved in
\cite{meigniez} shows that a compact manifold of dimension at least 4 always admits
a taut foliation, given that it admits a codimension-1 foliation.
Thus taut foliation lose much of their usefulness because they become too flexible
starting in dimension 4.
A codimension-1 foliation in dimension 3 is taut precisely if it admits a calibration,
and starting in dimension 5, admitting a calibration is a strictly stronger condition
than being taut (there are no symplectic codimension-1 foliations in dimension 4).
The more limited flexibility of calibrated symplectic foliations compared to
taut foliations is exactly analogous to the more limited flexibility of
symplectomorphism compared to volume-preserving maps (as exemplified, for example,
by Gromov's non-squeezing theorem \cite{gromov85}).
The hope, therefore, is that calibrated symplectic foliations are the correct way
to generalize tautness in dimension 3 to higher dimensions, and retain sufficient rigidity
to be able to deduce interesting consequences about the topology
of the ambient manifold.

\subsection*{Associativity of local Lie groupoids}

When the theory of Lie groups and Lie algebras was first developed,
Lie groups were studied locally, for no suitable notion of manifold was available
(\cite[pages 286--302]{bourbaki-lie}).
Once the modern notion of a Lie group has been established, then, it is natural to
ask whether every local Lie group is contained in a Lie group.
Perhaps surprisingly, the answer is ``no''.
While it is true that every local Lie group has a neighborhood of the identity
that is contained in a Lie group, this is not necessarily true for the entire local
Lie group.
We say that a local Lie group is \emph{globalizable} if it is contained in a Lie group.
The obstruction to globalizability of a local Lie group has long been understood:
a theorem by Mal'cev proves that a local Lie group is globalizable precisely
if it is \emph{globally associative} \cite{malcev}.

In \cite{olver}, Olver re-proves Mal'cev's result, and describes a construction
that, starting from Lie groups, can be used to obtain any local Lie group.
This effectively gives a sort of classification of local Lie groups.
Roughly speaking, every local Lie group is covered by a cover of a globalizable
local Lie group.
One can wonder whether these results by Mal'cev and Olver have corresponding
generalizations to local Lie groupoids.

As shown in chapters \ref{chapter:associativity} and
\ref{chapter:classification} of this thesis, corresponding results do indeed
exist. Mal'cev's theorem follows as a corollary of the following result.
In this statement, the associative completion $\AsCo(G)$
is a groupoid that can be naturally built for any local Lie groupoid $G$.

\begin{maintheorem}
    If $G$ is a bi-regular local Lie groupoid with products connected to the axes,
    then $\AsCo(G)$ is
    smooth if and only if $\Assoc(G)$ is uniformly discrete in $G$
    (and in that case, $G \to \AsCo(G)$ is a local diffeomorphism).
\end{maintheorem}

Olver's classification also goes through (and we strengthen it slightly).
Note that Olver's construction can only classify those local Lie groupoids
whose Lie algebroid is integrable, because the construction to build local Lie groupoids
starts from a (global) Lie groupoid, and therefore necessarily from an integrable
algebroid.
The statement of the classification is as follows.

\begin{maintheorem}
    Suppose $G$ is a bi-regular $s$-connected local Lie groupoid over $M$
    with integrable Lie algebroid $A$.
    Write $\tilde{G}$ for its source-simply connected cover,
    and write $\G(A)$ for the source-simply connected integration of $A$.
    Then we have the following commutative diagram.
    \begin{center}
    \begin{tikzpicture}
        \matrix(m)[matrix of math nodes, row sep=2.2em, column sep=2.6em]{
            & \tilde{G} & \\
            G & & U\subseteq \G(A) \\
            & \AsCo(G) & \\
        };
        \path[->] (m-1-2) edge node[auto,swap]{$p_1$} (m-2-1.north east);
        \path[->] (m-1-2) edge node[auto]{$p_2$} (m-2-3.north west);
        \path[->] (m-2-1.south east) edge (m-3-2);
        \path[->] (m-2-3.south west) edge (m-3-2);
    \end{tikzpicture}
    \end{center}
    Here, $p_1$ is the covering map,
    and the map $G \to \AsCo(G)$ is the natural map from a local Lie group to its
    associative completion.
    The map $p_2$ is a generalized covering of local Lie groupoids,
    which sends a point $g\in\tilde{G}$ to the class of the $A$-path associated
    to a $\tilde{G}$-path from $s(g)$ to $g$.
\end{maintheorem}

The third theorem about local Lie groupoids concerns the theory of integrability.
Every Lie algebroid can be integrated to a local Lie groupoid \cite{crainic-fernandes-lie}.
If one starts with a non-integrable algebroid, and integrates it to a local Lie groupoid,
one necessarily obtains a non-globalizable local Lie groupoid.
By Mal'cev's theorem, this non-globalizability must manifest itself as a lack of
global associativity.
Because the integrability of a Lie algebroid is controlled by its monodromy groups,
one expects there to be a relationship between the monodromy groups of an algebroid
and the associativity of any of its local integrations.
As we will prove in the last chapter, this relationship can be made precise:
under some technical assumptions, the ``associators'' of a local Lie groupoid
\emph{coincide} with the monodromy groups of its Lie algebroid, at least sufficiently
near the identities.
Here is the statement.

\begin{maintheorem}
    Suppose that $G$ is a shrunk local Lie groupoid over $M$. Let $x\in M$.
    Consider $G_x$ as a subset of $\G(\g_x)$ using the natural map
    $G_x \to \G(\g_x)$.
    Then
    \[ \Assoc_x(G) = \Mon_x \cap G_x .\]
\end{maintheorem}

\subsection*{Outline}

The main content of this thesis is organized into seven chapters.

In \cref{chapter:background}, we discuss some of the background material needed for the
rest of the thesis. It is intended to be a brief review, and the reader who is unfamiliar
with the topics in this first chapter is encouraged to read more comprehensive treatments
before moving on. None of the material in this chapter is original work.

In \cref{chapter:calibrations}, we give an overview of results about calibrated symplectic
foliations. We discuss what it means for a symplectic foliation to be calibrated (or
calibratable), and give some interpretations.
We illustrate by briefly mentioning the special case of cosymplectic structures,
and conclude by discussing a result about the transverse geometry of calibrated
symplectic foliations due to Mart\'inez-Torres et al.\ \cite{martineztorres}.
None of the material in this chapter is new, except perhaps the example on page \pageref{page:example-torus} and \cref{prop:morita-equivalence}.

In \cref{chapter:currents}, we discuss how the theory of currents (in the sense of de Rham)
can be used to try to understand calibratability of symplectic foliations.
This is the last chapter about symplectic foliations.
All of the core results in this chapter are due to Sullivan \cite{sullivan}.
In this thesis, we explain how Sullivan's theory of structure cycles becomes relevant in
understanding calibratability of symplectic foliations.
To the best of the author's knowledge,
no attempts have been made before at understanding calibratability using Sullivan's theory.
It seems likely that the understanding of the local structure of structure cycles
should give some type of answer to the existence question for calibrated symplectic
foliations, but no deep results have been obtained yet.
At the end of the chapter, we unify conjectures by Sullivan and Ghys, and illustrate with
a concrete example.

In \cref{chapter:llg}, we start our discussion of local Lie groupoids.
We define local Lie groupoids, and explain some of their basic properties.
These properties are all straightforward generalizations of corresponding
properties for local Lie groups \cite{olver}.
Then, we explain how one can associate a Lie algebroid to every local Lie groupoid.
While some of the results of this chapter do not seem to have been written down earlier,
no major theorems are proved.

In \cref{chapter:associativity},
we discuss the globalizability of local Lie groupoids.
One of the main results of this thesis is the generalization of Mal'cev's
theorem to local Lie groupoids. It is obtained in this chapter as a corollary of
a theorem describing smoothness of the associative completion of a local Lie groupoid.
The associators of a local Lie groupoid are introduced in this chapter,
and their relationship with the associative completion is explained.
The proof of Mal'cev's theorem for local Lie groupoids can be found in this chapter.
The recognition of associative completion as an important
operation, and the introduction of associators, are important in this chapter
and in the later ones.
The most important paper guiding this chapter was \cite{olver}, where the corresponding
results for local Lie groups are proved (though associators are not introduced there).

In \cref{chapter:classification},
associative completions are used as a tool to classify those local Lie groupoids
that have an integrable Lie algebroid.
We discuss coverings of local Lie groupoids, and we prove the classification result.

In \cref{chapter:integrability},
we prove the last of the three theorems, showing how the associators of a local
Lie groupoid are related to the monodromy groups of its Lie algebroid.
The theorem describing this relationship is the most important result of this thesis.
It demonstrates that the monodromy groups of a Lie algebroid manifest themselves
in a local integration, in a combinatorial way.

\setcounter{maintheorem}{0}

%% file: chapters/background.tex
\chapter{Preliminaries}
\label{chapter:background}

In this first chapter, we review some of the basic material that will be needed later on.
We will also establish some notations that will be used throughout this thesis.

\section{Foliations}

Suppose that $M$ is a smooth manifold of dimension $m$.
A foliation of $M$ is, roughly speaking, a partition of $M$ into
immersed submanifolds of dimension $p$, in such a way that the partition
locally looks like $\R^m$ partitioned into parallel copies of $\R^p$.

\subsection*{Definitions}

There are many ways to define what a foliation is.
We will give two.

\begin{definition}
    A \emph{foliation atlas} of dimension $p$ (and codimension $q=m-p$) of $M$
    ($0\leq p\leq m$) is an atlas
    $(\varphi_i : U_i \to \R^m)_i$
    such that the transition maps are of the form
    \[ \varphi_{ij}(x,y) = (f_{ij}(x,y),g_{ij}(y)) \]
    for the decomposition $\R^m = \R^p \times \R^q$.
\end{definition}

The important part is that the second component of $\varphi_{ij}(x,y)$ does not
depend on $x$.  This definition should be interpreted as follows: $\R^m$
consists of parallel copies of $\R^p$ (a $q$-parameter family of them). The
transition maps of the atlas must respect this decomposition of $\R^m$ into
parallel submanifolds (so that points on the same submanifold are mapped to
points on the same submanifold).  Thus, if a manifold is obtained by gluing
together flexible pieces of fabric, then a foliated manifold is obtained by
gluing together flexible pieces of striped fabric in such a way that the
stripes line up whenever you glue.

\begin{definition}
    A \emph{foliation} of $M$ of dimension $p$ is a maximal foliation atlas of
    dimension $p$ of $M$.
\end{definition}

We will use the letter $\F$ to denote a foliation.
The stripes in a foliation chart are called $\emph{plaques}$.
They are copies of open subsets of $\R^p$.
The plaques fit together globally to partition the manifold into injectively immersed
$p$-dimensional submanifolds of $M$, called the \emph{leaves} of the foliation.
The \emph{leaf space} of $\F$, written $M/\F$, is the quotient space of $M$
obtained by identifying points on the same leaf, equipped with the quotient topology.
It need not be a well-behaved space.

In some situations, it is useful to choose a foliation atlas that is slightly better
behaved than a general foliation atlas.
A \emph{regular} foliation atlas is a foliation atlas that satisfies some mild extra
conditions:
\begin{itemize}
    \item the closure of each chart domain is a compact subset of some foliated chart
        (not necessarily in our atlas),
    \item the covering by chart domains is locally finite,
    \item the closure of a plaque of one chart intersects at most one plaque in each other
        chart.
\end{itemize}
Every foliation has a regular foliation atlas \cite[lemma 1.2.17]{candelconlon}.
We will make use of regular foliation atlases exactly once, on \cpageref{page:where-we-need-a-regular-atlas}.

A foliation is completely determined by the tangent space to the leaf at each point $x\in M$.
These tangent spaces form a subbundle of $TM$, and this subbundle is integrable because the
leaves of $\F$ are integral to this distribution.
This gives another way of defining a foliation.

\begin{definition}[Alternative definition]
    A \emph{foliation} of $M$ of dimension $p$ is an integrable subbundle of $TM$
    of rank $p$.
\end{definition}

Using this definition, the leaves of the foliation are the maximal integral submanifolds
of the integrable distribution.
If we want to think of a foliation $\F$ as a subbundle of $TM$, we will write it as $T\F$
(the tangent bundle of the foliation). Will also use $T^*\F$ for the dual of $T\F$
(the cotangent bundle of the foliation).

We will say that a foliation is \emph{orientable} (\emph{oriented}) if the bundle $T\F$ is orientable (oriented). We will say that a foliation is \emph{co-orientable} (\emph{co-oriented}) if the normal bundle to the foliation, $TM/T\F$, is orientable (oriented).
Co-orientability is also called \emph{transverse orientability}.
If $\F$ is not (co-)orientable, we can obtain a (co-)orientable foliation by going
to an orientation cover.

Every co-orientable foliation of codimension 1 can be specified as the kernel of a $1-form$.
If $\alpha \in \Omega^1(M)$, the kernel of $\alpha$ is an integrable distribution (i.e.\ a
foliation) precisely if $\alpha\wedge d\alpha = 0$.

\subsection*{Examples}

Let us go over some basic examples of foliations.

\begin{example}[Trivial foliation]
    If $L$ is a $p$-dimensional manifold, then $L\times \R^q$ has an obvious foliation.
    Its leaves are $L\times\{y\}$ for $y\in\R^q$. The leaf space is $\R^q$.
    More generally, for any $q$-dimensional manifold $Q$,
    we can foliate $L\times Q$ by leaves $L\times\{q\}$ for $q\in Q$.
    The leaf space is then $Q$.
    Even more generally, if $f : M \to Q$ is a submersion, the fibers of $f$
    form the leaves of a foliation. Then $T\F$ is the kernel of $df : TM \to TQ$.
    The leaf space is $Q$. Foliations of this form are called \emph{simple foliations}.
    A foliation with a smooth leaf space is automatically simple.
\end{example}

\begin{example}[Linear foliation of the 2-torus]
    Consider
    \[ M = \Torus^2 = \frac{\R^2}{\Z^2} ,\]
    the 2-torus, with coordinates $x,y\in \R/\Z$.
    Let $\theta \in \R$ and consider the foliation given by the 1-form
    \[ \alpha = -\sin(\theta)\,dx + \cos(\theta)\,dy .\]
    This form satisfies $\alpha\wedge d\alpha = 0$, and so its kernel is a foliation.
    If $\tan(\theta) \in \mathbb{Q} \cup \{\infty\}$, this is a foliation by circles.
    Otherwise, this is a foliation by lines.
    \label{ex:linear-foliation-torus}
\end{example}

\begin{example}[Reeb foliation]
    This example is one of the most important foliations in geometry.
    \phantomsection\label{example:reeb}
    Let $f : (-1,1) \to \R$ be an even function such that $f(y) \to +\infty$ as $y\to \pm1$.
    Consider the following foliation of $\R\times[-1,1]$:
    \begin{itemize}
        \item for each $x \in \R$ the image of $(-1,1) \to \R\times[-1,1] :
              y\mapsto (x+f(y),y)$ is a leaf,
        \item the boundary components $\R\times\{\pm 1\}$ are leaves.
    \end{itemize}
    This is a smooth foliation of $\R\times[-1,1]$ if $f$ is chosen appropriately.
    (We have not talked about foliations of manifolds with boundary.
    There are two common types of these: either boundary components are leaves,
    as in this example, or the foliation is transverse to the boundary.
    Adapting the definition to the case with boundary is not hard, but
    we won't need it for anything but this example.)
    It is shown in \cref{fig:reeb}.
    We can spin this foliation around the axis $\R\times\{0\}$ to obtain a foliation
    of $\R\times \Disk^2$, and then mod out by translation by 1 unit in the
    $\R$-direction to obtain a foliation of $\Sphere^1\times \Disk^2$, the so-called
    \emph{Reeb foliation} of the solid torus, named after the French mathematician
    Georges Reeb.
    Specifically, the Reeb foliation of the solid torus has the following leaves:
    \begin{itemize}
        \item for each $x \in \Sphere^1$, the graph of $\Disk^2 \to \Sphere^1 \times \Disk^2 :
              (r,\theta) \mapsto (x + f(r), (r,\theta))$ is a leaf (where we have
              used polar coordinates on $\Disk^2$),
        \item the boundary component $\Sphere^1 \times \Sphere^1$ is the only compact leaf.
    \end{itemize}
    The Reeb foliation of the solid torus is shown in \cref{fig:reeb3d}.
    One can glue two such foliated tori together along their boundary to obtain
    a codimension-1 foliation of $\Sphere^3$.  If one does this in such a way that
    meridians in one solid torus correspond to longitudes in the other, and vice
    versa, one obtains the so-called \emph{Reeb foliation} of the 3-sphere.  (To
    ensure smoothness of this foliation, $f(y)$ needs to grow fast enough as
    $y\to 1$.)
\end{example}

\begin{figure}
\center
\begin{tikzpicture}
    \draw[thick,->] (-4,0) -- (4,0) node[anchor=north] {$x$};
    \draw[thick,->] (0,-1.5) -- (0,1.5) node[anchor=east] {$y$};
    \draw (-3.5,1) -- (3.5,1);
    \draw (-3.5,-1) -- (3.5,-1);
    \begin{scope}
        \clip (-3.5,-1) rectangle (3.5,1);
        \foreach \x in {-10, -9, ..., 10} {
            \draw (\x+0.1,-0.9) .. controls (\x-3,-0.85) and (\x-3,0.85) .. (\x+0.1,0.9);
        }
    \end{scope}
\end{tikzpicture}
\caption{A foliation of $\R\times[-1,1]$.
The curved leaves asymptotically approach the boundary components.
The Reeb foliation of the solid torus is obtained by spinning this foliation around the $x$-axis,
and then modding out by a translation in the $x$-direction.}
\label{fig:reeb}
\end{figure}
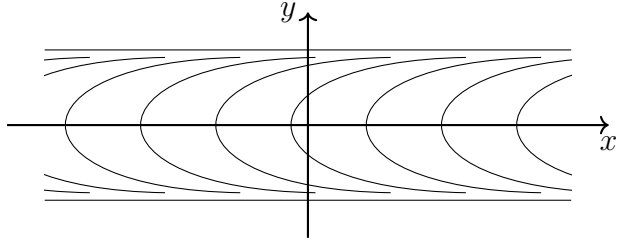

\begin{figure}
\center
\includegraphics[width=0.4\linewidth]{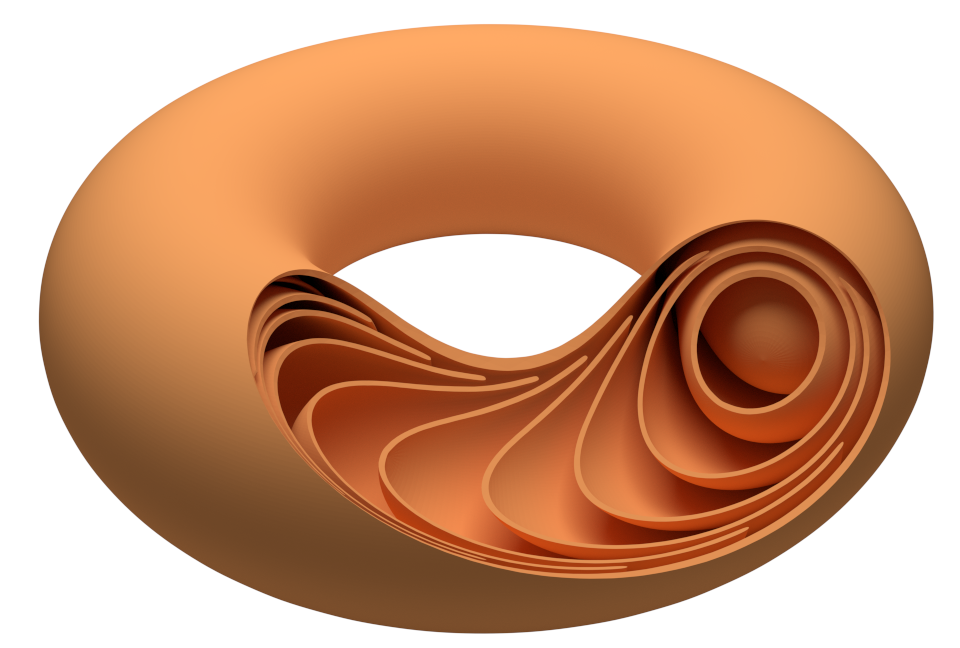}
\caption{A cutaway of the Reeb foliation of the solid torus.}
\label{fig:reeb3d}
\end{figure}

\begin{example}[Lawson's foliations of spheres]
    In \cite{lawson}, Lawson gave smooth codimension-1 foliations for all
    spheres of dimension $2^k+3$ ($k \geq 1$).  In the case $k=1$ this leads to
    a foliation of $\Sphere^5$ with exactly one compact leaf, similar to the Reeb
    foliation of $\Sphere^3$.  This compact
    leaf is diffeomorphic to $\Sphere^1\times L$, where $L$ is a circle bundle
    over a torus.  The compact leaf divides $\Sphere^5$ into two components,
    and each of these components is foliated by leaves all of which are
    diffeomorphic within their component.  The foliation is transversely
    oriented.
\end{example}

\begin{example}[Turbulization, see e.g.\ {{\cite[example 3.3.11]{candelconlon}}}]
    Let $(M,\F)$ be a foliated manifold of codimension 1,
    oriented and co-oriented.
    Suppose that $S$ is a properly embedded submanifold of $M$ of dimension 1,
    tranverse to $\F$ (the typical case is when $S$ is a circle transverse to $\F$).
    Fix a foliated tubular neighborhood $U$ of $S$ (this is a tubular neighborhood of $S$
    that is isomorphic to $S \times \Disk^{m-1}$ as foliated manifolds).
    Use coordinates $(z,u,r)$ on $S \times \Disk^{m-1}$, where $z$ is the coordinate
    along $S$, and the coordinates $u\in \Sphere^{m-2}$ and $r\in [0,1]$ are the direction
    and radius in $\Disk^{m-1}$.
    (These coordinates are not smooth where $r=0$, but this will not matter for us.)
    The foliation of $S\times\Disk^{m-1}$ is defined by $dz$.
    Now consider the 1-form
    \[ \alpha = \sin(\lambda(r)) \,dr + \cos(\lambda(r)) \,dz ,\]
    where $\lambda : [0,1] \to [0,\pi]$ is smooth, non-decreasing, and satisfies
    \begin{itemize}
        \item $\lambda(r) = 0$ if $0\leq r\leq 1/4$,
        \item $\lambda(1/2) = \pi/2$,
        \item $\lambda(r) = \pi$ if $3/4 \leq r\leq 1$,
        \item $\lambda$ is strictly increasing on $[1/4,3/4]$.
    \end{itemize}
    \Cref{fig:lambda-turbulization} shows the graph of such a function.
    This 1-form satisfies $\alpha\wedge d\alpha = 0$, and therefore defines
    a foliation. Moreover, this foliation agrees with the original one near $r=1$ (and also near $r=0$).
    We may therefore replace the foliation $\F$ in the tubular neighborhood of $S$
    by the one defined by $\alpha$.
    This results in a new foliation $\F'$ on $M$.
    Note that $\F'$ has a leaf diffeomorphic to $S \times \Sphere^{m-2}$ at $r=1/2$.
    This procedure of changing the foliation in a neighborhood of a 1-dimensional
    submanifold is called \emph{turbulization}.

    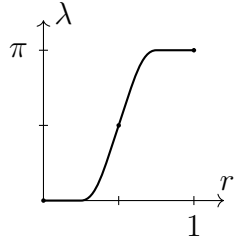
\begin{figure}
        \centering
        \begin{tikzpicture}[scale=2.0]
            \draw[->] (0,0) -- (1.2,0);
            \draw[->] (0,0) -- (0,1.2);
            \node[anchor=south west] at (0,1.1) {$\lambda$};
            \node[anchor=south west] at (1.1,0) {$r$};
            \draw[thick] (0,0) -- (0.25,0) .. controls (0.35,0) and (0.4,0.2) .. (0.5,0.5) .. controls (0.6,0.8) and (0.65,1) .. (0.75,1) -- (1,1);
            \fill[thick] (0.5,0.5) circle (0.015);
            \fill[thick] (0,0) circle (0.015);
            \fill[thick] (1,1) circle (0.015);
            \draw (0.5,-0.03) -- (0.5,0.03);
            \draw (-0.03,0.5) -- (0.03,0.5);
            \draw (1,-0.03) -- (1,0.03);
            \draw (-0.03,1) -- (0.03,1);
            \node[anchor=north] at (1,-0.03) {$1$};
            \node[anchor=east] at (-0.03,1) {$\pi$};
        \end{tikzpicture}
        \caption{The function $\lambda$ used in the turbulization procedure.}
        \label{fig:lambda-turbulization}
    \end{figure}
    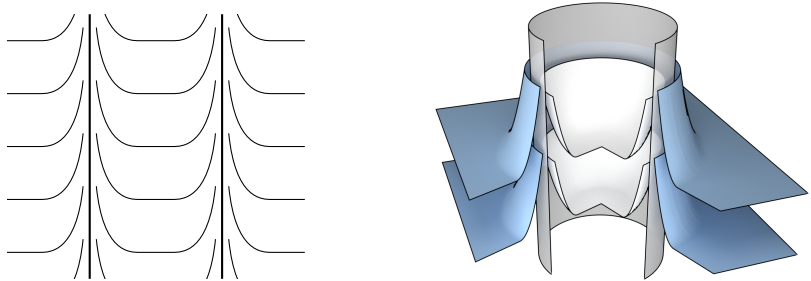
\begin{figure}
        \centering
        \begin{minipage}[b]{.4\linewidth}
            \centering
            \begin{tikzpicture}[scale=0.88]
                \draw[thick,-] (-1,-2) -- (-1,2);
                \draw[thick,-] (1,-2) -- (1,2);
                \begin{scope}
                    \clip (-3.5,-2) rectangle (3.5,2);
                    \foreach \y in {-5, -4, ..., 5} {
                        \draw (-0.9,0.8*\y+1) .. controls (-0.8,0.8*\y) and (-0.4,0.8*\y) .. (-0.25,0.8*\y) -- (0.25,0.8*\y) .. controls (0.4,0.8*\y) and (0.8,0.8*\y) .. (0.9,0.8*\y+1);
                        \draw (1.1,0.8*\y+1) .. controls (1.2,0.8*\y) and (1.6,0.8*\y) .. (1.75,0.8*\y) -- (2.25,0.8*\y);
                        \draw (-2.25,0.8*\y) -- (-1.75,0.8*\y) .. controls (-1.6,0.8*\y) and (-1.2,0.8*\y) .. (-1.1,0.8*\y+1);
                    }
                \end{scope}
            \end{tikzpicture}
        \end{minipage}%
        \begin{minipage}[b]{.4\linewidth}
            \centering
            \includegraphics[width=0.8\linewidth]{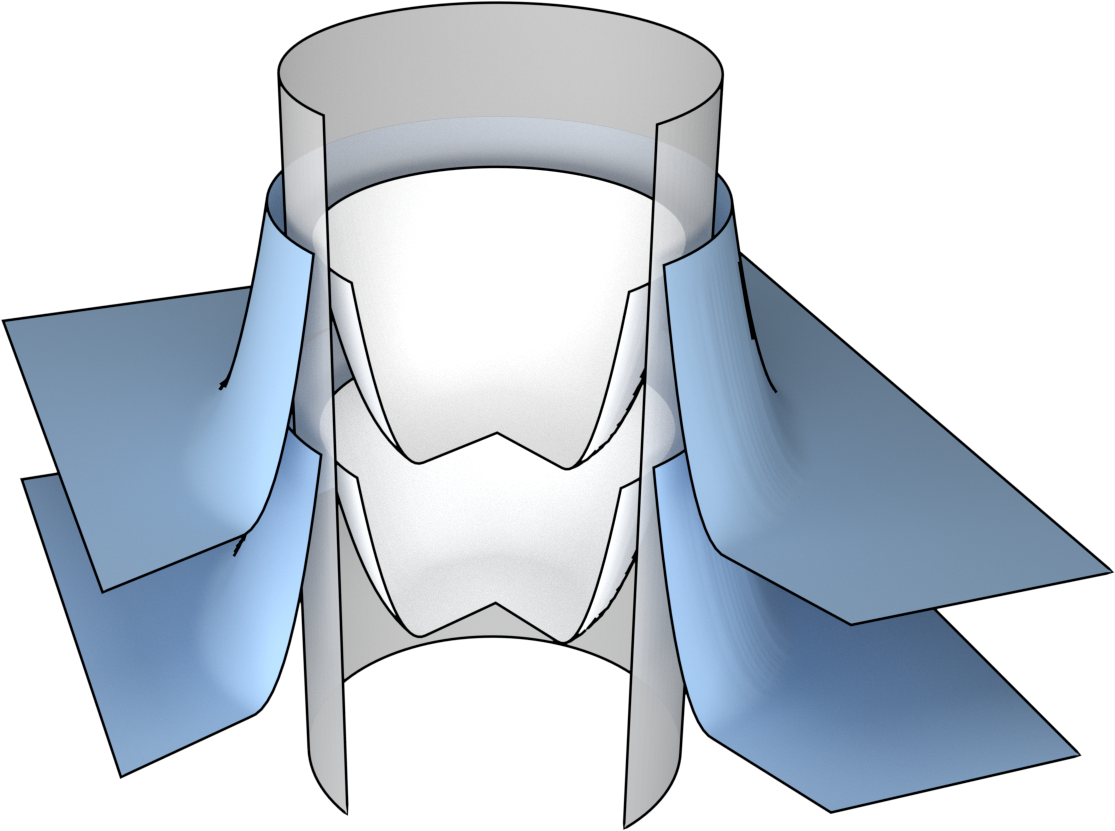}
        \end{minipage}%
        \caption{\emph{Left:} Turbulization in the two-dimensional case. The
            original foliation consisted of horizontal lines. The submanifold
            $S$ was a vertical line centered in the picture. The two vertical
            leaves correspond to $r=1/2$. \emph{Right:} Turbulization in the
            three-dimensional case (cutaway). The original foliation consisted of
            horizontal planes. The submanifold $S$ is a vertical line, the axis
            of the cylinder. There is a cylindrical leaf that corresponds to
            $r=1/2$.  The blue leaves on the outside approach the cylindrical
            leaf asymptotically.  The white cups on the inside are
            diffeomorphic to the plane, and they also approach the cylindrical
            leaf. The white cups together with the cylinder form the Reeb
            foliation of $\Sphere^1 \times \Disk^2$.}
    \end{figure}
\end{example}

\subsection*{Holonomy}

Suppose $\gamma : [0,1] \to M$ is a curve from $x\in M$ to $y\in M$ that lies
in a leaf of $\F$. Pick a small submanifold $T_x$ of $M$ through $x$,
transverse to $\F$, and a small submanifold $T_y$ of $M$ through $y$,
transverse to $\F$.  Then the foliation induces a diffeomorphism $f$ from a
small neighborhood $U_x$ of $x$ in $T_x$ to a small neighborhood $U_y$ of $y$
in $T_y$, as follows: suppose $x' \in T_x$ is sufficiently close to $x$.  By
moving the curve $\gamma$ slightly, we obtain a curve $\gamma'$ that starts at
$x'$, lies in a leaf of the foliation, and ends at a point $y' \in T_y$. Then
$f(x')=y'$.

\begin{definition}
    The \emph{holonomy of $\F$ along $\gamma$} for the transversals $T_x$ and
    $T_y$ is the germ of this diffeomorphism $f : U_x \to U_y$.
    The \emph{linear holonomy of $\F$ along $\gamma$} is the first-order
    approximation to the holonomy of $\F$ along $\gamma$, considered as a map
    $\nu_x\F \to \nu_y\F$.
\end{definition}

The linear holonomy does not depend on the choice of transversals.
Moreover, the holonomy of $\F$ along a curve only depends on the homotopy class of
this curve, if we take homotopies in the leaf and relative to the endpoints.
The linear holonomy, therefore, also depends only on this homotopy class.

\subsection*{Foliated cohomology}

A \emph{foliated differential $k$-form} on the foliated manifold $(M,\F)$
is a section of $\Lambda^kT^*\F$. A foliated $k$-form induces an ordinary
$k$-form on each leaf.
There is a foliated version of the de Rham differential, called the
\emph{leafwise de Rham differential} and written $d_{\F}$, which simply takes the
exterior derivative of a foliated form in a leafwise way.
The \emph{foliated cohomology} of a foliation $\F$ is the cohomology of the complex
\[ \cdots \to \Omega^{k-1}(\F) \xrightarrow{d_{\F}} \Omega^k(\F) \xrightarrow{d_{\F}} \Omega^{k+1}(\F) \to \cdots \]
This cohomology will be denoted by $H^{\bullet}(\F)$.

If $E$ is a vector bundle over $M$, an \emph{$\F$-connection} on $E$ (sometimes
called a \emph{partial connection}) has the same definition as an ordinary
connection, except that one can only differentiate
\emph{along} $\F$ (not in transverse directions).
In other words, it is a bilinear map
\[ \Gamma(T\F) \times \Gamma(E) \to \Gamma(E) : (\alpha,s) \mapsto \nabla_{\alpha}(s) \]
satisfying the usual conditions: it is $C^\infty(M)$-linear in $\alpha$, and it satisfies the Leibniz rule
\[ \nabla_{\alpha}(fs) = f\nabla_{\alpha}(s) + \Li_{\alpha}(f) s \]
with respect to $s$.
If $E$ is a vector bundle over $M$ with a flat $\F$-connection $\nabla$, then
we can define the cohomology of $\F$ with values in $E$. It is the cohomology
of the complex
\[ \cdots \to \Omega^{k-1}(\F;E) \xrightarrow{d_E} \Omega^k(\F;E) \xrightarrow{d_E} \Omega^{k+1}(\F;E) \to \cdots \]
where the differential $d_E$ uses the $\F$-connection.  For $\alpha \in
\Omega^k(\F;E)$ it is given by
\begin{align*}
(d_E\alpha)(V_0, \ldots, V_k) &= \sum_i (-1)^i \nabla_{V_i}(\alpha(V_0,\ldots,\hat{V}_i,\ldots,V_k)) \\
                              &+ \quad\sum_{i<j} (-1)^{i+j}\alpha([V_i,V_j],V_0,\ldots,\hat{V}_i,\ldots,\hat{V}_j,\ldots,V_k) .
\end{align*}
For this differential we have $d_E^2 = 0$ because the connection is flat.
The cohomology is denoted by $H^{\bullet}(\F;E)$.

The normal and conormal bundles of a foliation come equipped with a canonical flat
connection, the \emph{Bott connection}, which is such that parallel transport for this
connection corresponds to linear holonomy.
We will denote the normal bundle to $\F$ by $\nu(\F)$ and the conormal one by $\nu^*(\F)$.
The cohomology with values in these bundles is then written $H^{\bullet}(\F,\nu(\F))$ and $H^{\bullet}(\F,\nu^*(\F))$ (where the Bott connection is implicitly used).

\subsection*{Existence problem}

A driving problem in the early theory of foliations was the existence problem:
which manifolds admit foliations of a certain codimension?
For example, the 2-sphere does not have a foliation of dimension 1, because this
would imply the existence of a non-vanishing vector field on $\Sphere^2$.

The case of open manifolds was settled in 1970 by Haefliger \cite{haefliger70},
who gave a necessary and sufficient condition under which a distribution
can be homotoped to a foliation.
Let $BO_n$ be the classifying space of the orthogonal group $O_n$,
and let $B\Gamma_q$ be the classifying space of the groupoid of germs of
diffeomorphisms of $\R^q$. Then there is a commutative diagram

\begin{center}
    \begin{tikzpicture}
        \matrix(m)[matrix of math nodes, row sep=3.2em, column sep=0.0em]{
            && B\Gamma_q \times BO_p \\
            &BO_q \times BO_p & \\
            && BO_m \\
        };
        \path[->] (m-1-3) edge node[auto,swap]{} (m-2-2);
        \path[->] (m-1-3) edge node[auto]{} (m-3-3);
        \path[->] (m-2-2) edge node[auto]{} (m-3-3);
    \end{tikzpicture}
\end{center}

\begin{theorem}[Haefliger]
Let $M$ be a smooth manifold.
Let $\tau : M \to BO_m$ be a map classifying the tangent bundle of $M$.
Then there is a natural map from the space of codimension-$q$
foliations on $M$, up to integrable homotopy (i.e.\ homotopy through foliations),
to the space of homotopy classes of
lifts of $\tau$ to $B\Gamma_q \times BO_{m-q}$.
\end{theorem}

In a diagram, the situation is as follows: the map $\tau$ classifies the tangent bundle
of $M$.
Splitting this bundle as a sum of a rank-$p$ and rank-$q$ subbundles amounts to
finding a lift $\tau'$. The rank-$p$ subbundle is homotopic to an integrable
distribution (i.e.\ a foliation) precisely if $\tau'$ can further be lifted to $\tau''$.

\begin{center}
    \begin{tikzpicture}
        \matrix(m)[matrix of math nodes, row sep=3.2em, column sep=0.0em]{
             &[4em] & B\Gamma_q \times BO_p \\
            M&[4em] BO_q \times BO_p & \\
             &[4em] & BO_m \\
        };
        \path[->] (m-1-3) edge node[auto,swap]{} (m-2-2);
        \path[->] (m-1-3) edge node[auto]{} (m-3-3);
        \path[->] (m-2-2) edge node[auto]{} (m-3-3);
        \path[bend right=20,->] (m-2-1) edge node[auto,swap]{$\tau$} (m-3-3);
        \path[dashed,->] (m-2-1) edge node[auto]{$\tau'$} (m-2-2);
        \path[bend left=20,dashed,->] (m-2-1) edge node[auto]{$\tau''$} (m-1-3);
    \end{tikzpicture}
\end{center}

The existence of foliations on closed manifolds was solved in the seventies by Thurston
\cite{thurston74, thurston76}.
The result from 1974 settles the case $q>1$.
It uses the language of Haefliger structures (we do not go into detail).

\begin{theorem}[Thurston]
Let $M$ be a smooth manifold, $\HH$ a Haefliger $\Gamma_q$-structure on $M$
with $q>1$, and $i : \nu(\HH) \into TM$ a bundle monomorphism.
Then there is a foliation $\F$ of codimension $q$ on $M$ that is homotopic to $\HH$.
\end{theorem}

The result from 1976 settles the case $q=1$.

\begin{theorem}[Thurston]
    A closed manifold $M$ has a foliation of codimension 1 if and only if
    its Euler characteristic is zero.
\end{theorem}

\subsection*{Taut foliations}

We now restrict our attention to the case of co-oriented codimension-1 foliations.
Among these, there is a special type of foliation called a \emph{taut} foliation.
These special foliations have been extensively studied, particularly in the case $m=3$,
because the existence of taut foliations on 3-manifolds has implications for the
topology of the manifold.
Tautness of a foliation should be thought of as the counterpart of \emph{tightness}
in contact topology.

\begin{definition}
A co-oriented foliation $\F$ of codimension 1 on $M$ is called \emph{taut}
if for each leaf $L$, there is a loop in $M$ that is transverse to $\F$ and
intersects $L$.
\end{definition}

Some authors require that there exist a transverse loop that intersects
all leaves of $\F$. On compact manifolds, the case that is most relevant when studying
taut foliations, this definition is equivalent.

There are several criteria that are equivalent to tautness
(lemma 6.3.4, proposition 10.4.1 and corollary 10.5.10 in \cite{candelconlon}).

\begin{theorem}
If $\F$ is a smooth co-oriented codimension-1 foliation on a compact manifold,
the following are equivalent:
\begin{enumerate}[(a)]
\item $\F$ is taut, \label{item:criteria-tautness-taut}
\item no finite family of compact leaves $\{L_i\}_i$ bounds a submanifold
      $Q \subset M$ such that $Q$ lies on the positive side of every $L_i$,
\item no non-trivial foliation cycle for $\F$ bounds, \label{item:criteria-tautness-cycle}
\item there is a closed $(m-1)$-form which restricts to a volume form on each leaf, \label{item:criteria-tautness-form}
\item there is a vector field transverse to $\F$ whose flow preserves some
      volume form on $M$,
  \item there exists a Riemannian metric on $M$ such that every leaf of $\F$ is
      a minimal submanifold of $M$.
\end{enumerate} \label{thm:criteria-tautness}
\end{theorem}

These criteria will serve as a useful guide in trying to generalize the notion of
tautness in dimension 3 using symplectic foliations.

Let us give an example of how the existence of a taut foliation gives information
about the topology of the ambient manifold. \phantomsection\label{page:taut-in-dim3}

\begin{theorem}[Novikov, \cite{novikov}]
Let $\F$ be a codimension-1 foliation on a compact connected oriented manifold
$M$ of dimension 3. If $M$ has finite fundamental group, then $\F$ has a Reeb
component. In particular, $\F$ is not taut.
\end{theorem}

A \emph{Reeb component} is a solid torus
$\Sphere^1 \times \Disk^2$ embedded in $M$
that is saturated (i.e.\ is a union of leaves),
whose foliation is topologically conjugate to the Reeb foliation of the solid torus
(example on page~\pageref{example:reeb}).
If a foliation contains a Reeb component, it cannot be taut: any transverse
curve that crosses into the torus cannot come out again. The Reeb component
is effectively a dead end.

Here is another simple observation about taut foliations on simply connected manifolds
(not necessarily of dimension 3).

\begin{lemma}
    Let $M$ be a closed simply connected manifold.
    If $\F$ is a taut foliation on $M$, then it has no compact leaves.
\end{lemma}
\begin{proof}
    Suppose $L$ is a compact leaf. We claim that $M \setminus L$ is disconnected.
    Indeed, a tubular neighborhood of $L$ is diffeomorphic to $L\times (-1,1)$
    ($\F$ must be transversely orientable because $M$ is simply connected).
    We will use $t$ as a coordinate on $(-1,1)$.
    Because $\F$ is taut, there is a path in $M\setminus L$ from a point with
    $t<0$ to a point with $t>0$.
    Then we can close up the path inside the tubular neighborhood to obtain
    a loop in $M$ that intersects $L$ exactly once.
    This shows that $L$ cannot be trivial in $H_{n-1}(M)$.
    But by $\pi_1(M) = \{\ast\}$ and Poincar\'e duality, $H_{n-1}(M)$ is trivial.
    This contradiction proves the result.
\end{proof}

\section{Symplectic foliations and Poisson geometry}

In this section, we will combine symplectic structures with foliations
and see how this brings us to the world of regular Poisson geometry.
We will need the leaves of our foliation to be even-dimensional, so
we will assume that $M$ is of dimension $m = 2n+q$.

\begin{definition}
    A \emph{symplectic foliation of codimension $q$} on $M$ is a
    codimension-$q$ foliation $\F$ on $M$ together with a leafwise form
    $\omega \in \Gamma(\Lambda^2 T^*\F)$ such that the restriction of $\omega$ to
    every leaf is symplectic.
\end{definition}

\subsection*{Local normal form}

Suppose $(M,\F,\omega)$ is a symplectically foliated manifold.
In a sufficiently small foliation chart for $\F$, we can apply Moser's trick
to bring the symplectic forms on all leaves into Darboux form simultaneously.
More precisely, we have the following lemma.

\begin{lemma}
    Let $(M,\F,\omega)$ be a manifold with a symplectic foliation of codimension $q$.
    Then near any point $p\in M$ we may choose coordinates $(x_1, \ldots, x_n,
    y_1, \ldots, y_n, z)$ ($z\in\R^q$) such that in a neighborhood of $p$:
    \begin{itemize}
    \item the foliation $\F$ is given by $z = \text{constant}$,
    \item the leafwise symplectic form is given by $\omega = \sum_i dx_i \wedge dy_i$.
    \end{itemize}
\end{lemma}

\subsection*{Regular Poisson structures}

Suppose that we have a manifold $M$ equipped with a regular corank-$q$ Poisson
structure $\pi$.  Then $M$ is automatically endowed with a symplectic foliation
by taking
\begin{itemize}
\item $\F = \im(\pi^{\sharp})$,
\item $\omega_x(v,w) = \alpha(w)$ where $x\in M$, $v, w \in T_x\F$, $\alpha\in T^*_xM$ and $\pi^{\sharp}(\alpha) = v$.
\end{itemize}
The leaves of this symplectic foliation are known as the \emph{symplectic leaves} of $\pi$.

Conversely, given a symplectic foliation $(\F,\omega)$ on $M$, we can endow $M$ with a Poisson structure
by taking the Poisson tensor at a point $x\in M$ to be
the composition
\[ T^*_xM \to T^*_x\F \xrightarrow{(\omega_x^{\flat})^{-1}} T_x\F \to T_xM .\]
These two constructions (from symplectic foliation to regular Poisson structure
and vice versa)
are inverse to each other.
Thus we have a bijection
\[ \left\{ \parbox{9em}{\centering symplectic foliations of codimension $q$} \right\} \leftrightarrow \left\{ \text{regular corank-$q$ Poisson structures} \right\} \]
and a symplectic foliation is essentially the same thing as a regular Poisson structure.
The local normal form for symplectic foliations is then just the Weinstein splitting theorem
in the case of a regular Poisson structure, which says that the Poisson bivector locally
looks like
\[ \pi = \sum_i \partial x_i \wedge \partial y_i .\]

A deep result by Mitsumatsu states that Lawson's foliation of $\Sphere^5$
admits a leafwise symplectic structure.

\begin{theorem}[Mitsumatsu \cite{mitsumatsu}]
    Lawson's foliation on the 5-sphere $\Sphere^5$ admits a smooth leafwise
    symplectic structure.
\end{theorem}

The Reeb foliation shows that $\Sphere^3$ also admits a symplectic foliation.
It is not known whether this is the case for odd-dimensional spheres of dimension at least 7.

\section{Lie groupoids}

Lie groupoids are multi-object generalizations of Lie groups.
We briefly review some of the important points and establish notation.
We refer the reader to the excellent \cite{moerdijk-mrcun} for a more thorough
introduction.
The infinitesimal counterpart to a Lie groupoid is a Lie algebroid, but not all
algebroids can be integrated to groupoids.
A good reference for the theory of integrability of Lie algebroids is
\cite{lectures-integrability-lie} or the original article
\cite{crainic-fernandes-lie}.

\subsection*{Definition and examples}

There are several ways to define what a groupoid is. The shortest is the
following.

\begin{definition}
    A \emph{groupoid} is a small category whose morphisms are all isomorphisms.
\end{definition}

A groupoid therefore consists of two sets, $M$ (the set of objects) and $\G$
(the set of arrows), together with structure maps
\begin{align*}
    s : \G \to M & \quad \text{source map} \\
    t : \G \to M & \quad \text{target map} \\
    m : \G^{(2)} \to \G & \quad \text{multiplication} \\
    i : \G \to \G & \quad \text{inversion} \\
    u : M \to \G & \quad \text{unit map}
\end{align*}
that satisfy the axioms of a category.
The set $\G^{(2)}$ is $\G \timesst \G$, the space of composable arrows.
(We will more generally use $\G^{(k)} = \G \timesst ... \timesst \G$ for the $k$-fold
fibered product.)
Our convention throughout will be that arrows go ``from right to left'', meaning
that we can only calculate $gh = g\cdot h = m(g,h)$ if $s(g)=t(h)$.
The \emph{orbits} of a groupoid are its isomorphism classes (so they partition $M$).
The \emph{isotropy group} at an object $x\in M$, written $\G_x$, is the automorphism
group of $x$.

\begin{definition}
A \emph{Lie groupoid} is a groupoid where $\G$ and $M$ are both manifolds, $s$ and $t$
are smooth submersions, and $m$, $i$ and $u$ are smooth maps.
\end{definition}

\begin{remark}
    We require that $M$ is Hausdorff, but we do not require that $\G$ is Hausdorff.
    Indeed, many natural examples of Lie groupoids have non-Hausdorff spaces of arrows.
    One example is the holonomy groupoid of a foliation that has a vanishing cycle.
    We do, however, require that the source and target fibers are Hausdorff.
    This same remark will hold true in the case of local Lie groupoids later on.
\end{remark}

If follows from the axioms that $u : M \to \G$ is an embedding, and we will
regard $M$ as an embedded submanifold of $\G$. This means that, in notation,
we will not distinguish a point $x\in M$ and its unit $u(x)\in \G$, and just
write $x$ for either.
The usual notation for a Lie groupoid with space of arrows $\G$ and space of objects $M$
is $\G \tto M$ (the two arrows represent the maps $s$ and $t$).

\begin{example}
    A Lie groupoid where $M$ is a singleton is precisely a Lie group.
\end{example}

\begin{example}
    If $\G$ is a Lie group acting on a manifold $M$, then the \emph{action Lie groupoid}
    is the groupoid $\G\ltimes M\tto M$. Its space of arrows is $\G\times M$,
    with $s(g,x) = x$ and $t(g,x) = g\cdot x$.
    The product is $(h,y)\cdot(g,x) = (hg,x)$, the inversion is $(g,x)^{-1} =
    (g^{-1},g\cdot x)$ and the unit map is $u(x) = (e_{\G},x)$.
\end{example}

\begin{example}[Fundamental groupoid of a foliation]
    Suppose that $\F$ is a foliation of a manifold $M$.
    Then there is an important Lie groupoid $\Pi_1(\F) \tto M$,
    called the \emph{fundamental groupoid} of $\F$.
    Its space of arrows is
    \[ \Pi_1(\F) = \{ \text{paths in leaves of $\F$} \} / {\sim} ,\]
    where ${\sim}$ is homotopy in the leaf, relative to the endpoints.
    The source map assigns to any (class of a) path its starting point.
    The target map assigns to any (class of a) path its endpoint.
    The multiplication is concatenation of paths, and the inverse of a path
    is the same path traversed in the opposite direction.

    The isotropy group at a point $x\in M$ is $\pi_1(L,x)$, where $L$ is the leaf containing
    $x$.
    If $\F$ has just one leaf (the entire manifold), then $\Pi_1(\F)$ is known
    as the fundamental groupoid of $M$.
\end{example}

\begin{example}[Holonomy groupoid of a foliation]
    Another important groupoid associated to a foliation is its \emph{holonomy groupoid}.
    Its space of arrows is
    \[ \Hol(\F) = \{ \text{paths in leaves of $\F$} \} / {\sim} ,\]
    where two paths are equivalent for ${\sim}$ if their holonomy is the same.
    The structure maps are analogous to those of the fundamental groupoid.
    There is a natural map $\Pi_1(\F) \to \Hol(\F)$ because homotopic paths
    have the same holonomy.
    This map is a covering of Lie groupoids (it is a covering map at the level
    of source fibers).
\end{example}

\begin{example}[Symplectic groupoids]
    Suppose that $\G$ is a Lie groupoid over $M$.
    A differential form $\omega \in \Omega^{k}(M)$ is called \emph{multiplicative}
    if
    \[ m^*\omega = p_1^*\omega + p_2^*\omega ,\]
    where $m : \G \timesst \G \to \G$ is the multiplication and by $p_i : \G\timesst \G\to\G$
    the projection to the $i$'th factor.
    A \emph{symplectic groupoid} is a pair $(\G,\omega)$, where $\omega\in\Omega^2(\G)$
    is a symplectic form that is multiplicative.
\end{example}

If $\G \tto M$ is a Lie groupoid, then the orbits of this groupoid
form a (possibly singular) foliation of $M$.

While much of the theory of Lie groupoids is similar to that of Lie groups,
there are some important differences. The one that is most relevant for us
concerns the infinitesimal counterparts to Lie groupoids, which we review next.

\subsection*{Lie algebroids}

To every Lie group we associate a Lie algebra by looking at the Lie group
structure at the infinitesimal level near the identity.
Similarly, for Lie groupoids, there is a structure that allows us to capture
a lot about the groupoid at the infinitesimal level.
These structures are \emph{Lie algebroids}.

\begin{definition}
A Lie algebroid over $M$ is a tuple $(A,\rho,[\slot,\slot])$ where
\begin{itemize}
    \item $A$ is a vector bundle over $M$,
    \item $\rho : A \to TM$ is a bundle map called the \emph{anchor map},
    \item $[\slot,\slot] : \Gamma(A) \times \Gamma(A) \to \Gamma(A)$ is a Lie bracket on the space
        of sections of $A$,
\end{itemize}
such that the anchor map induces a Lie algebra homomorphism $\Gamma(A) \to \Gamma(TM)$
and the Leibniz identity $[X,fY] = f[X,Y] + \rho(X)(f) Y$ holds for all
$X,Y\in\Gamma(A)$ and $f\in C^\infty(M)$.
\end{definition}

\begin{remark}
    This definition gives little intuition about Lie algebroid upon first reading.
    Here is a short attempt at a more intuitive description.
    The tangent bundle of a manifold summarizes all the possible directions that one
    can walk in on this manifold. The fiber $T_xM$ is the vector space of all possible
    velocities an ant could have when walking on the manifold and passing through $x$.
    The Lie bracket of vector fields summarizes the extent to which ``ways of
    walking'' (vector fields) fail to commute.

    A Lie algebroid is a generalization of the tangent bundle. It gives, for each point,
    the vector space of all possible ways that you can start walking at that point.
    The bracket $[\slot,\slot]$ of the Lie algebroid summarizes the commutators
    of these ``ways of walking''.
    The anchor $\rho$ summarizes which effect a certain ``way of walking'' has on your
    position, by indicating how your position on the manifold changes if you walk
    that particular way.

    If one considers the arrows of a groupoid as taking you from one point (the source)
    to another (the target), the Lie algebroid of this Lie groupoid summarizes
    how these arrows correspond to ``ways of walking'' at the infinitesimal level.
\end{remark}

The Lie algebroid of a Lie groupoid is constructed as follows.
Let $T^s\G$ be the kernel of $ds : T\G \to TM$.
Then as a vector bundle, $A = T^s_M\G = T^s\G{\restriction_M}$.
The anchor map is induced from $dt : T\G \to TM$.
Right-invariant vector fields on $\G$ correspond to sections of $A$,
and the bracket $[\slot,\slot]$ on $\Gamma(A)$ is the one induced
from the usual Lie bracket of these right-invariant vector fields.
(We will repeat this construction in greater detail for local Lie groupoids later.)

The theory of Lie groupoids and Lie algebroids is largely analogous to that of
Lie groups and Lie algebras. Here are some key facts:
\begin{itemize}
    \item A morphism of Lie groupoids $\G_1 \to \G_2$ induces a unique morphism of Lie
        algebroids $A_1 \to A_2$.
        A morphism of Lie algebroids $A_1 \to A_2$ induces a unique morphism of Lie
        groupoids $\G_1 \to \G_2$ if $\G_1$ is source-simply connected (i.e.\ its source
        fibers are simply connected).
    \item Every Lie groupoid $\G$ is covered by a uniquely determined source-simply connected
        Lie groupoid.
        (More precisely, there is a source-simply connected Lie groupoid $\tilde{\G}$
        and a morphism $f: \tilde{\G} \to \G$ that is a local diffeomorphism,
        such that $f$ induces an isomorphism at the level of Lie algebroids.)
        This groupoid $\tilde{\G}$ is unique up to isomorphism.
\end{itemize}

It is a standard result in Lie theory that every Lie algebra is the Lie algebra
of some Lie group.
This no longer holds for Lie algebroids.
We say that a Lie algebroid is \emph{integrable} if it is the Lie algebroid of
some Lie groupoid.
The obstructions to the integrability of a Lie algebroid were first understood
by Crainic and Fernandes \cite{crainic-fernandes-lie}.
They showed that the integrability of a Lie algebroid is controlled by its
\emph{monodromy groups}.
The monodromy group at a point $x \in M$ is obtained as an image of $\pi_2(L)$ under
a certain map, where $L$ is the orbit of $x$.
Important in the construction of the monodromy groups is the idea of an $A$-path.

\subsection*{$A$-paths}

Suppose that $\G$ is a Lie group.  If we have a path $\gamma : I \to \G$
starting at the identity, we may differentiate it and right-translate to the
origin to get a path $a : I\to \g$, where $\g$ is the Lie algebra of $\G$.
This path effectively gives us the velocity of $\gamma$ at any time.
Conversely, if we have a path $a:I\to \g$, we can reconstruct the path in $\G$
that it was obtained from by integrating (assuming it started at the identity).
This way, there is a bijective correspondence between $\G$-paths and $\g$-paths.

If we have a family of paths $\gamma_t : I\to\G$ (all starting at the identity), we
can use this differentiation procedure to obtain a family of paths $a_t : I\to
\g$. If $(\gamma_t)_t$ is a homotopy of paths rel endpoints, this will have
some consequences for the family of $\g$-paths $(a_t)_t$. It turns out that one
can detect, solely from the family $(a_t)_t$ and the Lie algebra structure,
whether the original homotopy of $\G$-paths kept the endpoint of the paths fixed or not.
In other words, there is a condition on a family of $\g$-paths
\emph{that can be expressed using only the structure of the Lie algebra},
that tells us whether the endpoint of the $\G$-paths was kept fixed.
One says that two $\g$-paths are \emph{$\g$-homotopic} if they can be related to
each other by a family of $\g$-paths that satisfies this condition.
Note that one can speak about $\g$-homotopy without ever referring to a Lie group
integrating the Lie algebra.

Now if $\G$ is a simply connected Lie group, then $\G$ itself is in bijection
with the homotopy classes of $\G$-paths starting at the identity, rel endpoints.
This observation allows one to reconstruct a simply connected Lie group from its
Lie algebra.
This technique can be used to prove that every Lie algebra is the Lie algebra of some
Lie group (i.e.\ that every Lie algebra is integrable). One reference where this is
done is \cite{moerdijk-mrcun}.

While many proofs of the integrability of Lie algebras rely on some sort of classification
result about Lie algebras, the one outlined above does not.
By generalizing the procedure above to Lie algebroids, Crainic and Fernandes were
able to understand the obstructions to integrability of Lie algebroids \cite{crainic-fernandes-lie}.
Suppose that $A$ is a Lie algebroid over a manifold $M$.

\begin{definition}
    An \emph{$A$-path} is a pair $(a,\gamma)$, where $\gamma : I\to M$ is a path in $M$
    and $a : I \to A$ is a path in $A$, such that
    \begin{itemize}
        \item $a$ lies over $\gamma$, meaning that $a(t) \in A_{\gamma(t)}$ for all $t \in I$,
        \item $\rho(a(t)) = \frac{d\gamma}{dt}(t)$ for all $t\in I$.
    \end{itemize}
\end{definition}

Crainic and Fernandes construct an equivalence relation ${\sim}$ on the set of all $A$-paths,
called \emph{$A$-homotopy}, analogous to the $A$-homotopy for Lie algebras outlined above.

\begin{definition}
    If $A$ is a Lie algebroid, the quotient
    \[ \G(A) = \{ \text{$A$-paths} \} / {\sim} \]
    is called the \emph{Weinstein groupoid} of $A$.
\end{definition}

It is a source-simply connected topological groupoid, but it is not necessarily smooth.
Next, a morphism of groups $\partial : \pi_2(\Orbit_x) \to \G(\g_x)$ is constructed.
Here, $\Orbit_x$ is the orbit of a point $x\in M$ and $\G(\g_x)$ is the simply connected
integration of the isotropy Lie algebra $\g_x$ at $x$.
This morphism is called the \emph{monodromy homomorphism}.

\begin{definition}
    The image of the monodromy homomorphism $\partial : \pi_2(\Orbit_x) \to \G(\g_x)$
    is called the \emph{monodromy group} of $A$ at $x$, written $\Mon_x(A)$.
\end{definition}

The result characterizing the integrable Lie algebroids is the following.

\begin{theorem}[Crainic-Fernandes, \cite{crainic-fernandes-lie}]
    For a Lie algebroid $A$, the following statements are equivalent:
    \begin{itemize}
        \item $A$ is integrable,
        \item the Weinstein groupoid $\G(A)$ is smooth,
        \item the monodromy groups $\Mon_x(A)$ are uniformly discrete.
    \end{itemize}
    Moreover, in this case, $\G(a)$ is the unique $s$-simply connected Lie groupoid
    integrating $A$.
\end{theorem}

Here, ``uniformly discrete'' means that there is an open set $U\subset
A$ containing the zero section, such that the intersection
$U \cap \left( \bigcup_{x\in M} \Mon_x(A) \right)$
is the zero section.
In other words, the non-trivial monodromy is bounded away from the zero section.

\begin{example}[Symplectic groupoids integrate Poisson manifolds]
    If $(M,\pi)$ is a Poisson manifold, we can consider its \emph{cotangent algebroid}.
    As a vector bundle, this is $T^*M$. To turn this into a Lie algebroid, we prescribe
    the anchor map
    \[ \pi^{\sharp} : T^*M \to TM \]
    by $\beta(\pi^{\sharp}(\alpha)) = \pi(\alpha,\beta)$, and the bracket by
    \[ [\alpha, \beta] = \Li_{\pi^{\sharp}(\alpha)}(\beta) -
        \Li_{\pi^{\sharp}(\beta)}(\alpha) - d(\pi(\alpha,\beta)) .\]
    This makes $A$ into a Lie algebroid over $M$.
    If $A$ happens to be integrable (we say that $M$ is an \emph{integrable
        Poisson manifold}),
    then its integration $\G$ will be a symplectic Lie groupoid.
    In other words, because $A$ comes from a Poisson bracket, the integration $\G$
    will naturally come with a multiplicative symplectic structure.
    Conversely, if $\G\tto M$ is a symplectic Lie groupoid, then $M$
    has the structure of a Poisson manifold, whose cotangent algebroid
    is the algebroid of $\G$.
\end{example}

%% file: chapters/calibrations.tex
\chapter{Calibrated symplectic foliations of codimension 1}
\label{chapter:calibrations}

In this chapter we consider a special class of symplectic foliations, namely
the ones whose leafwise symplectic form can be extended to a closed 2-form
on the manifold.
\emph{Throughout this chapter, all foliations will be of codimension 1.}

\section{Extensions of the leafwise symplectic form}

Given a symplectic foliation $(\F,\omega)$ on $M$, we can ask about forms $\tilde{\omega}\in \Lambda^2T^*M$ extending $\omega$.
In other words, we want to understand the 2-forms $\tilde{\omega}$ on $M$ such that
\[ \tilde{\omega}(v,w) = \omega(v,w) \quad\text{for all $v, w\in T\F$} .\]
It is easy to see that such an extension always exists.
In fact, these extensions are in 1-to-1 correspondence with splittings
\[ TM = T\F \oplus \nu ,\]
by taking $\nu = \ker(\tilde{\omega})$. (Indeed, any extension has rank at least $2n$ everywhere,
and since the rank of a 2-form is always even, the rank must be $2n$.
The extension is then determined by specifying its kernel.)

\section{Existence of a closed extension}

We are interested in symplectic foliations for which the leafwise symplectic
form has a \emph{closed} extension (closed as elements of $\Omega^2(M)$, not merely
leafwise closed, which is assumed of a symplectic foliation).
In \cite{martineztorres}, these are called \emph{2-calibrated}. We will just
refer to them as \emph{calibrated}.

\begin{definition}
    A \emph{calibration} for a symplectic foliation $(M,\F,\omega)$ is an extension
    $\tilde{\omega} \in \Omega^2(M)$ of $\omega$ such that $d\tilde{\omega} = 0$.

    We refer to the tuple $(\F,\tilde{\omega})$ as a \emph{calibrated foliation} on $M$
    (it determines $\omega$, of course).
    A symplectic foliation that admits a calibration will be called \emph{calibratable}.
\end{definition}

\begin{remark}
    One could work with the same definition in higher codimensions.
    We will restrict ourselves to codimension 1, however.
    The geometric interpretation of a calibration that we will see in the next section
    no longer holds in higher codimensions.
\end{remark}

The following lemma is immediate.
\begin{lemma}
    If $(M,\F,\tilde{\omega})$ is a calibrated foliation, then $\F$ is taut.
\end{lemma}
\begin{proof}
    The $2n$-form $\tilde{\omega}^n$ is a volume form when restricted to any leaf.
    By theorem~\ref{thm:criteria-tautness}, this shows that $\F$ is taut.
\end{proof}

Clearly, the converse does not hold: the submersion $\Sphere^4\times \Sphere^1 \to \Sphere^1$
determines a taut foliation whose leaves do not admit symplectic structures.

Even among taut foliations that admit a leafwise symplectic structure,
there are some that do not admit calibrations. The following is an example.

\begin{example}
    We will consider the torus $\Torus^4$ as $\R^4/\Z^4$ with coordinates $(\theta_1,\theta_2,\theta_3,\theta_4)$. \phantomsection\label{page:example-torus}
    Let $\phi : \Torus^4\to \Torus^4$ be the diffeomorphism induced by the linear map $\R^4\to\R^4$ with matrix
    \[ \begin{pmatrix} 0 & 0 & -1 & 0 \\
     0 & -1 & 0 & 0 \\
     0 & 0 & 0 & 1 \\
     1 & 0 & 1 & 0 \\
    \end{pmatrix}. \]
    Let $(M,\F)$ be the suspension of $\phi$ with the obvious codimension-1 foliation
    (so $M = M\times [0,1] / ((p,0)\sim(\phi(p),1))$ with the foliation coming from the
    trivial one on $M\times[0,1]$).

    We claim that $(M,\F)$ has no calibrated symplectic foliation.
    As a basis for $H^2(\Torus^4;\R)$ we take
    \[ \{d\theta_1\wedge d\theta_2, d\theta_1\wedge d\theta_3, d\theta_1\wedge d\theta_4, d\theta_2\wedge d\theta_3, d\theta_2\wedge d\theta_4, d\theta_3\wedge d\theta_4\} .\]
    (For ease of notation, we left out brackets indicating equivalence classes.)
    Relative to this basis, $\phi^* : H^2(\Torus^4;\R) \to H^2(\Torus^4;\R)$ is given by
    \[ \begin{pmatrix}
     0 & 0 & 0 & 0 & 1 & 0 \\
     0 & 0 & 1 & 0 & 0 & 0 \\
     0 & 0 & 0 & 0 & 0 & -1 \\
     -1 & 0 & 0 & 0 & -1 & 0 \\
     0 & 0 & 0 & -1 & 0 & 0 \\
     0 & -1 & 0 & 0 & 0 & -1 \\
    \end{pmatrix}, \]
    which has no non-zero fixed points.
    This shows that $(M,\F)$ has no calibrated symplectic foliation,
    because in that case the class of the symplectic form on a leaf has to be
    invariant under $\phi^*$ (by Stokes' theorem; see also the geometric
    interpretation in the next section).

    There is, however, a leafwise symplectic structure on $(M,\F)$.
    Let
    \[ \omega_0 = d\theta_1\wedge d\theta_2 + d\theta_3\wedge d\theta_4 .\]
    Let
    \[ \omega_1 = -d\theta_1\wedge d\theta_2 - d\theta_1\wedge d\theta_3 + d\theta_2 \wedge d\theta_4 .\]
    Note that $\phi^*\omega_1 = \omega_0$.
    Let
    \[ \omega_t = t \omega_1 + (1-t)\omega_0 .\]
    An explicit calculation shows that
    \[ \omega_t\wedge\omega_t = (3t^2-3t+1) d\theta_1\wedge d\theta_2\wedge d\theta_3\wedge d\theta_4, \]
    which does not vanish for $t\in[0,1]$.

    Let $f : [0,1] \to [0,1]$ be smooth such that $f$ equals 0 near 0, and $f$ equals 1 near 1.
    The family of symplectic forms $(\omega_{f(t)})_{t\in[0,1]}$ on $\Torus^4\times [0,1]$ descends to a smooth leafwise symplectic structure on
    \[ M = \frac{\Torus^4\times[0,1]}{(p,0)\sim(\phi(p),1)} .\]
    (The reparameterization using $f$ is necessary for smoothness of the gluing, but is not essential to the spirit of the example.)
\end{example}

\section{Geometric interpretation of closedness}

One may wonder why it is interesting to look at \emph{closed} extensions of a leafwise
symplectic form.
It turns out that the closedness has a simple geometric interpretation.

Suppose $(M,\F,\omega)$ is a symplectically foliated manifold and $\tilde{\omega}$ a (not necessarily closed)
extension of $\omega$.
Let
\[ TM = T\F \oplus \nu \]
be the splitting determined by $\tilde{\omega}$, so that $\nu = \ker(\tilde{\omega})$.
This $\nu$ is a line field on $M$.

Closedness of $\tilde{\omega}$ is a local issue, and we will work in a chart
for our interpretation.
In the chart, we have local coordinates $x_1, \ldots, x_n, y_1, \ldots, y_n, z$
with the foliation given by $\ker(dz)$. Every leaf has a symplectic structure.
Let $X$ be the vector field tangent to $\nu$ such that $dz(X) = 1$.
Then the flow of $X$ maps leaves to leaves.

\begin{proposition}
    The extension $\tilde{\omega}$ is closed if and only if
    the flow of $X$ preserves the leafwise symplectic structure.
\end{proposition}

In other words, $\tilde{\omega}$ is closed if and only if the maps from leaves to leaves induced
by the flow of $X$ are symplectomorphisms whenever defined.
This proposition is essentially a special case of lemma~6.18 in \cite{mcduff-salamon},
which says that a connection on a symplectic fibration is symplectic if and only if
the connection 2-form is vertically closed.

\begin{proof}
    If $X$ preserves $\tilde{\omega}$, it clearly preserves $\omega$ (since $\omega$ is
    determined by $\tilde{\omega}$ and $\F$, and $\F$ is preserved by $X$).
    Conversely, if $X$ preserves $\omega$, then it preserves $\tilde{\omega}$,
    since $\tilde{\omega}$ is determined by $\omega$ and $X$ itself (and the flow of $X$
    preserves $X$). We therefore want to show that $\Li_X\tilde{\omega} = 0$ if and only if
    $d\tilde{\omega} = 0$.

    Now
    \[ \Li_X\tilde{\omega} = d i_X \tilde{\omega} + i_X d\tilde{\omega} ,\]
    and $i_X\tilde{\omega} = 0$ because $X$ lies along $\nu = \ker(\tilde{\omega})$.
    This shows that $X$ preserves the leafwise symplectic form if and only if $i_Xd\tilde{\omega} = 0$.
    But $i_Xd\tilde{\omega} = 0$ if and only if $d\tilde{\omega} = 0$ because $\tilde{\omega}$ is already
    leafwise closed.
    This proves the result.
\end{proof}

\begin{remark}
    Suppose $(M,\F,\omega)$ is a symplectically foliated manifold.
    We can regard it as a regular Poisson manifold.
    If the symplectic foliation is calibratable, this Poisson manifold is automatically
    integrable. In fact, its monodromy groups are trivial, as we now briefly explain.
    We give a shortened version of the explanation in
    \cite[section 5.2]{lectures-integrability-lie}.
    For more detail, we refer the reader to that section.

    For regular Poisson manifolds, the monodromy has a geometric interpretation.
    Suppose that $x\in M$ lies on a leaf $L$.
    Consider a sphere $\gamma : \Sphere^2 \to L$ inside the symplectic leaf,
    such that $\gamma$ maps the north pole $N \in \Sphere^2$ to $x$.
    We can calculate the symplectic area $A(\gamma)$ of this sphere by
    integrating the leafwise symplectic structure over it.

    Now take a short tranversal to $\F$ at $x$, parameterized by $t\in(-\eps,\eps)$
    (where $t=0$ corresponds to $x$).
    A \emph{deformation of $\gamma_0$} is a family $\gamma_t : \Sphere^2 \to M$
    of spheres, parameterized by $t\in(-\eps,\eps)$, such that $\gamma_0 = \gamma$,
    and such that for each $t$ the sphere $\gamma_t$ lies entirely in a leaf of $\F$.
    The \emph{tranversal variation} of $\gamma_t$ is the element
    \[ \text{var}(\gamma_t) = \left.\frac{d}{dt}\right|_{t=0} \gamma_t(N) \in \nu_x(\F) .\]
    It turns out that $\left.\frac{d}{dt}\right|_{t=0} A(\gamma_t)$ only depends on the
    homotopy class of $\gamma$ and on $\text{var}(\gamma_t)$, and that it is linear
    in $\text{var}(\gamma_t)$ (we are using $A$ to denote the symplectic area).
    This way, we can associate to every $[\gamma]\in\pi_2(L,x)$ an element $A'([\gamma])$
    given by
    \[ \left.\frac{d}{dt}\right|_{t=0} A(\gamma_t) = \langle A'([\gamma]), \text{var}(\gamma_t)\rangle . \]
    This quantity $A'([\gamma])$ measures how much the symplectic area of a sphere $[\gamma]$
    changes as we move it transversely.
    As shown in \cite[proposition 5.4]{lectures-integrability-lie}, we have
    \[ \Mon_x = \{ A'([\gamma]) \mid [\gamma] \in \pi_2(L,x) \} .\]
    By the geometric interpretation of closedness, however, we see that for a calibrated
    symplectic foliation the symplectic area does not change when moving transversely.
    This implies that the monodromy of $(M,\F,\omega)$ vanishes.
    In particular, every symplectically foliated manifold that is calibratable
    is an integrable Poisson manifold (and its isotropy groups are simply $(\R,+)$).
\end{remark}

\section{Local normal form}

Using the geometric interpretation of closedness we obtain the following local normal form
for calibrated foliations.

\begin{lemma}
    Let $(M,\F,\tilde{\omega})$ be a manifold with a calibrated foliation.
    Then near any point $p\in M$ we may choose coordinates $x_1, \ldots, x_n, y_1, \ldots, y_n, z$
    such that in a neighborhood of $p$:
    \begin{itemize}
    \item the foliation $\F$ is given by $dz = 0$,
    \item the calibration $\tilde{\omega}$ is given by $\tilde{\omega} = \sum_i dx_i \wedge dy_i$.
    \end{itemize}
\end{lemma}

Note the difference with the local normal form for symplectic foliations: the calibration $\tilde{\omega}$
is also brought to a normal form in this lemma, not just its restriction to leaves.

\begin{proof}
    Let $U$ be a small open neighborhood of $p$ in its leaf, with Darboux coordinates $x_1, \ldots, x_n, y_1, \ldots, y_n$.
    Let $X$ be a section of $\nu$, non-vanishing and parallel for the Bott connection (so that
    its flow maps leaves to leaves).
    Write $\phi_t$ for the time-$t$ flow of $X$.
    After shrinking $U$ if necessary and picking $\eps > 0$ small enough, the map
    \[ U \times (-\eps,\eps) \to M : q \mapsto \phi_t(q) \]
    is a diffeomorphism onto an open neighborhood of $p$ in $M$.
    We use (the inverse of) this map as a chart near $p$, with coordinates $x_i, y_i$ on $U$
    and $z$ for the coordinate along $(-\eps,\eps)$.
    Then the foliation is given by $\ker(dz)$ in this chart.
    Moreover, translation in the $z$-direction is a symplectomorphism between leaves.
    Because the $x_i, y_i$ are Darboux coordinates for $U$, this means the leafwise symplectic
    structure is given by $\sum_i dx_i \wedge dy_i$.
    Also, in this chart, the kernel of $\tilde{\omega}$ is, by construction, given by $\partial z$.
    This shows that in fact
    \[ \tilde{\omega} = \sum_i dx_i \wedge dy_i ,\]
    proving the result.
\end{proof}

\section{Calibratability and transverse differentiation}

\begin{lemma}
    A form $\alpha \in \Omega^k(\F)$ has a closed extension if and only if its class $[\alpha] \in H^k(\F)$ is in the image of the restriction map $H^k(M) \to H^k(\F)$.
\end{lemma}
\begin{proof}
    If $\alpha$ has a closed extension, then clearly it lies in the image of this map.
    Conversely, suppose $[\alpha]$ is the image of $[\beta]\in H^k(M)$.
    This means that
    \[ \beta{\restriction_\F} - \alpha = d_{\F}\gamma \]
    for some $\gamma \in \Omega^{k-1}(\F)$.
    Pick an extension $\tilde{\gamma}\in\Omega^{k-1}(M)$ of $\gamma$.
    Set
    \[ \tilde{\alpha} = \beta + d\tilde{\gamma} .\]
    Then $\tilde{\alpha}$ is closed, and its restriction to $\F$ equals $\alpha$ by construction.
    This proves the result.
\end{proof}

By definition, $\Omega^{\bullet}(M,\F)$ will be the kernel of the restriction map
$\Omega^{\bullet}(M) \to \Omega^{\bullet}(\F)$.
The following lemma follows from computation.

\begin{lemma}
    The maps
    $\phi^k : \Omega^{k}(M,\F) \to \Omega^{k-1}(\F,\nu^*)$
    given by
    \[ \phi^k(\alpha)(v_1, \ldots, v_{k-1})(v_k) = \alpha(v_1, \ldots, v_k) \]
    form an isomorphism of chain complexes $\Omega^{\bullet}(M,\F) \to \Omega^{\bullet-1}(\F,\nu^*)$.
\end{lemma}

We now have a short exact sequence of complexes
\[ 0 \to \Omega^{\bullet-1}(\F,\nu^*) \to \Omega^{\bullet}(M) \to \Omega^{\bullet}(\F) \to 0 .\]
It has an associated long exact sequence in cohomology, which near $H^k(M)$ looks like
\[ \cdots \to H^{k-1}(\F,\nu^*) \to H^k(M) \to H^k(\F) \xrightarrow{d_{\pitchfork}} H^k(\F,\nu^*) \to \cdots \]
The connecting morphism from $H^k(\F)$ to $H^k(\F,\nu^*)$ is a type of transverse differentiation of the leafwise form, which is why we denote it by $d_{\pitchfork}$.
We have the following corollary.

\begin{corollary}
    A leafwise form $\omega$ on $\F$ has a closed extension if and only if
    $d_{\pitchfork}[\omega] = 0$.
\end{corollary}

For every leaf $L$ of the foliation, there is a map $H^k(\F,\nu^*) \to H^k(L,{\nu^*}{\restriction_L})$.
Therefore, if a leafwise symplectic form is to have a closed extension, we must have $d_{\pitchfork}[\omega]{\restriction_L} = 0$.
This is not a sufficient condition, however, as the following example shows.

\begin{example}
    Consider the Reeb foliation of $\Sphere^3$.
    We denote the compact leaf by $T$.
    Pick a small tubular neighborhood $T\times (-1,1)$ around $T$ and a symplectic form $\omega_T$ on $T$.
    Equip the tubular neighborhood with the 2-form given by pulling back along the projection $T\times (-1,1) \to T$.
    Restrict the resulting form to the leaves to get a leafwise form near $T$, and extend this
    leafwise form to get a leafwise symplectic form $\omega$ on all of $\F$.

    We claim that the leafwise symplectic form has no closed extension,
    even though $d_{\pitchfork}[\omega]{\restriction_L} = 0$ for every leaf $L$.
    There is no closed extension because the Reeb foliation is not taut.
    The only leaf with $H^2(L,\nu^*) \neq 0$ is the compact leaf $T$ (note that $\nu^*$ is a trivial bundle since there is no linear holonomy).
    But we can find an extension of $\omega$ that is closed near $T$ (we pick the extension to equal the pullback we originally constructed on $T\times (-1,1)$), which shows that $d_{\pitchfork}[\omega]$ vanishes when restricted to $L$.
    (By the geometric interpretation of monodromy, we have shown that this Poisson
    manifold has trivial monodromy.)
\end{example}

%

\section{The cosymplectic case}

In this section we discuss manifolds with cosymplectic structures.
A cosymplectic structure is a special case of a calibrated structure.

\begin{definition}
    A \emph{cosymplectic structure} on $M$ is a closed non-vanishing 1-form
    $\alpha \in \Omega^1(M)$ together with a closed 2-form $\omega \in \Omega^2(M)$
    such that $(M,\ker(\alpha),\omega)$ is a calibrated foliation on $M$.
\end{definition}

In other words, a cosymplectic structure is a calibrated foliation for which the foliation
is determined by a \emph{closed} 1-form.
In particular, the foliation has no holonomy.
Note that $\omega$ being non-degenerate on $\ker(\alpha)$ is equivalent to $\alpha\wedge\omega^n$ being non-vanishing.

\begin{proposition}
    Let $(M,\F,\omega)$ be a calibrated symplectic foliation.
    Write $\pi$ for the induced Poisson bivector on $M$.
    Then the following are equivalent:
    \begin{enumerate}
        \item There is a closed 1-form $\alpha$ defining $\F$ (i.e.\ the calibration comes from a cosymplectic structure).
        \item There is a volume form on $M$ invariant under all Hamiltonian vector fields (one says that the Poisson structure is unimodular).
        \item There is a Poisson vector field $E$ transverse to $\F$.
    \end{enumerate}
\end{proposition}
\begin{proof}
    $(1) \Rightarrow (2)$ Take $\mu = \alpha\wedge\omega^n$, a volume form on $M$.
    Then $\mu$ is invariant under Hamiltonian vector fields:
    \begin{align*}
        \Li_{X_f}(\mu)
        &= \Li_{X_f}(\alpha\wedge\omega^n) \\
        &= (\Li_{X_f}\alpha) \wedge \omega^n + \alpha \wedge (\Li_{X_f}\omega^n)
    \end{align*}
    The last of these terms vanishes because $\Li_{X_f}\omega^n = 0$ along $\F$.
    We get
    \begin{align*}
        \Li_{X_f}(\mu)
        &= (\Li_{X_f}\alpha) \wedge \omega^n \\
        &= (d i_{X_f} \alpha + i_{X_f} d\alpha) \wedge \omega^n \\
        &= (0 + 0) \wedge \omega^n \\
        &= 0.
    \end{align*}

    $(2) \Rightarrow (1)$ Let $\mu$ be a volume form invariant under Hamiltonian flows.
    Then let $\alpha$ be the unique 1-form on $M$ defining $\F$ that also satisfies
    $\alpha\wedge\omega^n = \mu$.
    By the calculation above, we have that $i_{X_f} d\alpha = 0$ for any Hamiltonian
    function $f$.
    This means that $d\alpha$ vanishes whenever an $\F$-tangential vector is plugged in,
    so it only takes nonzero values in the normal direction.
    But there is only one normal direction, and $d\alpha$ is a 2-form. This shows that
    $d\alpha = 0$.

    $(1) \Rightarrow (3)$ Let $E$ be defined by $i_E\omega = 0$ and $\alpha(E) = 1$.
    Then clearly $\Li_E(\omega) = 0$, so $\omega$ is preserved by $E$.
    Also, $\Li_E(\alpha) = d(i_E\alpha) = 0$, so $\alpha$ is preserved by $E$.
    Therefore, $\F$ is also preserved by $E$.
    But $\pi$ is determined by $\omega$ and $\F$, and is therefore preserved by $E$.
    This shows that $E$ is Poisson.

    $(3) \Rightarrow (1)$ If $\Li_E\pi = 0$ then $\F$ is preserved by $E$.
    Let $\alpha$ be the 1-form determining $\F$ such that $\alpha(E) = 1$.
    Now $\Li_E(E) = [E,E] = 0$, so the flow of $E$ preserves $E$ and $\F$
    and therefore also $\alpha$.
    This means that
    $0 = \Li_E\alpha = i_E(d\alpha)$.
    But then $i_Ed\alpha = 0$ and $\alpha\wedge d\alpha = 0$, so we must have $d\alpha = 0$.
    This proves the result.
\end{proof}

Cosymplectic structures have been completely characterized in \cite{guillemin-miranda-pires}.
They are quite restricted. For example, if a cosymplectic structure has a compact leaf,
then all leaves are compact, and the manifold is a mapping torus of
a compact leaf.

\section{Transverse geometry in the calibrated case}

In \cite{martineztorres}, the authors prove the following result using approximately
holomorphic techniques following Donaldson.
We refer the reader to \cite[section 5.4]{moerdijk-mrcun} for the notions of essential equivalence
and Morita equivalence between groupoids.

\begin{theorem}
    Let $M$ be a closed smooth manifold and $(\F,\omega)$ a calibratable
    symplectic codimension-1 foliation. Then there exists a 3-dimensional
    closed submanifold $W$ of $M$, transverse to $\F$, which inherits a taut
    foliation $\F_W$ from $\F$, and such that the map between holonomy
    groupoids $\Hol(\F_W) \to \Hol(\F)$ is an essential equivalence.
\end{theorem}

In other words, symplectic codimension-1 foliations have transverse geometry that is captured completely by the induced foliation on some 3-dimensional submanifold.

They also prove the following version for homotopy groupoids.

\begin{theorem}
    Let $M$ be a closed smooth manifold and $(\F,\omega)$ a calibratable
    symplectic codimension-1 foliation. Then there
    exists a 5-dimensional closed submanifold $W$ of $M$, transverse to $\F$,
    which inherits a taut foliation $\F_W$ from $\F$, and such that the map
    between homotopy groupoids $\Pi_1(\F_W) \to \Pi_1(\F)$ is an essential
    equivalence.
\end{theorem}

These statements do not relate the symplectic geometry of the leaves of $W$ to that of $M$.
To gain insight into the symplectic geometry of this situation, one can wonder how the symplectic groupoids integrating $M$ and $W$ (which are both Poisson manifolds) are related.

The following proposition shows that a Morita equivalence on $\Pi_1$ is equivalent to a Morita equivalence on symplectic integrations.

\begin{proposition}
    Let $M$ be a closed smooth manifold and $(\F,\omega)$ a symplectic codimension-1 foliation with a calibration.
    Let $W$ be a closed submanifold of $M$, transverse to $\F$.
    It inherits a taut foliation $\F_W$ from $\F$.
    The map between source-simply connected symplectic integrations
    \[ \Sigma(W) \to \Sigma(M) \]
    is a symplectic Morita equivalence if and only if the map between fundamental groupoids
    \[ \Pi_1(\F_W) \to \Pi_1(\F) \]
    is a Morita equivalence.
    \label{prop:morita-equivalence}
\end{proposition}
\begin{proof}
    A map between Lie groupoids is a Morita equivalence if and only if it
    induces a homeomorphism
    on leaf spaces, isomorphisms on isotropy groups and isomorphisms of normal
    representations \cite[theorem 4.3.1]{delhoyo}.
    The maps $\Sigma(W)\to\Sigma(M)$ and $\Pi_1(\F_W)\to\Pi_1(\F)$ induce the
    same maps on leaf spaces.
    The normal space at any point $x\in M$, for any of the groupoids, is the normal
    space to the foliation, and the action on it is given by holonomy.
    It therefore suffices to check that one map induces isomorphisms on isotropy groups
    if and only if the other does.

    Suppose $x\in W$. Write $L$ for the leaf of $M$ that contains $x$, and $L_W$ for the leaf of $W$ that contains $x$. By proposition 3.21 in \cite{lectures-integrability-lie},
    there exists a diagram as follows, with exact rows:

    \begin{center}
    \begin{tikzpicture}
        \matrix(m)[matrix of math nodes, row sep=2.2em, column sep=2.6em]{
            \pi_2(L_W,x) & \nu_x^* & \Sigma(W)_x & \pi_1(L_W,x) & 1 \\
            \pi_2(L,x) & \nu_x^* & \Sigma(M)_x & \pi_1(L,x) & 1 \\
        };
        \path[->]
        (m-1-1) edge node[auto] {$\partial$} (m-1-2)
        (m-1-2) edge node[auto] {} (m-1-3)
        (m-1-3) edge node[auto] {} (m-1-4)
        (m-1-4) edge node[auto] {} (m-1-5)
        (m-2-1) edge node[auto] {$\partial$} (m-2-2)
        (m-2-2) edge node[auto] {} (m-2-3)
        (m-2-3) edge node[auto] {} (m-2-4)
        (m-2-4) edge node[auto] {} (m-2-5)
        (m-1-1) edge node[auto] {} (m-2-1)
        (m-1-2) edge node[auto] {} (m-2-2)
        (m-1-3) edge node[auto] {} (m-2-3)
        (m-1-4) edge node[auto] {} (m-2-4)
        (m-1-5) edge node[auto] {} (m-2-5)
        ;
    \end{tikzpicture}
    \end{center}

    The maps $\partial$ are the monodromy homomorphisms of the cotangent
    algebroids of $W$ and $M$, and because $\omega$ has a closed extension, we
    have $\partial = 0$.
    We therefore have
    the following commutative diagram:

    \begin{center}
    \begin{tikzpicture}
        \matrix(m)[matrix of math nodes, row sep=2.2em, column sep=2.6em]{
            1 & \nu_x^* & \Sigma(W)_x & \pi_1(L_W,x) & 1 \\
            1 & \nu_x^* & \Sigma(M)_x & \pi_1(L,x) & 1 \\
        };
        \path[->]
        (m-1-1) edge node[auto] {} (m-1-2)
        (m-1-2) edge node[auto] {} (m-1-3)
        (m-1-3) edge node[auto] {} (m-1-4)
        (m-1-4) edge node[auto] {} (m-1-5)
        (m-2-1) edge node[auto] {} (m-2-2)
        (m-2-2) edge node[auto] {} (m-2-3)
        (m-2-3) edge node[auto] {} (m-2-4)
        (m-2-4) edge node[auto] {} (m-2-5)
        (m-1-1) edge node[auto] {} (m-2-1)
        (m-1-2) edge node[auto] {} (m-2-2)
        (m-1-3) edge node[auto] {} (m-2-3)
        (m-1-4) edge node[auto] {} (m-2-4)
        (m-1-5) edge node[auto] {} (m-2-5)
        ;
    \end{tikzpicture}
    \end{center}

    The map $\nu_x^* \to \nu_x^*$ is the identity, and by the five lemma we see that
    the map $\Sigma(W)_x \to \Sigma(M)_x$ is an isomorphism if and only if the map
    $\pi_1(L_W, x) \to \pi_1(L,x)$ is.
\end{proof}

If we have a symplectic foliation of codimension 1 with a calibration, we can
essentially capture its symplectic integration using a 5-dimensional transverse submanifold.
A 3-dimensional submanifold suffices to capture the transverse geometry of the foliation.
As the following example shows, it is not generally possible to capture the Morita equivalence class of the symplectic integration using a 3-dimensional submanifold.

\begin{example}
    Consider $M = \Torus^4 \times \Sphere^1$ with the foliation with leaves
    $\Torus^4\times\{\ast\}$.  Equip $M$ with a leafwise symplectic form, such
    as $d\theta_1\wedge d\theta_2 + d\theta_3\wedge d\theta_4$ (where the
    $\theta_i$ are angular coordinates on $\Torus^4$).  Then $M$ has no
    3-dimensional closed transverse submanifold $W$ for which
    $\Sigma(W)\to\Sigma(M)$ is a Morita equivalence.  Indeed, by the lemma
    above this would mean that $L_W \into \Torus^4$ induces an isomorphism on
    $\pi_1$, where $L_W$ is a closed surface.  But there is no closed surface
    with fundamental group isomorphic to $\pi_1(\Torus^4)$.
\end{example}

%% file: chapters/currents.tex
\chapter{Currents}
\label{chapter:currents}

In this chapter, we first review some elements of the theory of currents
as developed by de Rham \cite{derham}.
Currents are generalizations of singular chains, and one particular class of them,
the so-called \emph{foliation currents}, have been used successfully in the study
of tautness of foliations.
Foliation currents are currents that ``lie tangent to'' a given foliation.

Foliation currents can be generalized to structure currents, which are required to
``lie tangent to'' a given cone structure. This setup is much more general.
The goal of this chapter is to show that these structure currents are
interesting in the study of calibrated symplectic foliations.
For a limited class of symplectic foliations, they allow the study of calibratability
to be done in a leafwise fashion.
We unify two conjectures due to Sullivan and Ghys, and try to give the reader
some intuition for why the new conjecture might hold.

Throughout this chapter, we will assume that $M$ is a smooth manifold
of dimension $m$. The codimension of the foliation under consideration will be $q$.

\section{Definitions}

We start by recalling the theory of currents, originally due to de Rham \cite{derham}.
All of the material here can be found in chapter 10 of \cite{candelconlon}.
Write $\Omega^p(M)$ for the space of all $p$-forms on $M$, equipped with the $C^{\infty}$-topology.
Write $\Omega_c^p(M)$ for the space of all compactly supported $p$-forms on $M$, equipped with the $C^{\infty}$-topology.

\begin{definition}
    A \emph{$p$-current} on $M$ is a continuous linear map $\Omega_c^p(M) \to \R$
    ($0\leq p\leq m$).
    We write $\Cur_p(M) = \Cur_p$ for the space of all $p$-currents on $M$.
    A \emph{compactly supported $p$-current} on $M$ is a continuous linear map
    $\Omega^p(M) \to \R$ ($0 \leq p\leq m$).
    We write $\Cur_p^c(M)$ for the space of all compactly supported $p$-currents on $M$.
\end{definition}

We will consider $\Cur_p^c(M)$ as a subset of $\Cur_p(M)$ by restriction.

\begin{remark}
    The \emph{support} of a current $c$ is defined in the way that is usual for
    distributions (generalized functions): a point $x\in M$ is not in the
    support if it has a neighborhood $U$ such that $c{\restriction_U} = 0$.
    This definition agrees with our use of ``compactly supported'' above.
\end{remark}

\begin{example}[Dirac currents]
    Let $v\in \Lambda^pTM$. Then we can consider $v$ as a current
    \[ \Omega^p(M) \to \R : \alpha \mapsto \alpha(v) .\]
    Such currents are called \emph{Dirac currents}. They are compactly supported.
\end{example}

\begin{example}[Smooth currents]
    Assume that $M$ is oriented, and let $\beta \in \Omega^q(M)$, where $p+q=m$.
    Then we can consider $\beta$ as a current
    \[ \Omega_c^p(M) \to \R : \alpha \mapsto \int_M \beta\wedge\alpha .\]
    Such currents are called \emph{smooth currents}, or also \emph{diffuse currents}.
    Note that $\beta$ is compactly supported as a current precisely if it is
    compactly supported as a form.
\end{example}

\begin{example}[Smooth singular chains]
    Every smooth $p$-simplex $c : \Delta_p \to M$ can be considered as a
    (compactly supported) $p$-current
    \[ \Omega^p(M) \to \R : \alpha \mapsto \int_{\Delta_p} c^*\alpha .\]
    More generally, every smooth singular $p$-chain can be considered as a $p$-current.
    Intuitively, these currents are not as concentrated as a Dirac current, and
    not as spread out as a diffuse current.
\end{example}

While currents are sometimes used as generalizations of differential forms
(they are ``distributional forms''), we will take the viewpoint that currents
are generalizations of chains. Indeed, currents can be naturally pushed forward,
but not pulled back.
A Dirac current should be thought of as an infinitesimally small singular $p$-chain
that is infinitely concentrated, where the ``infinitesimally'' and the ``infinitely''
are balanced out to get a finite number when integrating a $p$-form over it
(much like the Dirac delta function is like a peak that is infinitesimally
narrow and infinitely high, where the ``infinitesimally'' and the
``infinitely'' are balanced out to get finite answers when integrating).
A smooth current can be thought of as a singular chain that has been spread out
smoothly over the manifold (it has ``diffused'').

We now focus our attention on the case where $M$ is compact.
In this case, we do not need to keep track of the compactness of supports.
We equip the space $\Cur_p(M)$ with a topology.
We will say that a subset of $\Omega^p(M)$ is bounded if it is bounded relative
to each $C^k$-norm.  The topology on $\Cur_p$ is generated by sets of
the form
\[ \{ c \in \Cur_p(M) \mid \lvert c(\alpha)\rvert < \eps \text{ for all $\alpha \in B$} \} ,\]
where $\eps > 0$ and $B \subset \Omega^p(M)$ is bounded.
This topology on $\Cur_p(M)$ has several nice properties.

\begin{proposition}[{{de Rham \cite[theorem 13]{derham}}}]
    The topology on $\Cur_p(M)$ as described above turns the space of $p$-currents
    into a locally convex topological vector space. It is Hausdorff.
    Moreover, the continuous linear functionals on $\Cur_p(M)$ are precisely
    of the form $\Cur_p(M) \to \R : c \mapsto c(\alpha)$ for some $\alpha\in\Omega^p(M)$.
    In other words, $\Omega^p(M)$ and $\Cur_p(M)$ are each other's continuous dual.
\end{proposition}

We have seen that currents generalize singular chains.
As in the case of chains, we can define the \emph{boundary} of a current, as follows.
Exterior differentiation
$ d : \Omega^p(M) \to \Omega^{p+1}(M) $
has a dual map
$ \partial : \Cur_{p+1}(M) \to \Cur_p(M) $
defined by
\[ (\partial c)(\alpha) =  c(d\alpha) .\]
This map $\partial$ is the boundary operator on currents.

\begin{definition}
    A current $c \in \Cur_p(M)$ is \emph{closed} if $\partial c = 0$.
    A current $c \in \Cur_p(M)$ is \emph{exact} if $c = \partial c'$ for some
    current $c' \in \Cur_{p+1}(M)$.
    We will write $\ZZ_p(M)$ for the space of closed $p$-currents on $M$,
    and $\BB_p(M)$ for the space of exact $p$-currents on $M$.
\end{definition}

\begin{lemma}[{{see \cite[lemma 10.1.24]{candelconlon}}}]
    The sets $\ZZ_p(M)$ and $\BB_p(M)$ are closed subspaces of $\Cur_p(M)$.
\end{lemma}

Under the inclusion of singular chains into the space of currents,
this boundary operator corresponds precisely to the usual boundary operator
(this is Stokes' theorem).
Moreover, a differential form is closed iff it is closed as a current,
and a differential form is exact iff it is exact as a current.

The following definition and result are due to de Rham \cite{derham}.

\begin{definition}
    The \emph{de Rham homology} of $M$ is defined as
    \[ H_p^{\dR}(M) = \frac{\ZZ_p(M)}{\BB_p(M)} .\]
\end{definition}

\begin{theorem}[de Rham]
    The pairing
    $H_p^{\dR}(M) \times H^p(M;\R) \to \R : ([c],[\alpha]) \mapsto c(\alpha)$
    is well-defined and determines an isomorphism
    \[ H_p^{\dR}(M) \to H^p(M;\R)' \cong H_p(M;\R) .\]
    The inclusion of ordinary $p$-cycles into $\Cur_p(M)$ descends to homology
    to give the inverse map
    \[ H_p(M;\R) \to H_p^{\dR}(M) .\]
\end{theorem}

(Here, $H^p(M;\R)'$ is the dual of $H^p(M;\R)$. Note that these are finite-dimensional
because $M$ is assumed to be compact.)

\section{Foliation currents}

Suppose that $M$ is a compact manifold equipped with an oriented foliation $\F$
of dimension $p$.
There is a special class of currents associated to $\F$, the so-called
\emph{foliation currents}, first introduced by Sullivan \cite{sullivan}.
The book \cite{candelconlon} is also a good reference for this material.

\begin{definition}
    A \emph{Dirac foliation current for $\F$} (or \emph{Dirac $\F$-current} for short) is
    a $p$-current of the form
    \[ \Omega^p(M) \to \R : \alpha \mapsto \alpha(v_x) ,\]
    where $x\in M$ and $v\in \Lambda^pT_x\F$ are fixed, and $v$ is positively oriented
    with respect to the orientation of $\F$.
\end{definition}

Intuitively, a Dirac foliation current is a Dirac current that measures
``along $\F$''. To define what a foliation current is in general, we need to consider
the cone generated by Dirac foliation currents.

\begin{definition}
    A \emph{cone} in a real vector space is a non-empty subset of the vector
    space that is invariant under multiplication by a non-negative number.
\end{definition}

The cone generated by the Dirac foliation currents for $\F$ (i.e.\ the smallest cone
in $\Cur_p(M)$ that contains all Dirac foliation currents for $\F$) is not closed.
Note also that this cone intersects $\ZZ_p(M)$ only at 0 (if $p\geq 1$).
Once we take the closure of this cone, we obtain the cone of $\F$-currents.

\begin{definition}
    A \emph{foliation current for $\F$} (or \emph{$\F$-current} for short)
    is a $p$-current in the closure of the cone in $\Cur_p(M)$ spanned by Dirac
    foliation currents.
    We write $\Cur_{\F}$ for the space of foliation currents.
    We also write $\ZZ_{\F} = \Cur_{\F} \cap \ZZ_p$ for the space of \emph{foliation cycles}
    for $\F$, and $\BB_{\F} = \Cur_{\F} \cap \BB_p$ for the space of
    \emph{foliation boundaries} for $\F$.
\end{definition}

\begin{example}[Compact leaf]
    Suppose that $\F$ is an oriented $p$-dimensional foliation on $M$,
    and suppose that $\F$ has a compact leaf $L$.
    Then this leaf determines a current
    \[ c : \Omega^p(M) \to \R : \alpha \mapsto \int_L \alpha .\]
    This current is closed.
    It is also an $\F$-current:
    the integral $\int_L \alpha$ can be approximated by Riemann sums,
    and this approximation shows how to approximate $c$ by Dirac foliation currents.

    If, instead, we take some open subset $U \subset L$ of the leaf, we can consider
    the current $\alpha \mapsto \int_U \alpha$. It is still a foliation current,
    but it is not, in general, closed.
\end{example}

\begin{example}[Compact oriented manifold]
    Suppose that $M$ is a compact oriented manifold.
    We can consider the codimension-0 foliation on $M$ (its only leaf is $M$).
    One of the foliation currents for this foliation is the manifold itself:
    \[ \alpha \mapsto \int_M \alpha .\]
    This is just a particular case of the example above, but for this (rather
    uninteresting) foliation, all foliation cycles are easily understood:
    they are precisely the multiples of this one
    (this statement is a variation of exercise 10.1.25 in \cite{candelconlon}).
    One way to see this is as follows: suppose that $c$ is a foliation cycle
    on $M$. Then there is a constant $\lambda \in \R$ such that
    $[c] = \lambda [M] \in H_m(M;\R)$. Therefore, the foliation cycle $c - \lambda M$
    is exact. But if a current is exact, it vanishes on all closed forms, and all
    top-degree forms are closed. Therefore $c - \lambda M$ must be the current 0,
    showing that $c = \lambda M$ as currents.

    This observation agrees with the intuition that the only way to piece together
    infinitely many infinitesimally small top-dimensional singular simplices,
    and end up without boundary, is to end up with a multiple of the entire manifold.
\end{example}

The importance of foliation currents is illustrated by the following theorem.

\begin{theorem}[{{Sullivan, \cite[theorem II.2]{sullivan}}}]
Let $(M,\F)$ be a compact, connected manifold with an oriented and co-oriented
codimension-1 foliation on it. Then the following are equivalent:
\begin{enumerate}
\item there is a closed (resp.\ exact) $p$-form on $M$ that is non-vanishing on every leaf;
\item 0 is the only foliation boundary (resp.\ cycle) for $\F$.
\end{enumerate}\label{thm:sullivan-cycle}
\end{theorem}

This result shows the equivalence of
items (\ref{item:criteria-tautness-cycle}) and
(\ref{item:criteria-tautness-form}) in \cref{thm:criteria-tautness}.
The next result characterizes the cohomology classes of closed forms that
can be realized by closed forms as in \cref{thm:sullivan-cycle},
in certain cases.

\begin{theorem}[{{Sullivan, \cite[theorem I.7]{sullivan}}}]
    Let $(M,\F)$ be a compact, connected manifold with an oriented and co-oriented
    codimension-1 foliation on it. Suppose that there exists a closed $p$-form
    on $M$ that is non-vanishing on every leaf, but not an exact one.
    Let $C_{\F}$ be the image of $\ZZ_{\F}$ inside $H_p(M;\R)$.
    Then $C_{\F}$ is a compact, convex cone.
    The dual cone
    \[ C^{\F} = \{ [\alpha] \in H^p(M;\R) \mid [\alpha](c) \geq 0 \text{ for
        all $c\in C_{\F}$} \} \subset H^p(M;\R) \]
    has non-empty interior consisting of precisely those cohomology classes that can be
    represented by a form that is non-vanishing on every leaf. \label{thm:sullivan-cone}
\end{theorem}

\begin{example}[Linear foliation of the 2-torus]
    We will use the linear foliation of the 2-torus (an example we gave on
    page~\pageref{ex:linear-foliation-torus}).
    We orient the torus as $dx\wedge dy>0$.
    Consider the vector field
    \[ V = \cos(\theta)\dd{x}+\sin(\theta)\dd{y} .\]
    The foliation $T\F$ is spanned by $V$, and we will orient the foliation
    by prescribing that this vector field be positively oriented.  (So
    changing $\theta$ to $\theta + \pi$ gives the same foliation, with opposite
    orientation.)
    Now $H^1(M;\R)$ is 2-dimensional and spanned by $[dx]$ and $[dy]$.
    The dual basis in $H_1(M;\R)$ is given by $\{[\gamma_x],[\gamma_y]\}$,
    circles running in the $x$- and $y$-directions respectively.
    We will calculate the cones $C_\F$ and $C^\F$.

    Consider the foliation current
    \[ c : \Omega^1(M) \to \R : \beta \mapsto \int_0^1 \int_0^1 \beta(V) \,dx dy .\]
    It is closed. Moreover,
    $c(dx) = \cos(\theta)$ and $c(dy) = \sin(\theta)$
    so that the class of $c$ in $H_1(M;\R)$ is
    $\cos(\theta)[\gamma_x]+\sin(\theta)[\gamma_y]$.  We conclude that
    \[ \tag{$\dag$}\label{eq:bound-on-clowerf} C_\F \supset \{ \lambda
        (\cos(\theta)[\gamma_x]+\sin(\theta)[\gamma_y]) \mid \lambda \geq 0 \} .\]

    Now consider the closed form
    \[ a\,dx + b\,dy,\]
    where $a,b\in\R$.
    Whenever $a \cos(\theta) + b \sin(\theta) > 0$, this is a closed form that
    is positive on every leaf.
    Because $C^{\F}$ is closed, this shows that
    \[ \tag{$\ddag$}\label{eq:bound-on-cupperf} C^\F \supset \{ a [dx] + b [dy]
        \mid a \cos(\theta) + b\sin(\theta) \geq 0 \} .\]
    Since $C^\F$ is a closed cone dual to $C_\F$, we can conclude that the inclusions
    (\ref{eq:bound-on-clowerf}) and (\ref{eq:bound-on-cupperf}) are in fact equalities.
    These cones are shown in \cref{fig:cones-for-2-torus}.
\end{example}

\begin{figure}
    \centering
    \begin{tikzpicture}[scale=1.5]
        \begin{scope}
            \clip (-1.5,-1.5) rectangle (1.5,1.5);
        \end{scope}
        \begin{scope}
            \clip (-1,-1) rectangle (1.5,1.5);
            \draw[style=very thick] (0,0) -- (8.66,5);
        \end{scope}
        \draw[style=thin] (-1.5,0) -- (1.5,0);
        \draw[style=thin] (0,-1.5) -- (0,1.5);
        \draw[style=thick,->] (0,0) -- (1,0) node[below] {$[\gamma_x]$};
        \draw[style=thick,->] (0,0) -- (0,1) node[left] {$[\gamma_y]$};
        \draw[style=thin] (0.3,0) node[label={[shift={(0.26,-0.23)}]$\theta$}] {} arc (0:30:0.3);
    \end{tikzpicture}
    \hspace{1em}
    \begin{tikzpicture}[scale=1.5]
        \begin{scope}
            \clip (-1.5,-1.5) rectangle (1.5,1.5);
            \draw[style=thin] (5,-8.66) -- (-5,8.66) node[left] {$[\alpha]$};
            \fill[lightgray] (5,-8.66) -- (-5,8.66) -- (10,8.66) -- (10,-8.66);
        \end{scope}
        \draw[style=thin] (-1.5,0) -- (1.5,0);
        \draw[style=thin] (0,-1.5) -- (0,1.5);
        \draw[style=thick,->] (0,0) -- (1,0) node[below] {$[dx]$};
        \draw[style=thick,->] (0,0) -- (0,1) node[right] {$[dy]$};
        \draw[style=thick,->] (0,0) -- (-0.5,0.866) node[left] {$[\alpha]$};
    \end{tikzpicture}
    \caption{On the left a picture of $H_1(\Sphere^1 \times \Sphere^1;\R)$. The
        thick ray is the cone $C_{\F}$.
        On the right a picture of $H^1(\Sphere^1 \times \Sphere^1;\R)$. The
        shaded region, including its boundary, is the cone $C^{\F}$.
        Its interior consists of those classes that can be realized with a closed 1-form
        that is positive on each leaf.}
    \label{fig:cones-for-2-torus}
\end{figure}
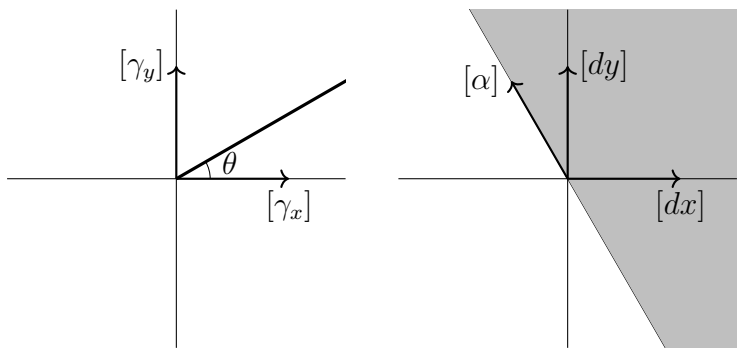

We saw earlier that on a compact oriented manifold, all foliation cycles of
the codimension-0 foliation are multiples of the manifold.
While that is an easy consequence of de Rham homology, there is vast generalization
of this statement that is much harder to prove. It is also due to Sullivan.
We first need to introduce some language.

Suppose that $\F$ is an oriented foliation. Fix a regular oriented foliated
atlas for $\F$, indexed by $i\in I$. \phantomsection\label{page:where-we-need-a-regular-atlas}
Pick a complete transversal $T_i$ in each chart (so
that $T_i$ is the space of plaques in this chart).

\begin{definition}
    A \emph{transverse measure} for $\F$ is a measure on $\bigsqcup_{i\in I} T_i$
    that is invariant under holonomy.
\end{definition}

Such transverse measures are often called \emph{transverse invariant measures} to
stress their invariance under holonomy. Transverse measures should be thought of
as measures on the leaf space $M/\F$.
Note that the choice of transversals does not matter: if we choose two different families
of transversals, there is a bijection between holonomy-invariant measures for
one family and the other.

If we have a transverse measure $\mu$ that is finite on compact sets, we can
associate to it a foliation cycle $c_{\mu}$ as follows
(see \cite[section 10.1]{candelconlon}).
Choose a partition of unity $\{f_i\}_{i\in I}$ subordinate to the foliation atlas.
The current $c_{\mu}$ is defined as
\[ c_{\mu} : \Omega^p(M) \to \R : \alpha \mapsto \sum_{i\in I} \int_{T_i}
\left( \int_{\text{plaque}} f_i \alpha \right) \,d\mu{\restriction_{T_i}} .\]
That is, in each chart we integrate $f_i\alpha$ over each plaque. This results in
a function on the transversal $T_i$, which we then integrate using the measure $\mu$.
We sum the contributions of all charts, giving us $c_{\mu}(\alpha)$.
This current does not depend on the choice of partition of unity, and it is a
foliation cycle.

We are now ready to state Sullivan's result.

\begin{theorem}[{{Sullivan, \cite{sullivan}}}]
    The correspondence $\mu \mapsto c_{\mu}$ is a bijection between \label{thm:transverse-measures}
    \[ \left\{ \parbox{15em}{\centering transverse measures for $\F$ that are finite on compact sets} \right\} \text{ and }
        \left\{ \text{foliation cycles for $\F$} \right\}.
    \]
\end{theorem}

\begin{example}[Foliation cycles for the Reeb foliation of $\Sphere^3$]
    Let $\F$ be the Reeb foliation of $\Sphere^3$.
    Write $L$ for the torus leaf (this is the only compact leaf).
    Choose $x\in L$.
    Choose a short compact transversal $T$ through $x$.
    This transversal intersects every leaf, so a transverse measure is completely
    determined by its restriction to $T$.

    We claim that every transverse measure is supported on $\{x\}$.
    Choose $y_0$ in the interior of $T$ close to (but different from) $x$.
    We will show that $y_0$ is not in the support of $\mu$.
    One of the elements $[\gamma] \in \pi_1(L,x)$ has holonomy that maps $y_0$ to a point
    $y_1 \in T$ that is closer to $x$.
    Transport this point $y_1$ to obtain another point $y_2 \in T$, even closer to $x$.
    Continuing this way, we obtain a sequence of points $y_0, y_1, y_2, ...$ in $T$
    converging to $x$.
    By holonomy-invariance, each of the intervals $(y_{2k},y_{2k+2})$ ($k\in\Z_{\geq 0}$)
    has the same measure.
    They are disjoint, and their union has finite measure. Therefore, each of them
    has measure 0. This shows that $y_1$ is not in the support of $\mu$,
    and therefore $y_0$ is not either.
    This proves that any transverse measure for $\F$ is a multiple of the torus leaf
    (when considered as a transverse measure, it is a multiple of the Dirac measure
    supported on the torus leaf).
\end{example}

\section{Structure currents}

Foliation currents are a very special class of currents.
Sullivan developed a more general theory, that of \emph{structure currents}.
These structure currents can be used to study more general structures on manifolds
than just foliations.
We review the essential parts of this theory.

Let $M$ be a manifold of dimension $m$, and let $0\leq p\leq m$.

\begin{definition}
    A \emph{cone structure} on $M$ is a subset $C$ of $\Lambda^pTM$
    that satisfies the following properties:
    \begin{enumerate}
        \item for each $x\in M$, the fiber $C_x = C\cap \Lambda^pT_xM$ is a convex cone in $\Lambda^pT_xM$;
        \item there exists a $\beta\in\Omega^p(M)$ such that $\beta(v) > 0$
            for every non-zero $v\in C$.
    \end{enumerate}
\end{definition}

This definition is more general than Sullivan's \cite[definition 1.2]{sullivan} in two ways:
\begin{itemize}
    \item the continuity requirement on the cones is replaced by the second condition above;
        this is a weaker condition, yet it suffices in all the proofs;
    \item we do not require that each cone $C_x$ be compact.
\end{itemize}

We will require the compactness separately when needed.

\begin{definition}
    We say that a cone structure is \emph{compact} if the image of $C\setminus\{0\}$
    under the map
    \[ \Lambda^pTM\setminus\{0\} \to (\Lambda^pTM\setminus\{0\}) / \R_{>0} \]
    is compact.
\end{definition}

In other words, a cone structure is compact if the set of all its directions is compact.
Suppose that $C$ is a cone structure on $M$. The following are generalizations of
(Dirac) foliation currents.

\begin{definition}
    A \emph{Dirac structure current} for $C$ is a current of the form
    \[ \Omega^p(M) \to \R : \alpha \mapsto \alpha(v) \]
    where $v\in C$.
\end{definition}
\begin{definition}
    A \emph{structure current} for $C$ (or \emph{$C$-current}) is a current that lies in the closure
    of the cone in $\Cur_p(M)$ spanned by Dirac structure currents for $C$.
    We write $\Cur_C = \Cur_C(M)$ for the space of structure currents.
    We also write $\ZZ_C = \Cur_C(M) \cap \ZZ_p$ for the \emph{structure cycles}
    (or $C$-cycles), and $\BB_C = \Cur_C(M) \cap \BB_p$ for the \emph{structure boundaries}
    (or $C$-boundaries).
\end{definition}

If $\beta\in\Omega^p(M)$ is as in the definition of a cone structure,
we will refer to $\Cur_C \cap \beta^{-1}(1)$ as a \emph{base} of the cone $\Cur_C$ (see \cref{fig:base-of-cone}).
Here is a simple but important observation about compact cone structures.

\begin{lemma}[{{see \cite{sullivan}}}]
    If $C$ is a compact cone structure, the base $\Cur_C\cap\beta^{-1}(1)$ is
    \label{lem:compact-base} compact.
\end{lemma}

\begin{figure}
    \centering
    \includegraphics[width=0.4\linewidth]{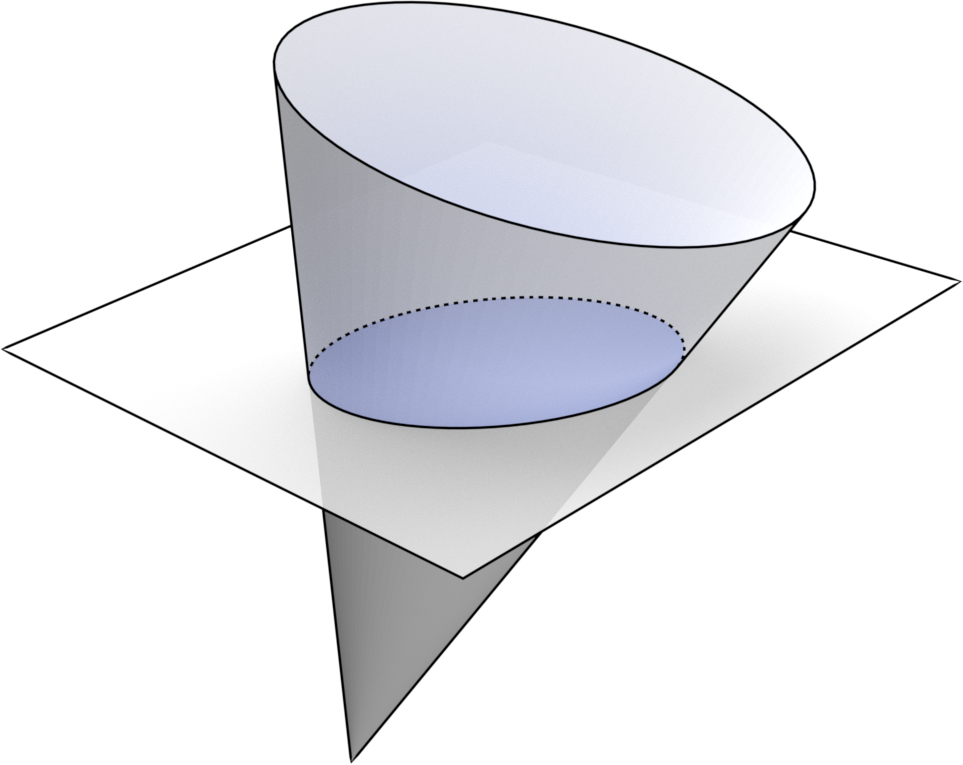}
    \caption{Low-dimensional analogue of the cone $\Cur_C$
        inside of $\Cur_p$. The cone has its tip at the origin.
        Intersecting it with the affine hyperplane $\beta^{-1}(1)$ gives us
        the base of the cone, shown in blue. This base is convex.
        If $C$ is a compact cone structure, the base of $\Cur_C$ is also compact.}
    \label{fig:base-of-cone}
\end{figure}

\begin{definition}
    Suppose that $C$ is a cone structure on $M$.
    A $p$-form $\alpha$ on $M$ is said to be \emph{positive on $C$} if $\alpha(v) > 0$
    for all non-zero $v\in C$.
\end{definition}

\begin{example}
    Suppose that $M$ is a manifold equipped with an oriented codimension-1 foliation $\F$
    (so $p = m-1$). Then $\Lambda^pTM$ is a line bundle over $M$.
    Taking $C_x$ to be the ray in $\Lambda^pT_xM$ determined by the chosen orientation of $\F$,
    we see that structure currents are precisely foliation currents in this case,
    and the existence of a closed form that is positive on $C$ amounts to tautness of $\F$.
\end{example}

\begin{example}
    Let us give a toy example. We will use $M = \R^2$ and $p=1$.
    Choose a non-vanishing vector field $X$ on $\R^2$.
    Set
    \[ C = \{ \lambda X_x \mid \lambda\in\R_{\geq 0}, x\in \R^2 \} .\]
    In other words, $C$ consists of the rays in $T\R^2$ determined by $X$.
    Note that $C$ is really the structure cone for the foliation obtained as the flow of $X$.

    We claim that there exists a closed form positive on $C$ precisely if there is
    a dimension-1 foliation $\F$ of the plane such that $\F$ and $X$ are transverse.
    If there exists a closed form that is positive on $C$, then its kernel
    defines such a foliation $\F$.
    Suppose, conversely, that there is such a foliation $\F$.
    Every foliation of the plane is the collection of (the connected components of)
    level sets of a function $f : \R^2\to\R$ with non-vanishing gradient \cite{kamke}
    (see also \cite[theorem 42]{kaplan} and \cite[section 2]{haefligerreeb}
    for more information about foliations of the plane).
    Then $df$ (or $-df$) is a closed form that is positive on $C$.
\end{example}

\begin{example}
    Suppose that $M$ is a manifold equipped with an almost complex structure $J : TM \to TM$,
    $J^2=-\id$.
    Choose
    \[ C = \{ v\wedge J(v) \mid v \in TM \} .\]
    Then a form $\alpha$ is positive on $C$ precisely if $\alpha(v,J(v)) > 0$ for all
    non-zero $v\in TM$.
    The existence of a closed form that is positive on $C$ therefore amounts to
    the existence of a symplectic structure on $M$ that tames $J$.
    This example appears in \cite[\S 10]{sullivan}.
\end{example}

\begin{example}
    Generalizing the previous example,
    suppose that $(M,\F)$ is a foliated manifold, and $J : T\F \to T\F$ is a leafwise
    almost complex structure.
    We choose
    \[ C = \{ v\wedge J(v) \mid v\in T\F \} .\]
    Then a closed form that is positive on $C$ is precisely a calibrated leafwise
    symplectic form that tames $J$.
\end{example}

\begin{example}
    We can also specialize to (the closure of) one specific leaf.
    Suppose that $(M,\F)$ is a foliated manifold and that $L$ is a leaf of $\F$.
    Suppose moreover that $J : T\F \to T\F$ is a leafwise almost complex structure.
    We choose
    \[ C = \{ v\wedge J(v) \mid v \in T_x\F, x\in \cl{L} \} .\]
    Then a closed form is positive on $C$ precisely if its restriction to $L$ is
    a symplectic form for which $\omega(v,J(v))$ is bounded away from 0 for unit-length $v$.
\end{example}

The following results are the generalizations of \cref{thm:sullivan-cycle} and
\cref{thm:sullivan-cone} to structure currents. Their proofs can be found in
\cite[theorem I.7]{sullivan} (and one can check that their proofs go through under our
slightly weaker assumptions).

\begin{theorem}
    Suppose that $C$ is a compact cone structure on a compact manifold $M$.
    \label{thm:sullivan-hahn-banach}
    Then the following are equivalent:
    \begin{enumerate}
        \item there is a closed (resp.\ exact) $p$-form on $M$ that is positive on $C$;
        \item 0 is the only structure boundary (resp.\ cycle) for $C$.
    \end{enumerate}
\end{theorem}

\begin{theorem}
    Let $C$ be a compact cone structure on a compact manifold $M$.
    Suppose that there is a closed $p$-form that is positive on $C$,
    but not an exact one.
    Let $H_C$ be the image of $\ZZ_C$ inside $H_p(M;\R)$.
    Then $H_C$ is a convex cone with compact base.
    The dual cone
    \[ H^C = \{ [\alpha]\in H^p(M;\R) \mid \alpha(c) \geq 0 \text{ for all $c\in H_C$} \}
        \subset H^p(M;\R) \]
    has non-empty interior consisting of precisely those cohomology classes that
    can be represented by a form that is positive on $C$.
    \label{thm:sullivan-cone-C}
\end{theorem}

This theorem shows that structure cycles can be used in trying to understand forms
that are positive on a chosen cone structure $C$.
The technical tool underlying the theorems above is a separation theorem
that follows from the Hahn-Banach theorem of functional analysis.
Even though we will not need to know anything about the proof of
\cref{thm:sullivan-hahn-banach}, we will give a
sketch of what is going on behind the scenes.
Here is the technical heart of the matter.

\begin{theorem}[Hahn-Banach separation theorem]
    Let $W$ be a locally convex topological real vector space.
    Suppose that $X \subset W$ is a non-empty compact convex subset,
    and $V\subset W$ a non-empty closed convex subset.
    If $X$ and $V$ are disjoint, then there is a continuous linear functional
    $\alpha : W \to \R$ and a real number $d$ such that $\alpha(V) < d < \alpha(X)$.
\end{theorem}

\begin{proof}[Sketch of proof of {{\cref{thm:sullivan-hahn-banach}}}]
    Suppose there is a closed (resp. exact) $p$-form $\omega$ on $M$ that is positive on $C$.
    Then $\omega$ takes positive values on all non-zero $C$-currents.
    However, it must vanish on exact (resp. closed) currents, and therefore 0 is the only
    structure boundary (resp. cycle) for $C$.

    Conversely, suppose that 0 is the only structure boundary (resp. cycle) for $C$.
    This means that $\Cur_C$ only touches the space $\BB_p$ (resp. $\ZZ_p$) at the origin.
    We can therefore, by Hahn-Banach, find a hyperplane containing $\BB_p$
    (resp. $\ZZ_p$) that doesn't intersect the base of the cone. The corresponding linear
    functional $\Cur_p \to \R$ is precisely the closed (resp. exact) $p$-form we are looking for.
\end{proof}

\Cref{fig:cone-situations} illustrates the possible ways that $\Cur_C$ can be positioned
relative to $\ZZ_p$ and $\BB_p$.

Let us list several interesting cone structures.

\begin{figure}
    \centering
    \begin{minipage}[b]{.3\linewidth}
        \centering
        \includegraphics[width=0.9\linewidth]{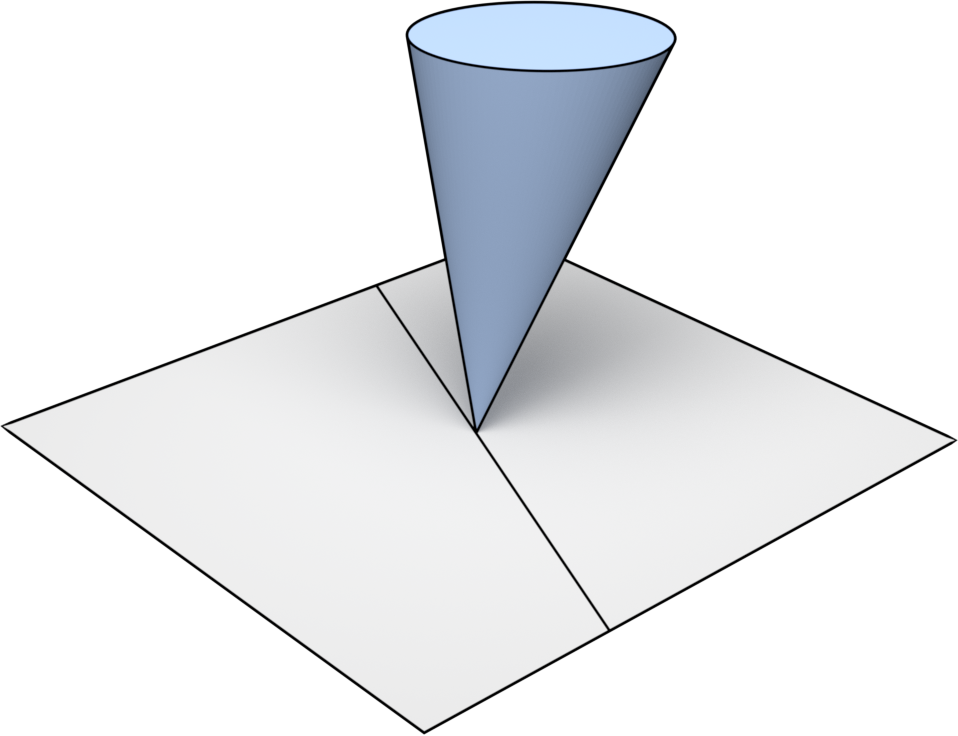}
    \end{minipage}%
    \begin{minipage}[b]{.3\linewidth}
        \centering
        \includegraphics[width=0.9\linewidth]{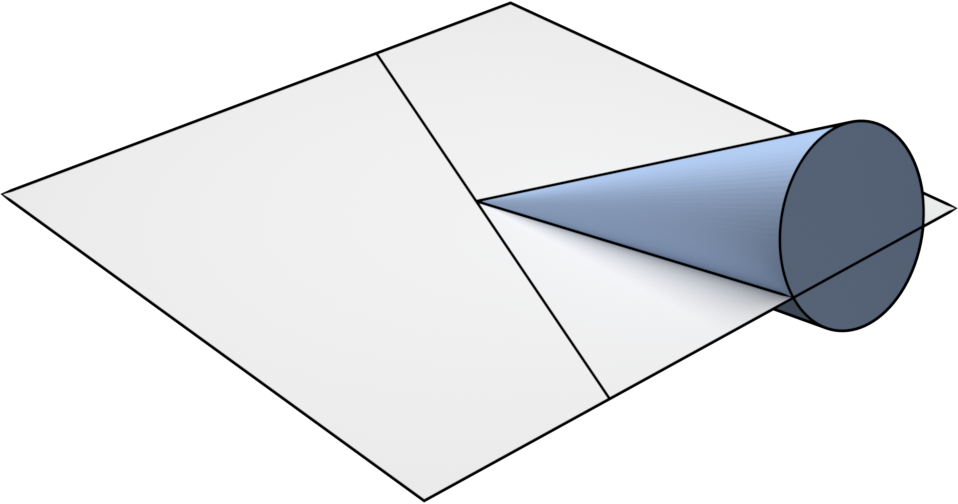}
    \end{minipage}%
    \begin{minipage}[b]{.3\linewidth}
        \centering
        \includegraphics[width=0.9\linewidth]{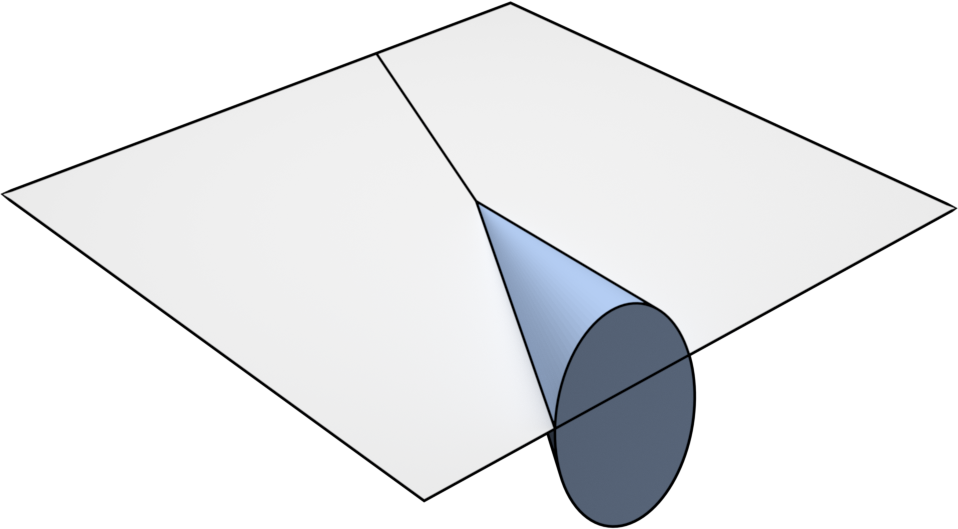}
    \end{minipage}%
    \caption{The possible ways that the cone $\Cur_C$ can be positioned
    relative to $\ZZ_p$ and $\BB_p$. The blue cone represents $\Cur_C$.
    The space $\ZZ_p$ is represented by the plane, and its subspace $\BB_p$
    is in this case shown as a line.
    \emph{Left:} the cone intersects $\ZZ_p$ only at the tip; in this case there is
    an exact form positive on $C$.
    \emph{Middle:} the cone intersects $\ZZ_p$ non-trivially, but only touches $\BB_p$
    at the tip; in this case there is a closed, but not an exact form that is positive
    on $C$.
    \emph{Right:} the cone intersects $\BB_p$ non-trivially; there is no closed form that
    is positive on $C$.}
    \label{fig:cone-situations}
\end{figure}

The following result gives some insight into the general form of structure currents.

\begin{proposition}[{{\cite[proposition I.8]{sullivan}}}]
    Any structure current $c$ can be represented as
    \[ c=\int_M s\,d\mu ,\]
    (meaning $c(\alpha) = \int_M \alpha(s) \,d\mu$)
    where $\mu$ is a non-negative measure on $M$, and $s$ is a $\mu$-integrable
    section of $\Lambda^pTM$ taking values in $C$. \label{prop:general-form-structure-current}
\end{proposition}

It is possible to multiply a structure current by a bounded measurable function,
by multiplying the function $s$ above by it.

\section{Extreme currents}

While foliation cycles can be understood using \cref{thm:transverse-measures},
no such theorem is available for structure cycles in general.
Nevertheless, some things can be said about them.
\Cref{thm:transverse-measures} describes foliation cycles as ``mixtures'' of leaves.
It is still true for structure cycles that they can be described as ``mixtures''
of certain structure currents. In fact, this is just a special case of a standard
result in the analysis of locally convex topological vector spaces, the Krein-Milman
theorem.

\begin{definition}
    Let $A\subset V$ be a subset of a vector space. A point $a\in A$ is an
    \emph{extreme point} of $A$ if for all $b,c\in A$ and $t\in(0,1)$ such that
    \[ a = tb + (1-t)c \]
    we must have $a=b=c$. In other words, a point is an extreme point if it is not
    a convex combination of two other points in $A$.
\end{definition}

\begin{theorem}[Krein-Milman]
    A compact convex set in a locally convex topological vector space
    is the closure of the convex hull of its extreme points.
    \label{thm:krein-milman}
\end{theorem}

If $A \subset V$ is a cone in a vector space, it has no non-zero extreme points.
One can study \emph{extreme rays} instead.

\begin{definition}
    Let $A\subset V$ be a cone in a vector space. Let $a\in A\setminus\{0\}$.
    We say that the ray $[a] = \{ \lambda a \mid \lambda\in\R_{>0} \}$ is an
    \emph{extreme ray} of $A$ if for all $b,c\in A$ and $t\in(0,1)$ such that
    \[ a = tb + (1-t)c \]
    we must have $b,c\in[a]$.
\end{definition}

This leads us to the following definition.

\begin{definition}
    A structure cycle $c \in \ZZ_C(M)$ is called \emph{an extreme cycle} if its ray
    $[c]$ is an extreme ray of $\ZZ_C(M)$.
    A structure boundary $c \in \BB_C(M)$ is called \emph{an extreme boundary} if its ray
    $[c]$ is an extreme ray of $\BB_C(M)$.
\end{definition}

From the Krein-Milman theorem, together with \cref{lem:compact-base},
we then have the following result.

\begin{proposition}
    Suppose that $C$ is a compact cone structure.
    Every structure cycle for $C$ can be approximated arbitrarily closely
    by a sum of positive multiples of extreme cycles.
    Every structure boundary for $C$ can be approximated arbitrarily closely
    by a sum of positive multiples of extreme boundaries.
    \label{prop:krein-milman-currents}
\end{proposition}

\begin{proof}
    We prove the result for cycles. The result for boundaries is analogous.
    Suppose that $c$ is a structure cycle for $C$.
    Let $\beta$ be as in the definition of a cone structure.
    Applying the Krein-Milman theorem to the base $\beta^{-1}(1) \cap \ZZ_C$ of
    the cone $\ZZ_C$ (this is a compact convex set) shows that $c / \beta(c)$
    can be approximated arbitrarily closely
    by a sum of positive multiples of extreme cycles.
    Therefore, the same holds for $c$.
\end{proof}

The following theorem allows us to be more precise about ``mixtures'' of extreme currents.
It says, roughly speaking, that every point in a compact convex set can be obtained
as a weighted barycenter of extreme points.

\begin{theorem}[Choquet-Bishop-de Leeuw, {{\cite[theorem 5.6]{bishop-deleeuw}}}]
    Let $X$ be a compact convex subset of a real locally convex topological vector
    space. Let $E$ be the set of extreme points of $X$. Let $\mathcal{S}$ be
    the $\sigma$-algebra generated by $E$ and by the Baire subsets of $X$.
    Then each $x\in X$ has a representation of the form
    \[ \int y \,d\mu(y) \]
    for some non-negative measure $\mu$ on $\mathcal{S}$ that satisfies
    \[ \mu(E) = \mu(X) = 1.\]
\end{theorem}

\begin{remark}
    The theorem requires the use of the $\sigma$-algebra $S$ because, in general,
    the set of extreme points of $X$ need not be Borel.
    It suffices to take the $\sigma$-algebra generated by the Baire sets
    (recall that Baire sets are Borel), and the set of extreme points.
\end{remark}

The theorem above immediately leads to the following variation of \cref{prop:krein-milman-currents}.

\begin{proposition}
    Suppose that $C$ is a compact cone structure.
    Write $E$ for the set of extreme structure cycles for $C$.
    Let $\mathcal{S}$ be the $\sigma$-algebra generated by $E$ and by the Baire
    subsets of $\ZZ_C$.
    Every structure cycle for $C$ can be written in the form
    \[ \int c'\,d\mu(c') ,\]
    where $\mu$ is a measure such that
    \[ \mu(E) = \mu(\ZZ_C) = 1 .\]
\end{proposition}

Let us now focus on extreme currents for cone structures that lie tangent to a foliation.
There is a straightforward restriction on what the support of an extreme cycle can look like.

\begin{definition}
    Two leaves $L_1$, $L_2$ of a foliation are called \emph{estranged} if
    there are saturated opens $U_1 \supset L_1$ and $U_2 \supset L_2$ such that
    $U_1 \cap U_2 = \emptyset$.
\end{definition}

\begin{proposition}
Suppose that $C\subset \Lambda^pT\F$ is a compact cone structure on $M$.
If $c$ is an extreme $C$-cycle, no two estranged leaves intersect its support.
\end{proposition}

\begin{proof}
    Suppose that $c$ is $C$-cycle whose support intersects the estranged leaves
    $L_1$ and $L_2$. Let $U_1 \supset L_1$ and $U_2\supset L_2$ be disjoint saturated opens.
    By the lemma below, the decomposition
    \[ c = \ind_{U_1} c + (1-\ind_{U_1}) c \]
    shows that $c$ can be written as a non-trivial sum of two $C$-cycles.
    Therefore, $c$ is not extreme (here, $\ind_{U_1}$ is the indicator function of
    the set $U_1$).
\end{proof}

If we think of currents as generalizations of singular chains, the following
should come as no surprise. The proof is similar to that on pages 246--249 of
\cite{candelconlon}.

\begin{lemma}
    Suppose that $(M,\F)$ is a compact foliated manifold, and let $C \subset \Lambda^pT\F$
    be a compact cone structure (so that $C$ ``lies tangent to'' $\F$).
    Suppose that $f : M \to \R$ is a bounded measurable function that is
    constant along leaves.
    If $c \in \Cur_C(M)$ is closed, then so is $fc$.
\end{lemma}

\begin{proof}
    Using \cref{prop:general-form-structure-current}, let $s$ and $\mu$ be such that
    \[ c : \Omega^p(M) \to \R : \alpha \mapsto \int_M \alpha(s) \,d\mu .\]
    Pick a point in $M$. We will prove that it is not in the support of $\partial(fc)$.
    Let $x_1, \ldots, x_{p}, y_1, \ldots, y_q$ be local foliated coordinates near the chosen point
    (with, let us say, all variables between $-1$ and 1 and the point at the origin).
    We will write $x = (x_1, \ldots, x_{m-q})$ and $y = (y_1, \ldots, y_q)$.
    Write $X = (-1,1)^{p}$ and $Y = (-1,1)^q$ so that points in the charts are parameterized by
    $(x,y) \in X\times Y$.

    Let $\nu$ be the push-forward of $\mu$ along the projection $\pi : X\times Y \to Y$.
    In other words, let $\nu$ be the measure on $Y$ given by
    \[ \nu(S) = \mu(\pi^{-1}(S)) .\]
    By a standard result in measure theory \cite[theorem 1, page 58]{bourbaki},
    the measure $\mu$ can be ``disintegrated'', and there is a family of measures
    $\mu_y$ on $X$, indexed by $y\in Y$
    such that
    \[ \int_{\text{chart}} g \,d\mu = \int_Y \left( \int_X g(x,y) \,d\mu_y(x) \right) \,d\nu(y) \]
    for every $\mu$-integrable function $g$ on the chart.
    
    For $y\in Y$, define the elements $c_y \in \Cur_p(X)$ by
    \[ c_y : \Omega^p_c(X) \to \R : \alpha \mapsto \int_X \alpha(s(x,y)) \,d\mu_y(x) .\]
    We will show that $c_y$ is closed for $\nu$-almost every $y\in Y$,
    using essentially the argument from \cite[lemma 10.2.17]{candelconlon}.
    Let $\beta \in \Omega^{p-1}_c(X)$ and define
    \[ g_{\beta} : Y \to \R : y \mapsto (\partial c_y)(\beta) = c_y(d\beta) .\]
    For any $h \in C^{\infty}(Y)$,
    we have
    \[ c_y(dh \wedge \beta) = \int_X (dh\wedge\beta)(s(x,y)) \,d\mu_y(x) = 0 \]
    because $s$ takes values in $\Lambda^pT\F$ and $dh$ annihilates $T\F$
    (we have used $h$ also as the name for the function $\pi^*h \in C^{\infty}(X\times Y)$).
    Therefore
    \[ c_y(h d\beta) = c_y(dh\wedge\beta + h d\beta) = c_y(d(h\beta)) .\]
    We now integrate $h$ against the measure $g_{\beta}\nu$.
    This results in
    \begin{align*}
        \int_Y h \,d(g_\beta\nu) 
        &= \int_Y h g_\beta \,d\nu \\
        &= \int_Y h c_y(d\beta) \,d\nu \\
        &= \int_Y c_y(d(h\beta)) \,d\nu \\
        &= \int_Y \int_X d(h\beta)(s(x,y)) \,d\mu_y(x) \,d\nu \\
        &= \int_{\text{chart}} d(h\beta)(s(x,y)) \,d\mu \\
        &= c(d(h\beta)) \\
        &= 0
    \end{align*}
    because $c$ is closed.
    This shows that integrating any smooth function against the measure $g_\beta\nu$ gives 0,
    proving that $g_\beta\nu$ is the zero measure.
    Therefore, $(\partial c_y)(\beta) = 0$ for $\nu$-almost all $y\in Y$.
    Because the topological space $\Omega^{p-1}_c(X)$ is separable,
    we can apply this argument for each $\beta$ in a countable dense subset of $\Omega^{p-1}_c(X)$,
    and conclude that for $\nu$-almost all $y\in Y$, the current $c_y$ vanishes
    on all these $\beta$'s simultaneously (and is thus closed).
    This shows that $c_y$ is closed for $\nu$-almost every $y\in Y$.

    It is now easy to show that the current $\partial(fc)$ vanishes at our chosen point.
    Let $\beta \in \Omega^{p-1}_c(X\times Y)$. It suffices to show that $(fc)(d\beta)$
    vanishes.
    But
    \begin{align*}
        (fc)(d\beta)
        &= \int_{\text{chart}} (f d\beta)(s) \,d\mu \\
        &= \int_{Y} \int_{X} (f d\beta)(s) \,d\mu_y(x) \,d\nu \\
        &= \int_{Y} \left( f \int_{X} (d\beta)(s) \,d\mu_y(x) \right) \,d\nu \\
        &= 0
    \end{align*}
    because the inner integral vanishes for $\nu$-almost every $y$.
    This proves the result.
\end{proof}

\begin{definition}
    We say that a foliation $\F$ on $M$ is \emph{pleasant}
    if every closed saturated subset $S \subset M$ that contains no pair of
    estranged leaves, is contained in the closure of some leaf.
\end{definition}

The following are examples of pleasant foliations, showing that plenty of them exist:

\begin{itemize}
    \item simple foliations (i.e.\ those obtained from submersions),
    \item foliations with only compact leaves,
    \item the Reeb foliation,
    \item foliations that have a dense leaf,
    \item the foliation obtained by gluing together (at the boundary) two
        pleasant foliations of manifolds with boundary (with the foliation
        tangent to the boundary components),
    \item the turbulization of a pleasant foliation.
\end{itemize}

Recall that $H_C$ is the set of classes that can be represented by a $C$-cycle.

\begin{proposition}
    Suppose that $\F$ is a pleasant foliation of a manifold $M$.
    Let $C \subset \Lambda^pT\F$ be a compact cone structure.
    Then we have
    \[ H_C = \cl{\ConvexHull\left( \bigcup_{\text{leaf $L$}} H_{C{\restriction_L}} \right)} ,\]
    where for each leaf $L$, the cone $C{\restriction_L}$ is the intersection of $C$
    with $\Lambda^pT\cl{L}$.
    \label{prop:pleasant-leafwise}
\end{proposition}

\begin{proof}
    By the Krein-Milman theorem (theorem \ref{thm:krein-milman}),
    the cone $\ZZ_C$ is the closure of the convex hull of its extreme rays.
    Therefore the image $H_C$ of $\ZZ_C$ under the projection $\ZZ_p \to H_p$
    is the convex hull of the images of these extreme rays.
    Suppose now that $c \in \ZZ_C$ is extreme.
    Because $\F$ is pleasant, this means that the support of $c$ is contained
    in the closure of a single leaf $L$.
    But then $c$ is also a $(C{\restriction_{\cl{L}}})$-cycle, and so
    $[c] \in H_{C{\restriction_{\cl{L}}}} \subset H_p$.
    This proves the result.
\end{proof}


\section{Regularization of currents}

Recall that a smooth $p$-current on an oriented manifold is a current of the form
\[ \alpha \mapsto \int_M \beta\wedge\alpha \]
for a fixed $\beta\in\Omega^q(M)$ ($p+q=m$).
Because smooth currents are easier to handle than other currents, it is convenient
to be able to approximate currents by smooth currents.
In \cite{derham}, de Rham proved that such approximation is always possible,
and that every closed current is homologous to a smooth one.

\subsection*{De Rham's regularization}

For the reader's convenience, we give here a variation of the statement from
\cite[theorem 12, page 80]{derham}. The original statement is stronger and a little
more general, but this version will do for our purposes.

\begin{theorem}[de Rham]
    Let $M$ be a smooth compact manifold.
    Then one can construct linear operators $R : \Cur_p \to \Cur_p$ and $A : \Cur_p \to \Cur_{p+1}$, depending on a positive parameters $\eps$, with the following properties:
    \begin{enumerate}
        \item For any $p$-current, $R(c)$ is smooth and
            \[ R(c) - c = \partial A(c) + A(\partial c) .\]
        \item For any neighborhood of the support $c$,
            one can find a sufficiently small $\eps > 0$ such that the supports
            of $R(c)$ and $A(c)$ are contained in this neighborhood.
    \end{enumerate}
\end{theorem}

\begin{corollary}[de Rham]
    Every closed current on a smooth compact manifold is homologous to a smooth current.
\end{corollary}

Let us outline the construction of the regularization operator $R$ as it is done by de Rham.
We refer the reader to de Rham's book \cite{derham} for the details.
\begin{itemize}
    \item Suppose $c$ is a compactly supported $p$-current on $\R^m$.
        Let $\rho$ be a smooth bump function supported in the ball of radius $\eps$ at
        the origin of $\R^m$, of integral 1.
        Let $t_{y}$ be the translation $\R^m\to\R^m:x\mapsto x+y$.
        Then the convolution of $c$ with $\rho$,
        \[ \alpha \mapsto c\left( \int_y \rho(y) t_y^*\alpha \right) , \]
        is a smooth $p$-current on $\R^m$.
    \item Let $B^m\subset \R^m$ be the unit ball.
        There is is a suitable diffeomorphism $h : \R^m \to B^m$, such that the
        maps $h t_y h^{-1}$ extend
        to smooth maps $\varphi_y : \R^m\to\R^m$ that are the identity outside $B^m$.
        Then if $c$ is a compactly supported $p$-current,
        \[ \tag{$\ast$} R_{B^m,\eps}(c) : \alpha \mapsto c\left( \int_y \rho(y) \varphi_y^*\alpha \right) \]
        is a compactly supported $p$-current. It is smooth in $B^m$. It equals
        $c$ outside of $\cl{B^m}$.

        Using $R_{B^m,\eps}$ we can smoothen a current locally on a manifold:
        suppose $B\subset M$ is the unit ball in a chart for $M$.
        Let $f \in C^\infty(M,[0,1])$ be supported in the chart, and $f=1$ on $B$.
        Then we have an operator
        \[ R_{B,\eps} : \Cur_p(M) \to \Cur_p(M) : c \mapsto (1-f)c + R_{B^m,\eps}(fc) .\]
        This operator gives currents that are smooth in $B$.
    \item Given our manifold $M$, we cover it using balls $B_1, ..., B_k$.
        Then the smoothing operator $R$ is given by
        \[ R = R_{B_k, \eps_k} \circ ... \circ R_{B_1, \eps_1} .\]
\end{itemize}

We do not describe the construction of $A$, since the details are not important to us.

\subsection*{Regularization of structure currents}

Suppose that $C$ is a compact cone structure on $M$.
It is not generally true that every $C$-cycle is homologous to a smooth
$C$-cycle.  However, the following result shows that $C$-cycles are
homologous to smooth currents that are very close to $C$-currents.

\begin{proposition}[see {{\cite[{{Theorem I.9}}]{sullivan}}}]
    Let $M$ be a smooth compact manifold.
    Suppose that $C \subset \Lambda^pTM$ is a compact cone structure.
    Let $D$ be another compact cone structure, such that $C \subset \interior(D)$.
    Then by choosing the parameters $\eps_i$ for $R$ sufficiently small,
    $R(c)$ will be a smooth $D$-current for every $C$-current $c$.
    \label{prop:regularization-C-currents}
\end{proposition}

This proposition follows from the following local result.

\begin{proposition}
    Let $C \subset S\Lambda^pT\R^m$ be a compact cone structure,
    and let $D$ be a cone structure whose interior contains $C$.
    Let $\eps > 0$ be so small that $(\varphi_y)_*C \subset D$ for all $y$ with $\lvert y\rvert \leq \eps$.
    Then for any compactly supported $C$-current on $\R^m$,
    the current $R_{B,\eps}(c)$ is smooth in $B$, and over $B$ it is a $D$-current.
    \label{prop:local-regularization-C-currents}
\end{proposition}

\begin{proof}[Proof of proposition \ref{prop:regularization-C-currents} from proposition \ref{prop:local-regularization-C-currents}]
    Cover $M$ with balls $B_1, ..., B_k$. Choose compact sets $C \subset C_1 \subset C_2 \subset ... \subset C_k = D$ such that
    $C_1$ is a neighborhood of $C$ and $C_{i+1}$ is a neighborhood of $C_i$.
    By taking $\eps_1$ small enough, we can ensure that $R_{B_1,\eps_1}$ maps
    $C$-currents to $C_1$-currents.
    By taking $\eps_2$ small enough, we can ensure that $R_{B_2, \eps_2}$ maps $C_1$-currents
    to $C_2$-currents.
    We continue this way, finally picking $\eps_k$ small enough such that $R_{B_k,\eps_k}$ maps
    $C_{k-1}$-currents to $D$-currents.
    By construction, this means that $R$ will map $C$-currents to $D$-currents.
    This proves the result.
\end{proof}

\begin{proof}[Proof of proposition \ref{prop:local-regularization-C-currents}]
    The crucial observation is that $(\varphi_y)_*C_y \subset D_y$ guarantees
    that $R_{B^m,\eps}(c)$,
    given by formula ($\ast$),
    is a $D$-current.
\end{proof}

\section{The local structure of structure cycles}

In light of theorems \ref{thm:sullivan-hahn-banach} and \ref{thm:sullivan-cone-C},
it is interesting to better understand structure cycles.
For foliation cycles, Sullivan's result (\cref{thm:transverse-measures})
completely describes what they look like.
For other structures, no such result is available.
There are, however, conjectures.
We state two conjectures, and generalize them to a single conjecture about structure cycles.

\subsection*{Unifying existing conjectures}

In \cite{sullivan}, Sullivan discusses the case of a compact complex manifold $M^{2n}$.
For each $k\in\{0,...,n\}$, there is a natural cone structure in dimension $p=2k$.
At a point $x\in M$, it is the compact convex cone in $\Lambda^pT_xM$ generated by
complex subspaces of complex dimension $k$.
The corresponding structure cycles are referred to as \emph{complex cycles}.

\begin{conjecture}[Sullivan, {{\cite[III.9]{sullivan}}}]
    Complex cycles can be locally represented near $x$ in $M$ as
    $c = \int c'\,d\nu$, where the $c'$ are irreducible subvarieties passing near $x$
    and $\nu$ is a non-negative measure on the space of such (regarded as a subspace of
    the currents on some neighborhood of $x$).
    \label{conj:sullivan}
\end{conjecture}

An analogous conjecture for the case of a leafwise (almost) complex structure was made
by \'Etienne Ghys, at the \emph{Workshop on topological aspects of symplectic foliations}
that took place at Universit\'e de Lyon 1 (France) in September 2017.
Ghys' conjecture was not stated in precise mathematical terms.
The following formulation is an attempt at making the conjecture more precise,
by phrasing it similarly to Sullivan's conjecture.
Any mistakes in formulating this conjecture are solely the author's.

\begin{conjecture}[Ghys]
    Suppose $\F$ is a foliation of a manifold $M$, with a leafwise (almost)
    complex structure $J$. Let $C$ be the structure cone associated to $J$:
    \[ C = \{ v\wedge J(v) \mid v\in T\F \} .\]
    Then every $C$-cycle can be locally represented near $x\in M$
    as $c = \int c' \,d\nu$, where the $c'$ are pseudoholomorphic curves passing near $x$
    and $\nu$ is a non-negative measure on the space of such (regarded as a subspace of the
    currents on some neighborhood of $x$).
    \label{conj:ghys}
\end{conjecture}

The conjectures above have a natural variation for general structure cycles.
We need the following definition.

\begin{definition}
    Suppose that $C \subset \Lambda^pTM$ is a compact cone structure on $M$.
    An embedded $p$-dimensional submanifold $S \subset M$ is said to be
    \emph{tangent to $C$} if for each point $x\in S$, we have $\Lambda^pT_xS \subset C$.
\end{definition}

One straightforward variation for more general cone structures is the following.

\begin{conjecture*}[first version]
    Suppose that $C \subset \Lambda^pTM$ is a compact cone structure on $M$.
    Then every $C$-cycle can be locally represented near $x\in M$
    as $c = \int c' \,d\nu$, where the $c'$ are submanifolds of $M$, tangent to $C$,
    passing near $x$, and $\nu$ is a non-negative measure on the space of such (regarded as
    a subspace of the currents on some neighborhood of $x$).
    \label{conj:version1}
\end{conjecture*}

As stated above, the conjecture is false.
Here is a counterexample.
Let $M = \R^4$ with coordinates $(x_1,y_1,x_2,y_2)$, and let $C$ be the cone structure with
\[ C_x = \spn \left\{ \dd{x_1}\wedge\dd{y_1} + \dd{x_2}\wedge\dd{y_2} \right\} .\]
There is no submanifold that is tangent to $C$.
Nevertheless, there are $C$-cycles, like
\[ \Omega^2_c(M) \to \R : \alpha \mapsto \int_{\R^4} \alpha\left( \dd{x_1}, \dd{y_1}\right) + \alpha\left(\dd{x_2}, \dd{y_2} \right) .\]
This example shows that, if we want the conjecture to have any chance of being true,
we need to require the existence of ``sufficiently many'' submanifolds
tangent to $C$ near any point $x\in M$.

In the example above, there were no submanifolds through $x$ tangent to $C$.
Indeed, if $S^p$ is a submanifold of $M$ through $x$, then $\Lambda^pT_xS$
only contains pure tensors, but $C_x$ contains no non-zero pure tensors.
(By a \emph{pure tensor}, we means a decomposable tensor, i.e.\ a tensor that can be
written as a wedge of 1-vectors.)
For any $x\in M$, let $C_x^s \subset C_x$ be the set of all $v\in C_x$ such that
there is a submanifold $S^p$ through $x$, tangent to $C$, such that $\Lambda^p
T_xS = \spn\{v\}$.
In other words, $C_x^s \subset C_x$ is the set of all directions at $x$ that
can be realized by a submanifold tangent to $C$.
The only elements of $C_x$ that can be realized by a sum of (positive multiples of)
submanifolds tangent to $C$ are those in the convex hull of $C_x^s$.
We amend our conjecture as follows.

\begin{conjecture*}[second version]
    Suppose that $C \subset \Lambda^pTM$ is a compact cone structure on $M$.
    Assume that for each $x\in M$, we have
    \[ \ConvexHull(C_x^s) = C_x .\]
    Then every $C$-cycle can be locally represented near $x\in M$
    as $c = \int c' \,d\nu$, where the $c'$ are submanifolds of $M$, tangent to $C$,
    passing near $x$, and $\nu$ is a non-negative measure on the space of such (regarded as
    a subspace of the currents on some neighborhood of $x$).
    \label{conj:version2}
\end{conjecture*}

Note that this extra condition is satisfied in Sullivan's conjecture \ref{conj:sullivan}
because the cone $C$ is \emph{defined} as the convex hull of $C_x^s$ (which is generated
by complex subspaces).
Now even if $\ConvexHull(C_x^s) = C_x$ for all $x\in M$, it is not clear that
there are sufficiently many local submanifolds tangent to $C$: indeed, it is possible
that the cones $C_x$ vary from point to point in such a way that no submanifold through
$x$ can be found tangent to $C$, even though there are many pure tensors in $C$.
We make the following definition.

\begin{definition}
    Let $C \subset \Lambda^pTM$ be a compact cone structure.
    We say that $C$ is \emph{integrable} if for every point $x\in M$ there
    is a chart around $x$ in which $C$ becomes constant
    (meaning that, in this chart, $C$ is invariant under translations).
\end{definition}

\begin{example}
    Suppose that $D$ is a $p$-dimensional oriented distribution on $M$.
    Then $\Lambda^pD$ is a line bundle over $M$, and we can take $C_x$ to be the
    ray in $\Lambda^pD$ corresponding to the chosen orientation.
    This results in a compact cone $C$.
    The cone structure is integrable precisely if $D$ is integrable,
    in which case this example reduces to that of foliation currents.
\end{example}

In the case of integrable $C$, the set $C_x^s$ is precisely the intersection of $C_x$
with the decomposable tensors, because in a chart that makes $C$ translation-invariant,
affine $p$-dimensional subspaces are suitable submanifolds.
Therefore, the condition $\ConvexHull(C_x^s) = C_x$ is also satisfied in Ghys' conjecture
\ref{conj:ghys} if we require integrability of $C$, because the cone $C_x$ is defined as
the convex hull of a set of pure tensors.
(For example, in Ghys' conjecture, one could ask that $J$ be integrable.)
The following conjecture, weaker than the second version, is our final version.

\begin{conjecture}[final version]
    Suppose that $C \subset \Lambda^pTM$ is an integrable compact cone structure on $M$.
    Assume that for each $x\in M$, we have
    $\ConvexHull(C_x^s) = C_x$.
    Then every $C$-cycle can be locally represented near $x\in M$
    as $c = \int c' \,d\nu$, where the $c'$ are submanifolds of $M$, tangent to $C$,
    passing near $x$, and $\nu$ is a non-negative measure on the space of such (regarded as
    a subspace of the currents on some neighborhood of $x$).
    \label{conj:final}
\end{conjecture}

\begin{example}[Contact structure]
    Suppose $D = \ker(\alpha)$ is a contact structure on $M^3$ given by
    the contact form $\alpha$.
    The associated cone $C$ is not integrable.
    For every $x\in M$, we do have $\ConvexHull(C_x^s) = C_x$.
    Note that there are \emph{no} submanifolds of $M$ tangent to $C$.
    Nevertheless, the conjecture happens to hold for this particular case, because there
    also don't exist any non-zero $C$-cycles. Indeed, there is an exact form
    that is positive on $D$, namely $d\alpha$, so by \cref{thm:sullivan-hahn-banach},
    there are no non-zero $C$-cycles.
\end{example}

\subsection*{The case of smooth currents}

Let us now focus on \emph{smooth} currents.
Write $q = m - p$ (recall that $m$ is the dimension of $M$, and that
$C \subset \Lambda^pTM$).
The following lemma says that smooth cycles are locally sums of foliation cycles.

\begin{lemma}
Suppose that $\beta$ is a closed $q$-form on $M$, regarded as a closed $p$-current.
Let $x\in M$.
Near $x$, we have
\[ \beta = \sum_{i=1}^k c_i ,\]
where the $c_i$ are foliation cycles of foliations of open subsets of $M$.
\end{lemma}

\begin{proof}
    Locally, $\beta = d\gamma$. In local coordinates, $\gamma$ is given by
    \[ \sum_{i=1}^k f^i dx_{j_1^i}\wedge...\wedge dx_{j_{q-1}^i} .\]
    Let $c_i$ be the differential of the $i$'th term, meaning
    \[ c_i = df^i \wedge dx_{j_1^i} \wedge ... \wedge dx_{j^i_{q-1}} .\]
    Then $\beta = \sum_i c_i$.
    But each $c_i$ is a foliation cycle of a foliation on an open subset of $M$:
    it is a foliation cycle for the foliation
    \[ \ker(df^i) \cap \ker(dx_{j^i_1}) \cap ... \cap \ker(dx_{j^i_{q-1}}) \]
    on the set $\{ x\in M \mid c_i(x)\neq 0\}$.
\end{proof}

One would like to be able to relate the foliations that appear to the cone structure $C$.
Ideally, the foliations are tangent to the cone structure.

\begin{definition}
    Suppose that $C \subset \Lambda^pTM$ is a compact cone structure on $M$.
    A $p$-dimensional foliation of $M$ is said to be
    \emph{tangent to $C$} if each of its leaves is tangent to $C$.
\end{definition}

The lemma above says that every smooth cycle is locally the sum of foliation cycles.
This raises the following question: is every smooth $C$-cycle locally the sum
of foliation cycles tangent to $C$?
For example, if $q=1$, the answer is ``yes'', because a 1-form is already decomposed.
If $q=2$ and $m\leq 3$, the answer is also ``yes'', because a 2-form on a manifold of dimension at most 3
is always decomposable.

%% file: chapters/llg.tex
\chapter{Local Lie groupoids}
\label{chapter:llg}

In this chapter, we discuss the basics of local Lie groupoids.
First, we define them.
Next, we discuss their associativity properties, which will be important
in all that follows.
In the third section, we go over regularity and connectedness assumptions for local
Lie groupoids.
Finally, we explain how to associate a Lie algebroid to every local Lie groupoid,
in the same way as for global Lie groupoids.

\section{Definitions}

We start by defining local Lie groupoids. Instead of using the symbol $\G$ as we did for
(global) Lie groupoids, we will now use $G$ to indicate that this is only a \emph{local}
Lie groupoid.
There are several different definitions of local Lie groupoids.
We will adopt the following.

\begin{definition}
    A \emph{local Lie groupoid} $G$ over a manifold $M$ is a manifold $G$,
    together with maps
    \begin{itemize}
        \item $s, t : G \to M$ submersions (the \emph{source} and \emph{target} maps),
        \item $u : M \to G$ a smooth map (the \emph{unit} map),
        \item $m : \U \to G$ a submersion (the \emph{multiplication}), where
            $\U\subset  G\timesst G$ is an open neighborhood of
            \[ (G\timesst M) \cup (M\timesst G)
                = \bigcup_{g\in G} \{ (g,u(s(g))), (u(t(g)), g) \} ,\]
        \item $i : \V \to \V$ a smooth map (the \emph{inversion}), where $\V \subset  G$ is an open neighborhood of $u(M)$
            such that $\V\timesst \V \subset \U$,
    \end{itemize}
    such that the following axioms hold:
    \begin{itemize}
        \item $s(m(g,h)) = s(h)$ and $t(m(g,h)) = t(g)$ for all $(g,h) \in \U$,
        \item $m(m(g,h),k) = m(g,m(h,k))$ in a neighborhood of
        \[ \tag{$\dag$} (M\timesst M\timesst G) \cup (M\timesst G\timesst M) \cup (G\timesst M\timesst M)\]
         in $G\timesst G\timesst G$,
        \item $m(g,u(s(g))) = m(u(t(g)),g) = g$ for all $g\in G$,
        \item $s(i(g)) = t(g)$ and $t(i(g)) = s(g)$ for all $g\in\V$,
        \item $m(i(g),g) = u(s(g))$ and $m(g,i(g)) = u(t(g))$ for all $g\in \V$.
    \end{itemize}\label{def:llg}
\end{definition}

A local Lie groupoid where $M$ is a singleton is a \emph{local Lie group}.

\begin{remark}
It follows from the definition that $u : M \to G$ is an embedding,
and we will consider $M$ as a submanifold of $G$ using this embedding.
In particular, if $x\in M$, we will write the unit at $x$ as just $x$, leaving
out the $u$.
For the multiplication we usually write $gh$ instead of $m(g,h)$.
For the inversion, we write $g^{-1}$ instead of $i(g)$.
\end{remark}

\begin{remark}
Our definition of local Lie groupoids is just one of many possible versions.
Here are some variations.
\begin{itemize}
\item One could merely require that $m$ be defined in a neighborhood of $M\timesst M \subset G\timesst G$. This makes it possible to have elements $g\in G$ that cannot be multiplied by
\emph{anything}, effectively excluding them completely from the groupoid structure.
In our definition, for every $g\in G$ the products $g s(g)$ and $t(g)g$ are defined (and they
equal $g$).
\item It follows from our definition that $i : \V \to \V$ is an involutive diffeomorphism.
An element $g\in \V$ will be called \emph{invertible}.
One could weaken the definition by asking for an inversion map $i : \V \to G$ instead of $i : \V \to \V$. It is then possible to have an invertible element $g\in \V$ such that $g^{-1}$ is not invertible (i.e.\ is not in $\V$).
This change does not change the theory of local Lie groupoids, because it is possible to
recover an inversion $i : \V \to \V$ by restricting the inversion map.
\item Some authors require $\V = G$, so that every groupoid element is invertible.
We do not need this stronger requirement.
We will, however, later require that our local Lie groupoid be strongly connected.
As we will see later, a consequence of strong connectedness is that every
element is some product of invertible elements.
\item Instead of requiring associativity in a neighborhood of ($\dag$), one
could ask that $(gh)k = g(hk)$ whenever both sides are defined.
This is a stronger requirement. However, we will see later
that this does not change the theory of local Lie groupoids because this stronger requirement
is always satisfied after restricting $m$ and $i$ sufficiently.
\end{itemize}
\end{remark}

\begin{remark}
We do not assume that the manifold $G$ is Hausdorff.
All other manifolds, including $M$ and all source and target fibers, are assumed to
be Hausdorff.
\end{remark}

There are two different ways of obtaining another local Lie groupoid $G'$
from a given local Lie groupoid $G$. Both of them are relevant for us.
\begin{itemize}
    \item We say that $G'$ is obtained by \emph{restricting} $G$, if both local
        groupoids have the same manifolds of arrows and objects, the same
        source and target maps, and the multiplication and inversion in $G'$
        are obtained by restricting the ones of $G$ to smaller domains.
    \item We say that $G'$ is obtained by \emph{shrinking} $G$ if $G'$ is an
        open neighborhood of $M$ in $G$, the source and target maps are the
        restrictions of $s$ and $t$ to $G'$, multiplication is the restriction
        of $m$ to ${G'}^{(2)}=\U \cap (G' \timesst G')\cap m^{-1}(G')$, and
        inversion is the restriction of $i$ to $\V'=(\V \cap G') \cap i(\V \cap
        G')$.
\end{itemize}

When restricting a local Lie groupoid $G$, its space of arrows does not change.
When shrinking a local Lie groupoid $G$, its space of arrows becomes smaller.
It is important to clearly distinguish these two operations.

\begin{example}[Restriction of a Lie groupoid]
Let $\G\tto M$ be a Lie groupoid. Any open $\U\subset \G^{(2)}$ containing
$(\G\timesst M) \cup (M\timesst \G)$ determines a restriction $G$ of $\G$. On
the other hand, any open neighborhood $G'\subset \G$ of the identity manifold
$M$ determines a local Lie groupoid $G'$ shrinking $\G$. 
\end{example}

\begin{example}[$\R \cup \{\infty\}$]
Consider $G = \R \cup\{\infty\}$, equipped with addition
(where $r+\infty=\infty+r=\infty$ for all $r\in\R$).
This is a local Lie group.
The domain of multiplication is $G\times G \setminus \{(\infty,\infty)\}$.
The domain of inversion is $\R$.
Note that it is not possible to extend the multiplication to include $\infty +
\infty = \infty$ because this would make the map $m$ discontinuous.

More generally, one can turn the one-point compactification of a non-compact connected
Lie group into a local Lie group by setting $g\cdot\infty=\infty\cdot g=\infty$.
\end{example}

Morphisms between local Lie groupoids are defined in a more or less obvious way.
\begin{definition}
    Suppose that $G_1$ and $G_2$ are local Lie groupoids over $M_1$ and $M_2$,
    respectively.
    A \emph{morphism of local Lie groupoids} is a pair $(F,f)$, where
    $F : G_1\to G_2$ and $f:M_1\to M_2$ are smooth maps
    such that
    \begin{itemize}
        \item $F \circ u_1 = u_2 \circ f$,
        \item $f \circ s_1 = s_2 \circ F$ and $f\circ t_1 = t_2\circ F$,
        \item $(F\times F)(\U_1) \subset \U_2$ and $F(m_1(g,h)) = m_2(F(g),F(h))$ whenever $(g,h) \in \U_1$,
        \item $F(\V_1) \subset \V_2$ and $F \circ i_1 = i_2\circ F$.
    \end{itemize}
\end{definition}

\section{Degrees of associativity}
\label{sec:assoc}

In our definition of a local Lie groupoid, we required that $(gh)k=g(hk)$ for all
$(g,h,k)$ in some neighborhood of
\[ (M\timesst M\timesst G) \cup (M\timesst G\timesst M) \cup (G\timesst M\timesst M) \]
in $G\timesst G\timesst G$.
This condition is called \emph{local associativity}, and it is the weakest
associativity condition we will use.
There are other notions of associativity for local Lie groupoids, and they are (perhaps
surprisingly) not equivalent.
We give an overview of all the notions we will need.

\begin{definition}
    A local Lie groupoid is called \emph{3-associative} if
    \[ (gh)k=g(hk) \]
    whenever both sides of the equation are defined (in other words, whenever
    $(g,h)$,$(h,k)$,$(gh,k)$ and $(g,hk)$ are all in the domain of multiplication $\U$).
\end{definition}

This definition is perhaps the most natural choice for an associativity requirement.
Not every local Lie groupoid is 3-associative. We will give an example later in this
section.
A stronger notion than that of 3-associativity is the following.
The following definition is taken from \cite{olver}, with group replaced by groupoid.
\begin{definition}
    A local Lie groupoid is called \emph{associative to order $n$} (or
    \emph{$n$-associative})
    if for every $3 \leq m\leq n$
    and every ordered $m$-tuple of groupoid elements $(x_1,...,x_m)\in G^m$,
    all corresponding
    defined $m$-fold products are equal (in other words, if an $m$-fold product can be
    evaluated by parenthesizing in multiple ways, all result in the same answer).
\end{definition}

For every $n\geq 3$ there are $n$-associative local Lie groupoids that are not
$(n+1)$-associative (this is already true in the case of local Lie groups,
\cite[section 3]{olver}).
The following notion is the strongest version of associativity we will use.

\begin{definition}
    A local Lie groupoid is called \emph{globally associative} if it is associative to every
    order $n\geq 3$.
\end{definition}

\begin{remark}
    At this point, one can ask why 3-associativity does not imply
    $n$-associativity for higher $n$, the way it does for Lie groupoids.
    Let us illustrate why 3-associativity does not imply 4-associativity for local groups.
    There are five different ways of parenthesizing a 4-fold product, as shown
    in this diagram:

    \begin{center}
    \begin{tikzpicture}
        \node[] at (0,0) (A) {$(ab)(cd)$};
        \node[] at (2,1) (B) {$a(b(cd))$};
        \node[] at (-2,1) (E) {$((ab)c)d$};
        \node[] at (-1.2,2.2) (D) {$(a(bc))d$};
        \node[] at (1.2,2.2) (C) {$a((bc)d)$};
        \draw[-] (A) -- (B);
        \draw[-] (B) -- (C);
        \draw[-] (C) -- (D);
        \draw[-] (D) -- (E);
        \draw[-] (E) -- (A);
    \end{tikzpicture}
    \end{center}

    In the diagram, there is an edge between two parenthesizations if their equality
    follows from 3-associativity, and in a (global) groupoid, this proves 4-associativity.
    However, in a local groupoid, it is possible that some of these 4-fold products are not
    defined. The graph above can then become disconnected, like so:

    \begin{center}
    \begin{tikzpicture}
        \node[] at (0,0) (A) {$(ab)(cd)$};
        \node[draw,cross out] at (2,1) (B) {$a(b(cd))$};
        \node[] at (-2,1) (E) {$((ab)c)d$};
        \node[draw,cross out] at (-1.2,2.2) (D) {$(a(bc))d$};
        \node[] at (1.2,2.2) (C) {$a((bc)d)$};
        \draw[-] (A) -- (B);
        \draw[-] (B) -- (C);
        \draw[-] (C) -- (D);
        \draw[-] (D) -- (E);
        \draw[-] (E) -- (A);
    \end{tikzpicture}
    \end{center}

    In this case, the product may be different on different components of the graph
    (so that $a((bc)d) \neq ((ab)c)d = (ab)(cd)$ in this example).
\end{remark}

Let us give an example of a locally associative local Lie group that is not
3-associative. The example is a slightly modified version of an example given
in \cite{olver}.
Let $X$ be the complex plane $\C$, with a small ball centered at $1$ taken out.
Let $G$ be the universal cover of $X$.
Let $\tilde{0}\in G$ be a preimage of $0\in X$ under the projection.
This point $\tilde{0}$ will be the neutral element in our local Lie group.
Let $B \subset X$ be the open ball of radius $\frac34$ centered at $0$,
and let $\tilde{B}$ be the ball above $B$ containing $\tilde{0}$.

The multiplication will be defined on a subset of $(G\times\tilde{B}) \cup (\tilde{B}\times G)$.
Suppose that $(\tilde{g},\tilde{h}) \in G\times\tilde{B}$, and write $g$ and $h$ for the
projections to $M$. The product $\tilde{g}\tilde{h}$ will be defined
if the line segment (in $\C$) from $g$ to $g+h$ lies completely in $X$.
In that case, the product $\tilde{g}\tilde{h}$
is obtained as the endpoint of the lift of this line segment,
starting at $\tilde{g}$.
Similarly, for $(\tilde{g},\tilde{h}) \in \tilde{B}\times\G$, the product is defined as
the endpoint of the lift of the line segment from $h$ to $g+h$,
starting at $\tilde{h}$ (if that segment lies completely in $X$).
Inverses are defined on $\tilde{B}$, in the obvious way.

This local Lie group is locally associative.
However, it is not 3-associative.
By picking $a,b,c \in \tilde{B}$ appropriately, one can ensure that $(ab)c\neq a(bc)$,
as shown in figure \ref{fig:lallg}.

\begin{figure}
    \centering
    \begin{overpic}[width=0.6\linewidth]{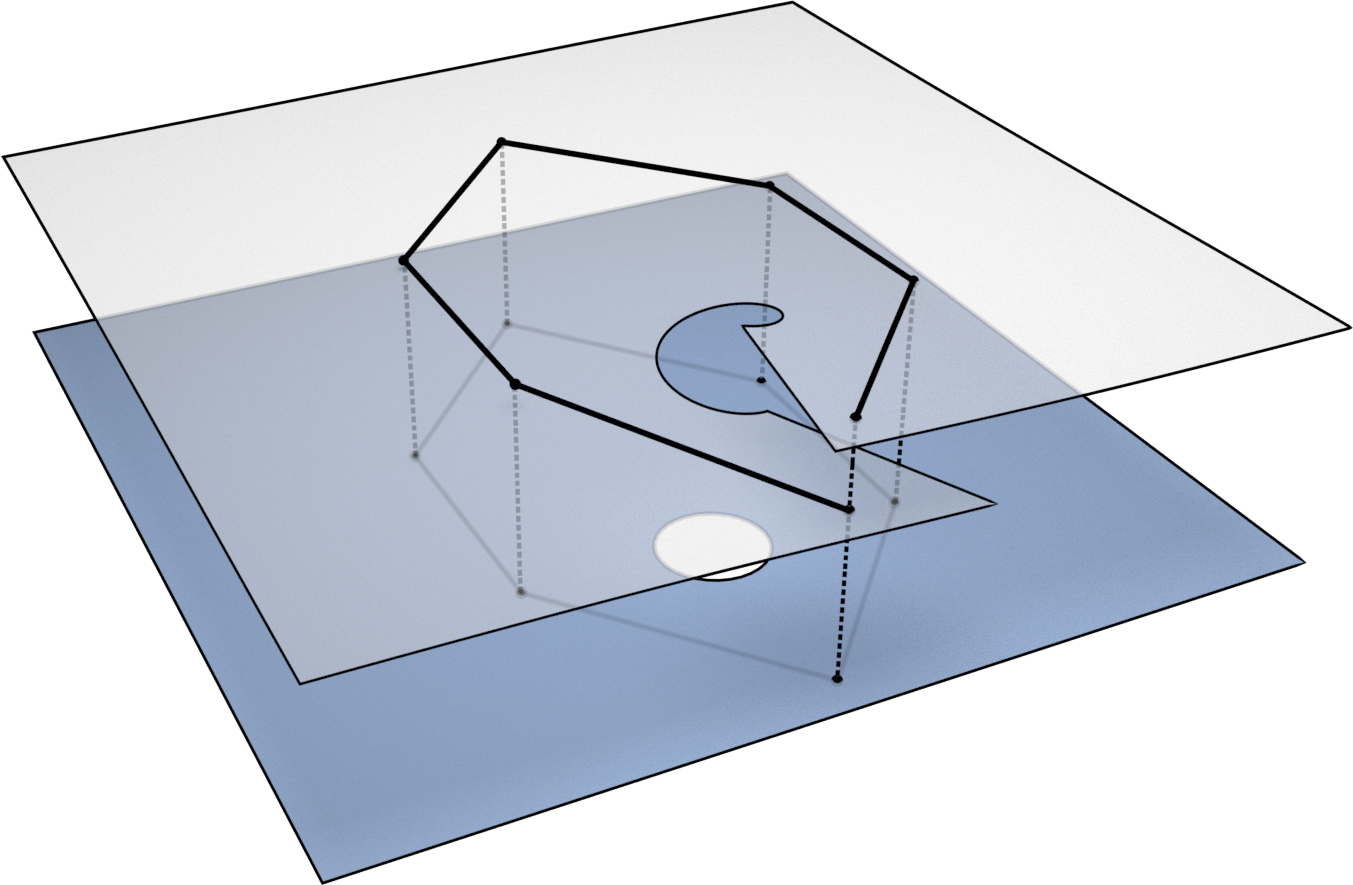}
        \put (37,56.3) {$\tilde{0}$}
    \end{overpic}
    \caption{A locally associative local Lie groupoid that is not 3-associative.
    The manifold $X$ is the complex plane with a disk removed (shown at the bottom).
    The local Lie group is the universal cover of $X$ (shown on top).
    Starting at the neutral element $\tilde{0}$, we can evaluate a triple
    product in two different ways, ending up over the same point of $X$,
    but on different sheets. This local Lie group is not 3-associative,
    but it is locally associative.}
    \label{fig:lallg}
\end{figure}

While local associativity is not the same as $3$-associativity,
the following proposition shows that they are the same if we consider local groupoids
up to restriction.

\begin{proposition}
    Suppose that $G$ is a local Lie groupoid. Then there is a restriction of
    $G$ that is $3$-associative.
\end{proposition}
\begin{proof}
    \emph{Case 1: if $G$ is compact.}
    Let $G_0$ be a neighborhood of $M$ in $G$ that is 3-associative.
    Let $K$ be a neighborhood of $M$ in $G_0$ such that the set
    \[ (G\timesst K\timesst K) \cup (K\timesst G\timesst K) \cup (K\timesst K\timesst G) \]
    is contained in the domain of 3-associativity (this is the set
    $\{(a,b,c)\in G^3\mid (ab)c=a(bc)\}$).
    Let $B\subset K$ be a neighborhood of $M$ in $G$, consisting of invertible elements,
    small enough to ensure that if $x,y\in B$, then $x^{-1}y$ is defined an lies in $K$.
    Now restrict the multiplication to a domain $\U'$ that is contained in
    \[ (G\timesst B) \cup (B\timesst G) .\]
    We claim that this makes the multiplication 3-associative.
    Suppose that $a,b,c\in G$ are such that $a(bc)$ and $(ab)c$ are both defined.
    We will show that $a(bc)=(ab)c$.
    If all three of $a,b,c$ are in $G_0$, this is immediate.
    We may therefore assume that one of $a,b,c$ is not in $G_0$:
    \begin{itemize}
        \item if $a\notin G_0$, we must have $b\in B$ because $ab$ is defined.
            We must also have $bc\in B$ because $a(bc)$ is defined.
            By our choice of $B$, the product $b^{-1}(bc)$ must be defined and lie in $K$.
            But this product equals $c$, so we have $c\in K$.
            This shows that $(a,b,c)\in G\timesst K\timesst K$, and so $(a,b,c)$ lies in
            the domain of 3-associativity.
        \item if $b\notin G_0$, we must have $a\in B\subset K$, because $ab$ is defined.
            We must alo have $c\in B\subset K$ because $bc$ is defined.
            This shows that $(a,b,c)\in K\timesst G\timesst K$, and so $(a,b,c)$ lies in
            the domain of 3-associativity.
        \item the case $c\notin G_0$ is analogous to the first one, by symmetry.
    \end{itemize}

    \emph{Case 2: if $G$ is not compact.}
    Let $G_1\subset G_2\subset ...\subset G_k\subset ... \subset G$ be an exhaustion
    of $G$ by open, precompact local subgroupoids.
    The proof of case 1 shows that for each $k$ there is a set $\U'_k \subset
    G_k\timesst G_k$ where 3-associativity holds. Then
    \[ \U' = \bigcup_{k=1}^\infty \U'_k \]
    gives an appropriate restriction of $G$ to ensure 3-associativity.
\end{proof}

\begin{remark}
    If we are working with local Lie groupoids up to restriction, the result above
    says that we can always assume our local groupoid to be 3-associative.
    We will make use of this fact by assuming 3-associativity whenever we need it,
    keeping in mind that we merely need to restrict our groupoid to ensure this.
    Note that we do not need to shrink the groupoid to obtain 3-associativity.
\end{remark}

\begin{remark}
    In fact, the proof above generalizes to show that every local Lie groupoid
    has a restriction that is $n$-associative for any $n\geq 3$. This generalized
    proof quickly becomes unwieldy, though. We will not use the result for $n>3$.

    In the article \cite{the-article}, a (somewhat) different proof of the result
    for $n=3$ is given. That proof also generalizes, and is perhaps less unwieldy
    for $n>3$.
\end{remark}

\section{Regularity and connectedness}

In this section, we study several basic properties of local Lie groupoids.
Most of the material in this section is an adaptation of material in \cite{olver}
to the groupoid case.

Consider $\R \cup \{\infty\}$ as before.
The point at infinity plays a special role here: left (or right) multiplication by $\infty$
is a constant map.
We want to exclude these special points in our study of local Lie groupoids.
We write $T^sG$ for the kernel of the map $ds : TG \to TM$, and $T^tG$ for the kernel of $dt : TG\to TM$.

\begin{definition}
    Let $G$ be a local Lie groupoid over $M$.
    Let $g\in G$.
    We say that $G$ is \emph{right-regular at $g$} if
    right multiplication by $g$ induces an isomorphism between
    $T^s_{t(g)} G$ and $T^s_g{G}$.
    We say that $G$ is \emph{left-regular at $g$} if
    left multiplication by $g$ induces an isomorphism between
    $T^t_{s(g)} G$ and $T^t_g G$.
    We say that $G$ is \emph{bi-regular at $g$} if
    it is both left- and right-regular at $g$.
    The local groupoid $G$ is called \emph{bi-regular} if it is bi-regular at all $g\in G$.
\end{definition}

In the above example, the groupoid is bi-regular at all points except $\infty$.

\begin{definition}
    An element $g$ of a local Lie groupoid is called \emph{inversional}
    if it can be written as a well-defined product of invertible elements.
    A local Lie groupoid is called \emph{inversional} if all of its elements are inversional.
\end{definition}

For $\R \cup \{\infty\}$ as above, all elements except $\infty$ are inversional.
Every Lie groupoid is obviously inversional.
Restricting a local groupoid (by restricting its multiplication and inversion maps)
can change it from inversional to not inversional.
For example, we can consider the Lie group $\R \times \Z$ as a local Lie groupoid (with
the usual addition). If we restrict its inversion map to a domain $\V'$,
the resulting local Lie groupoid is inversional precisely if
$\{ n\in \Z \mid \R\times\{n\} \cap \V' \neq \emptyset \} \subset \Z$
generates $\Z$.

\begin{definition}
    We say that a local Lie groupoid is \emph{source-connected} (or \emph{$s$-connected}) if
    all of its source fibers are connected.
    We say that a local Lie groupoid is \emph{target-connected} (or \emph{$t$-connected}) if
    all of its target fibers are connected.
\end{definition}

To define connectedness of a local Lie groupoid, one could just ask that $G$ be connected
as a manifold. However, this is not the correct notion for our purposes. In \cite{olver},
connectedness of a local Lie group requires that every neighborhood of the identity generate
the local group (in the same way that in a connected Lie group, every
neighborhood of the identity generates the group). We have the analogous
criterion for groupoids.

\begin{definition}
    Let $U$ be a neighborhood of $M$ in $G$.
    We say that \emph{$U$ generates $G$} if every element of $g$ can be written
    as a well-defined product of elements in $U$.
\end{definition}

The notion of connectedness that we require is the following.
We call it \emph{strong connectedness} to avoid any ambiguity.

\begin{definition}
    A local Lie groupoid $G$ over $M$ is \emph{strongly connected} if
    \begin{enumerate}[(a)]
        \item \label{item:Mconn} $M$ is connected
        \item the domains $\U$ and $\V$ of the multiplication and inversion maps
            are connected,
        \item \label{item:stconn1} the set $\{ (g,h) \in \U \mid s(g) = t(h) = x \}$ is connected for all $x \in M$,
        \item \label{item:stconn2} $G$ is $s$-connected and $t$-connected, and
        \item $G$ is bi-regular.
    \end{enumerate}
\end{definition}

In the case of a local Lie group, condition (\ref{item:Mconn}) is automatic,
and conditions (\ref{item:stconn1}) and (\ref{item:stconn2}) both amount to connectedness of the manifold $G$.

We have the following result. It implies, in particular, that a strongly connected
local Lie groupoid is inversional.

\begin{proposition}
    If $G$ is right-regular and $s$-connected,
    then any neighborhood of $M$ in $G$ generates $G$.
    If $G$ is left-regular and $t$-connected, the same conclusion holds.
\end{proposition}
\begin{proof}
    We will discuss the case where $G$ is right-regular and $s$-connected. The
    other case is analogous.
    Let $U$ be a neighborhood of $M$ in $G$.
    Pick an element $x\in M$. We will show that the set
    \[ A = \{ g\in s^{-1}(x) \mid \text{$g$ can be written as a product of elements in $U$} \} \]
    is both open and closed in $s^{-1}(x)$ (and therefore, by $s$-connectedness, it is the
    entire source fiber).
    This suffices to prove the result.

    By right-regularity of $G$, it is clear that $A$ is open.
    Pick $g\in \partial A$. We will show that $g\in A$.
    There is a neighborhood $V$ of $t(g)$ in $s^{-1}(t(g))$, consisting
    of invertible elements, contained in $U$, such that
    for all $h \in V$ we have
    \[ (h,g) \in \U \quad\text{and}\quad (h^{-1},hg) \in \U .\]
    By right-regularity of $G$, the set $\{ hg \mid h \in V \}$ is a neighborhood
    of $g$ in $s^{-1}(x)$, and so contains an element of $A$.
    Pick such an element $a\in A$.
    Then $a = hg$ for some $h\in V$, and by our choice of $V$, the product $g = h^{-1}a$ is defined.
    But since $a$ can be written as a product of elements of $U$, this last equation
    shows that $g$ also can.
\end{proof}

The following lemma shows that, in some sense,
little information is lost by restricting a strongly connected local Lie groupoid.

\begin{lemma}
    Suppose that we have two strongly connected local Lie groupoid
    structures on $G$ over $M$, with the same soure, target and unit maps,
    and the same domains for multiplication and inversion.
    Assume that both are 3-associative.
    In other words, we have two tuples $(s,t,u, \U, \V,m_j : \U \to G,i_j : \V\to G)$
    for $j\in\{1,2\}$, each making $G$ into a strongly connected local Lie groupoid
    over $M$.
    If the two structures have a common restriction, so if there is some
    $\U' \subset \U$ and $\V'\subset \V$ such that
    $m_1{\restriction_{\U'}} = m_2{\restriction_{\U'}}$ and
    $i_1{\restriction_{\V'}} = i_2{\restriction_{\V'}}$
    (neighborhoods of $M$ and $M\timesst G \cup G\timesst M$),
    then $m_1 = m_2$ and $i_1 = i_2$.
    \label{lem:stronglyconnecteddeterminedbyrestriction}
\end{lemma}
\begin{proof}
    Let us first show that the multiplication maps agree.
    Pick a point $x\in M$.
    Let $F = \{ (g,h) \in \U \mid s(g) = t(h) = x \}$.
    By strong connectedness of $G$, the set $F$ is connected.
    We will show that $m_1{\restriction_F} = m_2{\restriction_F}$.
    Since $x$ was arbitrary, this will prove that $m_1 = m_2$.
    We argue that the set
    \[ A = \{ (g,h) \in F \mid m_1(g,h) = m_2(g,h) \} \]
    is open and closed (so that it is all of $F$).
    Closedness of $A$ is obvious, so we just have to show that $A$ is open.

    Pick $(g,h) \in A$.
    Let us write $gh$ for $m_1(g,h) = m_2(g,h)$.
    Let $N_h$ be a compact neighborhood of $s(h)$ in $t^{-1}(s(h))$, small enough
    such that for all $\overline{h} \in N_h$ we have
    \begin{itemize}
        \item $(h,\overline{h}) \in \U'$ (so that we can write $h\overline{h}$ without ambiguity),
        \item $(g,h\overline{h}) \in \U$,
        \item $(gh, \overline{h}) \in \U'$.
    \end{itemize}
    Then
    \begin{align*}
        m_1(g,h\overline{h})
        &= m_1(gh,\overline{h}) \\
        &= m_2(gh,\overline{h}) \\
        &= m_2(g,h\overline{h}),
    \end{align*}
    showing that $(g,h\overline{h}) \in A$
    (so that we can write $gh\overline{h}$ without ambiguity).
    Now let $N_g$ be a neighborhood of $t(g)$ in $s^{-1}(t(g))$ small enough
    such that for all $\overline{h} \in N_h$ and $\overline{g} \in N_g$ we have
    \begin{itemize}
        \item $(\overline{g},g) \in \U'$ (so that we can write $\overline{g}g = m_1(\overline{g},g) = m_2(\overline{g},g)$),
        \item $(\overline{g},gh\overline{h}) \in \U'$,
        \item $(\overline{g}g,h\overline{h}) \in \U$.
    \end{itemize}
    Then
    \begin{align*}
        m_1(\overline{g}g,h\overline{h})
        &= m_1(\overline{g},gh\overline{h}) \\
        &= m_2(\overline{g},gh\overline{h}) \\
        &= m_2(\overline{g}g,h\overline{h}),
    \end{align*}
    showing that $(\overline{g}g,h\overline{h}) \in A$.
    But by bi-regularity of $m_1$ and $m_2$, elements of the form $(\overline{g}g,h\overline{h})$ form
    a neighborhood of $(g,h)$ in $F$.
    This shows that $A$ is open, and therefore that the multiplications agree.
    Because of 3-associativity, the multiplication map completely determines the inversion
    map (up to restriction). This shows that the inversion maps also agree.
\end{proof}

\begin{remark}
The above lemma shows that we do not lose any information by restricting a
local Lie groupoid,
on the condition that it is sufficiently connected.
This essentially means that we can identify a local Lie groupoid with any of
its restrictions without changing the theory of local Lie groupoids.

The situation is different when talking about \emph{shrinking} local Lie groupoids.
If we identify $G$ and $G'$ whenever $G'$ can be obtained by shrinking $G$, we are
studying \emph{germs} of local Lie groupoids.
When doing this, one misses out on some of the subtler points of the theory of associativity
for local Lie groupoids.
As we will see, the associativity properties of a local Lie groupoid can change drastically
upon shrinking.
\end{remark}

\section{The Lie algebroid of a local Lie groupoid}

In the same way as for Lie groupoids, one can associate a Lie algebroid to each
local Lie groupoid. However, there are some subtleties about left- and right-invariant
vector fields that we will discuss explicitly.

Suppose that $G$ is a local Lie groupoid over $M$.
Just like for Lie groupoids, the vector bundle underlying the algebroid is
\[ A = T^s_MG ,\]
the restriction of $T^sG$ to $M$.
The anchor is the map $\rho : A \to TM$ obtained by restricting $dt : TG \to TM$ to $A$.

In order to define the algebroid bracket, we must discuss invariant vector fields on $G$.
To each section $\alpha \in \Gamma(A)$, we can associate a vector field
\[ \tilde{\alpha}_g = d_{t(g)}R_g(\alpha_{t(g)}) ,\]
where $R_g$ is right translation by $g\in G$. Note that $d_{t(g)}R_g$ maps
$T^s_{t(g)}G$ to $T^s_gG$, so that $\tilde{\alpha}$ is a well-defined vector field
that is tangent to the source fibers.
A vector field $X \in \mathfrak{X}(G)$ is \emph{right-invariant} if it is tangent
to the source fibers and
\[ dR_h(X_g) = X_{gh} \quad\text{for all $g,h\in\U$} .\]

\begin{lemma}
    Suppose that $G$ is a
    3-associative local Lie groupoid.
    Then the right-invariant vector fields are precisely of the form $\tilde{\alpha}$
    for $\alpha \in \Gamma(A)$.
\end{lemma}

\begin{proof}
It is clear that every right-invariant vector field must be of the form $\tilde{\alpha}$
for some $\alpha\in\Gamma(A)$. The converse follows from 3-associativity in the usual way.
\end{proof}

The Lie bracket of two right-invariant vector fields is again right-invariant,
and we define the algebroid bracket by requiring
\[ \widetilde{[\alpha,\beta]} = [\tilde{\alpha},\tilde{\beta}] .\]
Then $(A,\rho,[,])$ is the Lie algebroid of $G$.
Note that for a local Lie groupoid that is not 3-associative, we can define its Lie
algebroid by going to a 3-associative restriction.

%% file: chapters/associativity.tex
\chapter{Associativity and globalizability}
\label{chapter:associativity}

In this chapter, we discuss the \emph{globalizability} of local Lie groupoids,
i.e.\ whether they are part of a global Lie groupoid or not.
The main result of this chapter is the generalization of Mal'cev's theorem,
which answers this question for local Lie groups,
to local Lie groupoids.

\section{Mal'cev's theorem for local Lie groups}
\label{sec:malcev}

Perhaps the most natural examples of local Lie groupoids are open neighborhoods of
the identity manifold in Lie groupoids, and restrictions thereof.
We would like to characterize the local Lie groupoids that arise in this way.

\begin{definition}
    We call a local Lie groupoid \emph{globalizable} if it is a restriction
    of an open neighborhood of the identity manifold in a Lie groupoid.
\end{definition}

A theorem by Mal'cev characterizes globalizable local Lie groups.
In the language we will define in a moment, the theorem states the following.

\begin{theorem}[Mal'cev]
    A strongly connected local Lie group is globalizable if and only if it is
    globally associative.
\end{theorem}

We will generalize this result to local Lie groupoids.
Our proof will be analogous to the proof for local groups, though
a little bit of extra effort is required in proving smoothness of the
constructed groupoid.

\section{Example of a non-globalizable local Lie groupoid}
\label{sec:example-non-globalizable}

Let us illustrate the globalization problem and the various notions of associativity
with some examples of local Lie groupoids.

\subsection*{A Lie groupoid}

We start by describing a very simple global Lie groupoid.
We take $M = \Sphere^2$, equipped with its usual area form $\omega$ of total area $4\pi$.
We will write $A$ for the map that calculates the area enclosed by a loop:
\[ A : \Omega(M) \to \frac{\R}{4\pi} : \gamma \mapsto \int_{\Gamma} \omega ,\]
where $\gamma : [0,1] \to M$ is a loop and $\Gamma : [0,1]\times[0,1] \to M$
is any homotopy contracting $\gamma$ to the trivial loop
at $\gamma(0) = \gamma(1)$.
Note that this area is well-defined up to $4\pi$, because any two such homotopies will
have areas differing by an element $\int_{[\alpha]} \omega \in 4\pi\Z$ where $[\alpha]\in\pi_2(\Sphere^2)$.

Now let
\[ P = \{ \text{piecewise smooth paths in $M$} \} \times \frac{\R}{4\pi} .\]
We say that two elements $(\gamma_1,a_1), (\gamma_2,a_2) \in P$ are equivalent for ${\sim}$
if $\gamma_1(0) = \gamma_2(0)$ and $\gamma_1(1) = \gamma_2(1)$ and
\[ a_2 = a_1 + A(\gamma_2^{-1}\cdot\gamma_1) ,\]
where $\cdot$ represents concatenation of paths, and $\gamma_2^{-1}$ is the reverse of $\gamma_2$.

Let
\[ \G = P/{\sim} .\]
This is a smooth manifold of dimension 5.
It is a Lie groupoid over $M$ with multiplication
\[ [\gamma_1,a_1] [\gamma_2,a_2] = [\gamma_1\cdot\gamma_2, a_1+a_2] .\]

The Lie algebroid of $\G$ is the one associated to the area form $\omega$ (example 2.26
in \cite{lectures-integrability-lie}).
More precisely, as a bundle it is $TM \oplus \R$ (with anchor the projection
to the first summand),
and the bracket is given by
\[ [(X,f),(Y,g)] = ([X,Y],\mathcal{L}_X(g) - \mathcal{L}_Y(f) + \omega(X,Y)) .\]

\subsection*{A globalizable local Lie groupoid}

We now focus our attention on an open part of the $\G$ we just constructed.
Note that any element of $\G$ that does not have antipodal source and target
has a unique representative whose path is a geodesic for the usual round metric.
We can restrict to the elements whose source and target are not antipodal,
and obtain a (globalizable) local Lie groupoid.
If we represent each such element using the geodesic representative, we can write
down this local Lie groupoid explicitly as
\[ G' = \{ (y,x) \in M\times M \mid x + y\neq 0 \} \times \frac{\R}{4\pi} .\]
The multiplication of $(z,y,a), (y,x,a') \in G'$ is then defined whenever $x+z\neq 0$,
and is given by
\[ (z,y,a) \cdot (y,x,a') = (z,x,a+a'+A(\Delta xyz)) ,\]
where $A(\Delta xyz) \in \R/4\pi$ is the signed area of the spherical triangle $\Delta xyz$.
This local Lie groupoid $G'$ is globalizable, because it is an open part of $\G$
under the inclusion
\[ (y,x,a) \mapsto [\text{geodesic from $x$ to $y$}, a] .\]

\subsection*{A non-globalizable local Lie groupoid}
\label{sss:non-glob}

By introducing a slight variation, we can make the above local Lie groupoid non-globalizable.
We will no longer quotient out the areas by $4\pi$.
Explicitly, we still have $M = \Sphere^2$ and we set
\[ G'' = \{(x,y)\in M\times M \mid x+y\neq 0 \} \times \R .\]
We will define the multiplication of $(z,y,a)$ and $(y,x,a')$ only if $x+z\neq 0$
and $-\pi < A(\Delta xyz) < \pi$. \phantomsection\label{example:gprimeprime}
In that case, it is defined as above, by the formula
\[ (z,y,a) \cdot (y,x,a') = (z,x,a+a'+A(\Delta xyz)) .\]
To check that this local Lie groupoid is 3-associative, let
$(z,y,a_1), (y,x,a_2), (x,w,a_3) \in G''$.
Then
\begin{align*}
    ((z,y,a_1)\cdot(y,x,a_2))\cdot(x,w,a_3) &= (z,w,a_1+a_2+a_3+A(\Delta xyz)+A(\Delta wxz)) \\
    (z,y,a_1)\cdot((y,x,a_2)\cdot(x,w,a_3)) &= (z,w,a_1+a_2+a_3+A(\Delta wxy)+A(\Delta wyz))
\end{align*}
(if both are defined)
and $A(\Delta xyz)+A(\Delta wxz)=A(\Delta wxy)+A(\Delta wyz) \pmod{4\pi}$, because both equal the area
of the quadrangle $wxyz$ (which is defined up to $4\pi$).
But by the restriction that the areas have to be in $(-\pi,\pi)$,
this implies $A(\Delta xyz)+A(\Delta wxz)=A(\Delta wxy)+A(\Delta wyz)$,
proving 3-associativity.

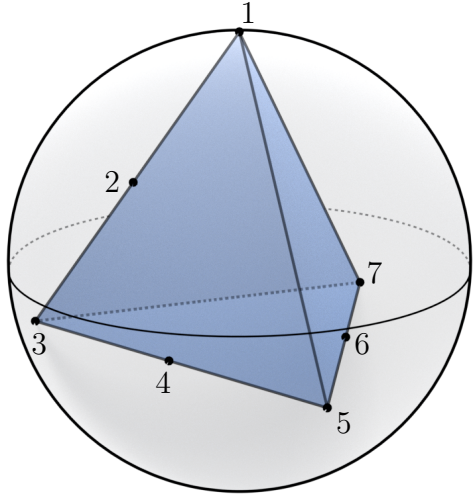
\begin{figure}
    \centering
    \begin{overpic}[width=0.4\linewidth]{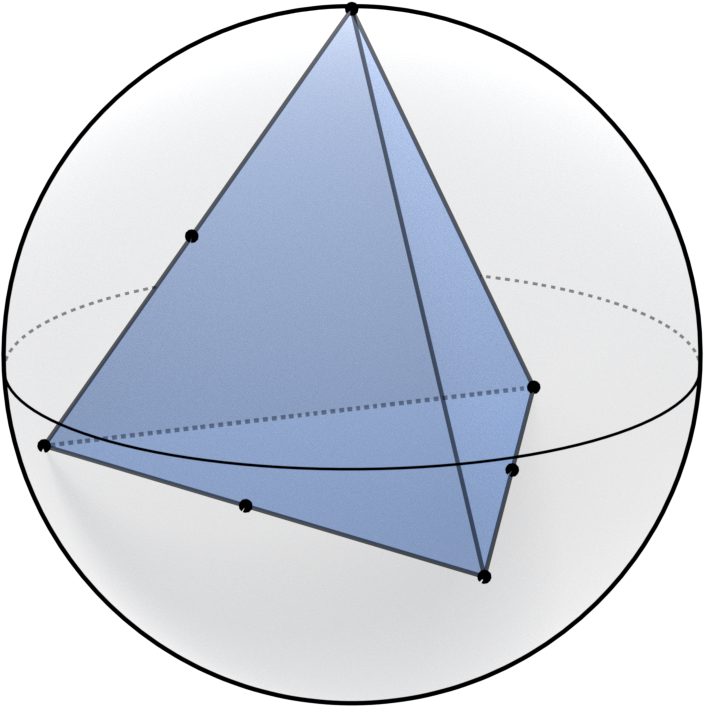}
        \put (50,100.6) {1}
        \put (21,64.5) {2}
        \put (5.4,29.8) {3}
        \put (31.9,21.5) {4}
        \put (70.5,13) {5}
        \put (74.4,30) {6}
        \put (77,45) {7}
    \end{overpic}
    \caption{Demonstrating the lack of 6-associativity. The tetrahedron is inscribed in a sphere, and the even-numbered points are midpoints of the edges they lie on.}
    \label{fig:tetrahedron}
\end{figure}

This local Lie groupoid is not globally associative.
Consider figure \ref{fig:tetrahedron}, which shows a regular tetrahedron inscribed in $\Sphere^2$.
Let $x_j \in \Sphere^2$ be the normalization of the $j$'th marked point in the picture (i.e.\ the radial projection of the point to the circumscribed sphere).
Let
\begin{align*}
    A &= (x_2,x_1,0) & D &= (x_5,x_4,0) \\
    B &= (x_3,x_2,0) & E &= (x_6,x_5,0) \\
    C &= (x_4,x_3,0) & F &= (x_7,x_6,0).
\end{align*}
Then
\[ (x_7,x_1,2\pi) = F(E((D(CB))A)) \neq ((F((ED)C))B)A = (x_7,x_1,-2\pi) .\]

Of course, global associativity is necessary for globalizability.
Our groupoid $G''$ is therefore not globalizable.
Note, however, that there is a neighborhood of the identity manifold in $G''$
that \emph{is} globalizable:
if $G''' = \{ (y,x,a) \in G'' \mid -2\pi < a < 2\pi \}$,
then $G'''$ is globalizable (and in fact isomorphic to an open part of $G'$ and hence of $\G$).
As we will see later, this is always the case for local Lie groupoids whose algebroid
is integrable.

\section{Associative completions}
\label{sec:asco}

In this section, we review a construction that associates to a local Lie
groupoid a global groupoid. In \cite[appendix A]{bailey-gualtieri}, this construction is called
the \emph{formal completion} of the local groupoid (and it is used only for
globally associative structures). The construction is also used in
\cite{olver}, though it is not explicitly named.
We will call this construction the \emph{associative completion} of the local Lie
groupoid, to highlight that it has the effect of
\begin{itemize}
    \item making the structure globally associative,
    \item globalizing the structure to a (global) groupoid, so that the
        multiplication is complete.
\end{itemize}

Suppose that $G$ is a local Lie groupoid. Then $\bigsqcup_{n\geq 1} G^n$ is the set of words on $G$.
A word $(w_1,...,w_n)$ is called \emph{well-formed} if $s(w_i) = t(w_{i+1})$ for all $i\in\{1,...,n-1\}$.
The set of well-formed words on $G$ of length $n$ is nothing but
\[ \underbrace{G \timesst G \timesst \cdots \timesst G}_{\text{$n$ copies}} .\]
Write
\[ W(G) = \{ \text{well-formed words on $G$} \} .\]
Note that if $G$ happens to be an open part of a global Lie groupoid, then a groupoid
globalizing it
may be obtained as a quotient of $W(G)$ by an appropriate equivalence relation, as follows.

\begin{definition}
If $w = (w_1,...,w_k,w_{k+1},...,w_n)$ is a well-formed word and $(w_k,w_{k+1}) \in \mathcal{U}$,
we will say that the word
\[ w' = (w_1,...,w_kw_{k+1},...,w_n) \]
is obtained from $w$ by \emph{contraction}.
We also say that $w$ is obtained from $w'$ by \emph{expansion}.
If two words are related by an expansion or contraction, we will say that they are \emph{elementarily equivalent}.

We say that two words are \emph{equivalent} (written $w\sim w'$) if they are equivalent under the equivalence
relation generated by elementary equivalences.
\end{definition}

\begin{definition}
If $G$ is a local Lie groupoid, then
\[ \AsCo(G) = W(G)/{\sim} \]
is called the \emph{associative completion} of $G$.
\end{definition}

\begin{proposition}
    If $G$ is an inversional local Lie groupoid, then $\AsCo(G)$ is a groupoid
    (in the algebraic sense).
\end{proposition}
\begin{proof}
    The source and target maps are given by
    \[ s([w_1,...,w_n]) = s(w_n) \quad\text{and}\quad t([w_1,...,w_n]) = t(w_1) .\]
    The multiplication map is juxtaposition.
    The unit at $x\in M$ is given by $[x]$.
    It remains to show that $[w_1,...,w_n]$ has an inverse.
    Because $G$ is inversional, every $w_i$ can be written as a product of invertible elements
    $w_i^1,...,w_i^{k_i}$.
    Then we have
    \[ [w_1,...,w_n]^{-1} = [(w_n^{k_n})^{-1},...,(w_n^1)^{-1},...,(w_1^{k_1})^{-1},...,(w_1^1)^{-1}] .\]
    These operations turn $\AsCo(G)$ into a groupoid.
\end{proof}

If $f : G \to H$ is a morphism of local Lie groupoids,
we have a commutative diagram
\begin{center}
\begin{tikzpicture}
    \matrix(m)[matrix of math nodes, row sep=2.2em, column sep=3.6em]{
        G & H \\
        \AsCo(G) & \AsCo(H). \\
    };
    \path[->] (m-1-1) edge node[auto]{$f$} (m-1-2);
    \path[->] (m-1-1) edge (m-2-1);
    \path[->] (m-2-1) edge node[auto]{$\AsCo(f)$} (m-2-2);
    \path[->] (m-1-2) edge (m-2-2);
\end{tikzpicture}
\end{center}

It is interesting to study the kernel of the map $G \to \AsCo(G)$ (i.e.\ those elements
that are mapped to a unit).
The following elements are clearly in this kernel.

\begin{definition}
    Suppose that $G$ is a local Lie groupoid over $M$, and that $x\in M$.
    Then an element $g\in G_x = s^{-1}(x)\cap t^{-1}(x) \subset G$ is
    called an \emph{associator at $x$} if there is a word
    \[ (w_1,...,w_k) \in G^k \]
    that can be evaluated (in $G$) to both $x$ and $g$ (by choosing different
    ways of putting in brackets).
    We write
    $\Assoc_x(G)$
    for the set of all associators at $x$.
    We write $\Assoc(G)$ for the set of all associators in $G$.
\end{definition}

\begin{definition}
    Suppose $G$ is a local Lie groupoid. We say that $G$ \emph{has products connected to the axes}
    if for every $(g,h) \in \U$, there is a path $\gamma$ from $t(h)$ to $g$ in $G$
    such that $(\gamma(\tau),h) \in \U$ for all $\tau$,
    or there is a path $\gamma$ from $s(g)$ to $h$ in $G$ such that
    $(g,\gamma(\tau)) \in \U$ for all $\tau$.
\end{definition}

Note that every local Lie groupoid has a restriction with products connected to the axes.

The following result seems fairly innocent, but it is non-trivial and at the heart
of many results about associativity for local groupoids.

\begin{proposition}
    Suppose $G$ is a bi-regular local Lie groupoid that has products connected to the axes.
    Let $g,h \in G$.
    Then $(g) \sim (h)$ if and only if there is a sequence of expansions,
    followed by a sequence of contractions, starting at $(g)$ and ending at $(h)$.
    \label{prop:equivalence-relation}
\end{proposition}

We will prove this proposition in section \ref{sec:proof-malcev}.
It has the following corollary.

\begin{corollary}
    For a bi-regular local Lie groupoid $G$ with products connected to the axes,
    the kernel of $G\to\AsCo(G)$ is precisely $\Assoc(G)$.
\end{corollary}

In that case, we have the following relations:
\begin{align*}
    & \text{$G \to \AsCo(G)$ is injective} \\
    \iff & \text{$G$ is globally associative} \\
    \Longrightarrow\; & \text{$\Assoc(G)$ is trivial (i.e.\ $\Assoc(G) = M \subset G$)}.
\end{align*}

Note that triviality of $\Assoc(G)$ does not imply that $G$ is globally associative.
(A morphism of local Lie groupoids with trivial kernel is not necessarily injective.
For a simple example, consider the morphism of Lie groups $\R \to \R/\Z$.
Its restriction to $(-0.6,0.6)$ has trivial kernel, but it is not injective.)

\begin{definition}
    We say that $\Assoc(G) \subset G$ is \emph{uniformly discrete} if
    there is an open neighborhood $U$ of $M$ in $G$ such that
    $U\cap\Assoc(G) = M$.
\end{definition}

We can now state one of our main results.
We will prove it in the next section.

\begin{maintheorem}
    If $G$ is a bi-regular local Lie groupoid with products connected to the axes,
    then $\AsCo(G)$ is
    smooth if and only if $\Assoc(G)$ is uniformly discrete in $G$
    (and in that case, $G \to \AsCo(G)$ is a local diffeomorphism).
    \label{thm:smoothness-of-asco}
\end{maintheorem}

(By ``$\AsCo(G)$ is smooth'', we mean, of course, that $\AsCo(G)$ has a smooth
structure as a quotient of $W(G)$, where $W(G)$ is considered as a manifold with
components of various dimensions.)
The following corollary characterizes globalizable local Lie groupoids.
The results is entirely analogous to the situation for local Lie groups.

\begin{corollary}[Mal'cev's theorem for local Lie groupoids]
    A strongly connected local Lie groupoid is globalizable if and only if it is globally associative.
    \label{cor:malcev-for-groupoids}
\end{corollary}
\begin{proof}
    Call the local Lie groupoid $G$.
    If $G$ globalizable, it is clearly globally associative.
    If $G$ is globally associative, then restrict it to get a local Lie groupoid
    $G'$ with products connected to the axes.
    Then $G' \hookrightarrow \AsCo(G')$ is the inclusion
    of $G'$ as a restriction of an open part $U$ of a global Lie groupoid.
    (Injectivity follows from global associativity.)
    This means that $G$ and $U$ have a common restriction (namely $G'$), so that
    by strong connectedness and lemma~\ref{lem:stronglyconnecteddeterminedbyrestriction}
    the inclusion $G\hookrightarrow\AsCo(G')$ is also the inclusion of a
    restriction of $U$.
    The Lie groupoid $\AsCo(G')$ therefore globalizes $G$.
\end{proof}

Every lie algebroid can be integrated to a local Lie groupoid \cite[corollary 5.1]{crainic-fernandes-lie}.
Because the globalizability of a local Lie groupoid is related to its higher associativity,
this points to a link between associativity and integrability. The following corollary
is a first indication.

\begin{corollary}
    A Lie algebroid is integrable if and only if
    there is a local Lie groupoid integrating it that has uniformly discrete associators.
\end{corollary}
\begin{proof}
    Suppose a Lie algebroid is integrable. Then by definition,
    there is a Lie groupoid integrating it, and this is a local Lie groupoid with uniformly discrete
    (indeed, trivial) associators.

    Suppose that a Lie algebroid has a local Lie groupoid integrating it, with uniformly discrete associators.
    Then by taking an open neighborhood of the units in this local Lie groupoid,
    small enough to intersect the set of associators only in identities, we can get
    a strongly connected local Lie groupoid integrating the Lie algebroid with trivial associators.
    This local Lie groupoid is therefore globalizable, and we can conclude that there exists a
    Lie groupoid integrating our Lie algebroid.
\end{proof}

\section{Proof of Mal'cev's theorem}
\label{sec:proof-malcev}

We now prove Mal'cev's theorem by establishing theorem~\ref{thm:smoothness-of-asco}.
Our first step is to prove proposition~\ref{prop:equivalence-relation},
which in turn requires some lemmas.
They will be about well-formed words on $G$.
Let us call any (sub)word of the form $(a_1,...,a_k,a_k^{-1},...,a_1^{-1})$ a \emph{block}.

\begin{lemma}
    Let $G$ be a strongly connected local Lie groupoid
    with products connected to the axes.
    Let $w$ and $w'$ be well-formed words on $G$.
    Suppose that $w'$ is a contraction of $w$. Then there is a series of expansions starting
    at $w$ and ending at a word that equals $w'$ with a block inserted at some position.
    \label{lem:contractionintoexpansion}
\end{lemma}

\begin{lemma}
    Let $G$ be a strongly connected local Lie groupoid.
    If $z\in G$, a well-formed word of the form $(z,\text{block})$ can be transformed
    using expansions only into
    a word of the form $(\text{some other block},z)$.
    \label{lem:switchblocks}
\end{lemma}

We postpone the proofs of these lemmas to the end of this section.

\begin{proof}[Proof of proposition~\ref{prop:equivalence-relation}]
    Suppose $g\sim h$. Then there is some sequence of elementary equivalences,
    that will transform the word $(g)$ into the word $(h)$ (expansions and contractions
    may show up mixed, in any order).
    Let $w^0 = (g)$, $w^1$, ..., $w^{n-1}$, $w^n = (h)$ be the sequence of words connecting $(g)$ to $(h)$.
    This means that each $w^{i+1}$ is either an expansion or a contraction of $w^i$.

    The plan is now as follows: we will build up a sequence $u^0, ..., u^n$ starting at $u^0 = (g)$,
    where each $u^{i+1}$ is obtained from $u^i$ by a series of expansions, and $u^i$ equals
    $w^i$ preceded by zero or more blocks.
    Then $u^n$ is of the form $(\text{zero or more blocks}, h)$, and by a series of contractions we can
    get rid of the blocks and transform this into $(h)$. This shows that there is a sequence of expansions,
    followed by contractions, connecting $(g)$ to $(h)$.

    To build up the sequence $u^i$, set $u^0 = w^0$. Once $u^i$ has been defined, we obtain $u^{i+1}$
    as follows. If $w^{i+1}$ is an expansion of $w^i$, then $u^{i+1}$ is obtained from $u^i$ by the
    analogous expansion (leaving the blocks in front alone).
    If $w^{i+1}$ is a contraction of $w^i$, the procedure is more complicated. We will make use of the
    lemmas above.

    To find $u^{i+1}$, we first apply lemma \ref{lem:contractionintoexpansion} to the contraction
    of $u^i$ corresponding to the contraction from $w^i$ to $w^{i+1}$ (leaving the blocks in front alone).
    We obtain a word of the form
    \[ (\text{blocks},\text{letters of $w^{i+1}$},\text{block}, \text{letters of $w^{i+1}$}) .\]
    Using lemma \ref{lem:switchblocks} we can move the block that has appeared in $w^{i+1}$ to the front
    using expansions, obtaining a word of the form
    \[ (\text{blocks},w^{i+1}) .\]
    This then is the word we take for $u^{i+1}$.
    This shows how to build up the required sequence $u^0, ..., u^n$, proving the result.
\end{proof}

We are now ready to prove the theorem.

\begin{proof}[Proof of theorem \ref{thm:smoothness-of-asco}]
    Suppose first that $\AsCo(G)$ is smooth (with smooth structure determined
    by the submersion $W(G) \to \AsCo(G)$).
    Note that the fibers of $G \to \AsCo(G)$ are countable,
    so that $\dim(\AsCo(G)) \geq \dim(G)$. Because $W(G)$ has a component
    of dimension $\dim(G)$ and $\AsCo(G)$ is a quotient of $W(G)$,
    we must have $\dim(\AsCo(G)) \leq \dim(G)$, and therefore
    $\dim(\AsCo(G)) = \dim(G)$ (this is of course the only reasonable dimension
    to expect).
    The map $G \to \AsCo(G)$ is therefore a local diffeomorphism.
    The inverse image of $M \subset \AsCo(G)$ under $G \to \AsCo(G)$,
    which is precisely $\Assoc(G)$, is therefore
    an embedded submanifold of $G$. This shows that $\Assoc(G)$ is uniformly discrete
    in $G$.

    Suppose now that $\Assoc(G)$ is uniformly discrete.
    We will prove that $\AsCo(G)$ is smooth.
    Given an element $h\in\AsCo(G)$, represented by a word $(x_1,...,x_n)$, we will construct
    a chart near $h$ modeled on an open part of $G$.
    For each $x_i$, take a small submanifold $N_i \subset G$ of the same dimension as $G$
    through $x_i$ that is
    transverse to both source and target fibers (i.e.\ a local bisection of the local Lie groupoid near $x_i$).
    Then $s$ and $t$ define diffeomorphisms between the $N_i$ and open subsets of $M$.
    Write $s_i = s{\restriction_{N_i}}$ and $t_i = t{\restriction_{N_i}}$ for these diffeomorphisms.

    Let $k\in\{1,...,n\}$.
    Let $U$ be a small neighborhood of $x_k$.
    Then the chart we use near $h$ is defined as
    \[ \varphi : U \to \AsCo(G) : y \mapsto [y_1,...,y_{k-1},y,y_{k+1},...,y_n] \]
    where
    \begin{align*}
        y_{k+1} &= t_{k+1}^{-1}(s(y)) \\
        y_{k+2} &= t_{k+2}^{-1}(s(y_{k+1})) \\
        &\cdots \\
        y_n &= t_n^{-1}(s(y_{n-1}))
    \end{align*}
    and
    \begin{align*}
        y_{k-1} &= s_{k-1}^{-1}(t(y)) \\
        y_{k-2} &= s_{k-2}^{-1}(t(y_{k-1})) \\
                 &\cdots \\
        y_1 &= s_1^{-1}(t(y_2)) .
    \end{align*}
    If we pick $U$ sufficiently small, this is well-defined.
    Note that it maps $x_k$ to $h$.
    We claim that $\varphi$ is injective, if we pick $U$ small enough
    (this is where we need the uniform discreteness).
    Indeed, suppose that $\varphi(y) = \varphi(z)$.
    Write the representatives of $\varphi(y)$ and $\varphi(z)$ as given above as
    \[ (y_1,...,y_{k-1},y,y_{k+1},...,y_n) \quad\text{and}\quad (z_1,...,z_{k-1},z,z_{k+1},...,z_n) .\]
    Because $\varphi(y) = \varphi(z)$, we must have $s(\varphi(y)) = s(\varphi(z))$ and
    therefore $s(y_n) = s(z_n)$.
    But then $y_n$ and $z_n$ are both on $N_n$, and have the same source. Therefore, $y_n = z_n$.
    This implies that $s(y_{n-1}) = s(z_{n-1})$.
    But then $y_{n-1}$ and $z_{n-1}$ are both on $N_{n-1}$, and have the same source.
    Therefore, $y_{n-1} = z_{n-1}$.
    Continuing this way, we find that $y_i = z_i$ for $i > k$.
    The same argument applied starting from the left and using the target map shows that
    $y_i = z_i$ for $i<k$.
    We conclude that
    \[ [y_1,...,y_{k-1},y,y_{k+1},...,y_n] = \varphi(y) = \varphi(z) = [y_1,...,y_{k-1},z,y_{k+1},...,y_n] .\]
    Multiplying both sides by $[y_1,...,y_{k-1}]^{-1}$ (from the left) and by $[y_{k+1},...,y_n]^{-1}$ (from the
    right) shows that $[y] = [z]$.
    If $B$ is a neighborhood of $M$ in $G$ such that $B\cap \Assoc(G) = M$, then for $U$ small enough,
    we must have $z=y\delta$ for $\delta$ in $B$.
    But then $[\delta]$ must be trivial in $\AsCo(G)$, so that it is an associator,
    implying that $\delta \in M$. We conclude that $y=z$.

    Let us check smoothness of the transition maps.
    Given a point in $\AsCo(G)$, a chart is determined by
    \begin{itemize}
        \item a choice of representative,
        \item a choice of $k$,
        \item a choice of local bisections.
    \end{itemize}
    We will show smoothness of the following transitions:
    \begin{itemize}
        \item for a fixed representative and index $k$, the transition between charts coming
            from different choices of local bisections,
        \item for a fixed representative, the transition between charts for two different
            choices of $k$,
        \item the transition between charts constructed using different representatives of
            the same element of $\AsCo(G)$.
    \end{itemize}
    Together, these show that all the transitions are smooth.

    Let us start by checking the first kind of transition map.
    Suppose $(x_1,...,x_n)$ represents $h\in\AsCo(G)$. Pick $k\in\{1,...,n\}$.
    Making one choice of local bisections $N_i$ leads to a chart
    \[ \varphi : y \mapsto [y_1,...,y_{k-1},y,y_{k+1},...,y_n] .\]
    Note that the $y_i$ depend smoothly on $y$.
    Now make a different choice of local bisections $N_i'$, leading to
    a second chart.
    We will calculate the transition map between the two charts.
    Suppose $y$ is near $x_k$.
    Let $z_n$ be the point on $N_n'$ with source $s(y_n) = s(h)$.
    Then
    \[ z_n = \eps_n y_n \]
    for some invertible $\eps_n$ near $t(y_n)$ (on the condition that
    $y$ is sufficiently close to $x_k$).
    Note that $\eps_n$ depends smoothly on $y$, and $\eps_n = t(y_n)$ if $y = x_k$.
    The product $y_{n-1}\eps_n^{-1}$ is defined for $y$ close enough to $x_k$.
    Let $z_{n-1}$ be the point on $N_{n-1}'$ with source $t(z_n)$.
    Then
    \[ z_{n-1} = \eps_{n-1} (y_{n-1} \eps_n^{-1}) \]
    for some invertible $\eps_{n-1}$ near $t(y_{n-1})$ (on the condition that
    $y$ is sufficiently close to $x_k$).
    Note that $\eps_{n-1}$ depends smoothly on $y$, and $\eps_{n-1} = t(y_{n-1})$
    if $y = x_k$.
    Continuing this way, we construct a sequence $z_n, ..., z_{k+1}$ on the $N_i'$ and
    $\eps_n, ..., \eps_{k+1}$
    such that
    \[ z_i = \eps_i(y_i\eps_{i+1}^{-1}) ,\]
    where we will set $\eps_{n+1} = s(y_n)$ for ease of notation.
    All the $\eps_i$ depend smoothly on $y$ (we may have to restrict $y$ to be in a
    smaller and smaller neighborhood of $x_k$).
    Similarly, we start from the left and construct a sequence $z_1, ..., z_{k-1}$
    on the $N_i'$ and $\eps_1, ..., \eps_{k-1}$ such that
    \[ z_i = (\eps_{i-1}^{-1}y_i)\eps_i \]
    with all the $\eps_i$ depending smoothly on $y$.
    Note now that
    \begin{align*}
        \varphi(y) &= [y_1,...,y_{k-1},y,y_{k+1},...,y_n] \\
                   &= [y_1,\eps_1,...,\eps_{i-2}^{-1},y_{k-1},\eps_{k-1},y,\eps_{k+1}, y_{k+1},\eps_{k+2}^{-1},...,\eps_n,y_n] \\
                   &= [z_1,...,z_{k-1},(\eps_{k-1}y\eps_{k+1}),z_{k+1},...,z_n] .
    \end{align*}
    where the product in the middle will be defined if $y$ is sufficiently close to $x_k$.
    This shows that the transition map between our two charts is given by
    \[ y \mapsto \eps_{k-1}y\eps_{k+1} ,\]
    which is smooth in $y$ (because the $\eps_i$ are).
    This proves smoothness of a transition of the first kind.
    
    Let us now check smoothness of a transition between charts for the same
    representative, but different choices of $k$. This is much easier.
    We may assume that the two choices for $k$ differ by 1,
    so notation-wise we will look at the chart associated to the representative
    $(x_1,...,x_n)$ for index $k$ and index $k+1$.
    Again, write
    \[ y \mapsto [y_1,...,y,y_{k+1},...y_n] \]
    for the chart associated to index $k$.
    Suppose that $y$ is near $x_k$.
    Let $y_k$ be the point on $N_k$ with target $t(y)$.
    This point depends smoothly on $y$.
    For $y$ sufficiently close to $x_k$, we may write
    $y_k = y\eps$
    where $\eps$ is near $s(x_k)$ and depends smoothly on $y$.
    Then
    \begin{align*}
        [y_1,...,y,y_{k+1},...y_n] &= [y_1,...,y,\eps,\eps^{-1},y_{k+1},...,y_n] \\
                                   &= [y_1,...,y_k,(\eps^{-1}y_{k+1}),...,y_n]
    \end{align*}
    for $y$ sufficiently close to $x_k$.
    This shows that the transition map is $y\mapsto \eps^{-1}y_{k+1}$, which is smooth.

    We now check the smoothness of the third and final type of transition,
    that between different representatives of the same point.
    This type is the easiest to check.
    Indeed, we only need to check smoothness for the transition for a single contraction
    or expansion, but that is clear from smoothness of $\G$'s multiplication.

    Note, finally, that $\AsCo(G)$ is second-countable because $W(\G)$ is.
    We have proved that $\AsCo(G)$ is smooth.
\end{proof}

\subsection*{Proofs of the lemmas}

Finally, we give the proofs of the lemmas.

\begin{proof}[Proof of lemma \ref{lem:contractionintoexpansion}]
    Suppose that
    \[ w = (x_1,...,x_k,x_{k+1},...,x_n) \quad\text{and}\quad w' = (x_1,...,,x_kx_{k+1},...,x_n) .\]
    Because $G$ has products connected to the axes,
    there is either a curve in $G$ from $t(x_{k+1})$ to $x_k$ in $s^{-1}(t(x_{k+1}))$
    that can be right-translated by $x_{k+1}$, or a curve from $s(x_k)$ to $x_{k+1}$ in
    $t^{-1}(s(x_k))$ that can be left-translated by $x_k$.
    We will assume the former, and show that $w$ can be expanded to a word of the form
    \[ (x_1,...,\text{block},x_kx_{k+1},...,x_n) .\]
    The argument for the latter case is entirely similar, ending with a word of the form
    \[ (x_1,...,x_kx_{k+1},\text{block},...,x_n) .\]
    Let $\gamma$ be such a curve in $s^{-1}(t(x_{k+1}))$ from $t(x_{k+1})$ to $x_k$
    and let $\gamma'$ be the curve obtained from $\gamma$ by right-translating by $x_{k+1}$.

    Let $T \subset [0,1]$ be the set of all $\tau\in[0,1]$ such that
    there is a sequence $a_1, ..., a_l \in G$ such that
    \[ \gamma(\tau) = a_l(\cdots(a_1(s(x_k)))) \]
    and
    \[ x_{k+1} = a_1^{-1}(\cdots(a_l^{-1}(\gamma'(\tau)))) \]
    are defined and true. We claim that $T = [0,1]$.
    We will prove this by showing that $T\subset[0,1]$ is both open and closed (clearly it is
    nonempty, because $0\in T$).

    We first prove that $T$ is open. Pick $\tau_0\in T$.
    Let $a_1,...,a_l$ be a sequence such that $\gamma(\tau_0) = a_1(\cdots(a_l(s(x_k))))$
    and
    $x_{k+1} = a_1^{-1}(\cdots(a_l^{-1}(\gamma'(\tau_0))))$.
    Now $t(\gamma(\tau_0))$ has a neighborhood in $s^{-1}(t(\gamma(\tau_0)))$
    of invertible elements
    such that
    for every $g$ in this neighborhood we have
    \[ (g,\gamma(\tau_0)) \in \U \text{ and } (g^{-1}, g\gamma'(\tau_0)) \in \U .\]
    Then by bi-regularity of $G$, for all $\tau$ close enough to $\tau_0$ we can write
    \[ \gamma(\tau) = g\gamma(\tau_0) = g(a_l(\cdots(a_1(s(x_k))))) \]
    and
    \[ x_{k+1} = a_1^{-1}(\cdots(a_l^{-1}(g^{-1}(\gamma'(\tau))))) .\]
    This shows that $T$ is open.
    
    We now show that $T$ is closed.
    Pick $\tau_0 \in \partial T$.
    Now $t(\gamma(\tau_0))$ has a neighborhood in $t^{-1}(t(\gamma(\tau_0)))$
    of invertible elements
    such that
    for every $g$ in this neighborhood we have
    \[ (g^{-1}, \gamma'(\tau_0)) \in \U \text{ and } (g^{-1}, \gamma(\tau_0)) \in \U \text{ and } (g,g^{-1}\gamma(\tau_0)) \in \U .\]
    Using bi-regularity of $G$ and picking $\tau \in T$ sufficiently close to $\tau_0$ we can find a $g$
    in this neighborhood
    such that $\gamma(\tau) = g^{-1}\gamma(\tau_0)$. If $a_1,...,a_l$ are such that
    \[ \gamma(\tau) = a_l(\cdots(a_1(s(x_k)))) \]
    and
    \[ x_{k+1} = a_1^{-1}(\cdots(a_l^{-1}(\gamma'(\tau)))) ,\]
    then
    \[ \gamma(\tau_0) = g(a_l(\cdots(a_1(s(x_k))))) \]
    and
    \[ x_{k+1} = a_1^{-1}(\cdots(a_l^{-1}(g^{-1}(\gamma'(\tau_0))))) .\]
    This shows that $\tau_0 \in T$, proving that $T$ is closed.
    Therefore, $T=[0,1]$.

    We have proved that $1\in T$, so we can find a sequence $a_1, ..., a_l$ such that
    \[ x_k = a_l(\cdots(a_1(s(x_k)))) = a_l(\cdots(a_1)) \]
    and
    \[ x_{k+1} = a_1^{-1}(\cdots(a_l^{-1}(\gamma'(1)))) = a_1^{-1}(\cdots(a_l^{-1}(x_kx_{k+1}))) .\]
    This shows that we can expand
    \[ (x_1,...,x_k,x_{k+1},...,x_n) \]
    into
    \[ (x_1,...,a_l,...,a_1,a_1^{-1},...,a_l^{-1},x_kx_{k+1},...,x_n) ,\]
    proving the lemma.
\end{proof}

\begin{proof}[Proof of lemma \ref{lem:switchblocks}]
    Suppose the initial word is $(z,x_1,...,x_n,x_n^{-1},...,x_1^{-1})$.
    Note that $z$ and $x_1^{-1}$ are in the same source fiber.
    By source-connectedness, we can take a path $\gamma$ in $s^{-1}(s(z))$
    from $z$ to $x_1^{-1}$.
    Let $T\subset [0,1]$ consist of all $\tau\in[0,1]$ for which there is a sequence
    $a_1,...,a_k$ such that
    \[ \gamma(\tau) = a_k^{-1}(\cdots(a_1^{-1}(z))) \text{ and } z = a_1(\cdots(a_k(\gamma(\tau)))) .\]
    We will show that $T \subset [0,1]$ is both open and closed, and so $T = [0,1]$ (clearly
    it is non-empty, because $0\in T$).

    Let $\tau_0 \in T$. Let $a_1,...,a_k$ be a sequence as above.
    Then there is an open neighborhood of $t(\gamma(\tau_0))$ in $s^{-1}(t(\gamma(\tau_0)))$
    consisting of invertible elements
    such that for each $g$ in this neighborhood we have
    \[ (g^{-1},\gamma(\tau_0)) \in \U \text{ and } (g,g^{-1}\gamma(\tau)) \in \U .\]
    Therefore, using bi-regularity of $G$, for $\tau$ sufficiently close to $\tau_0$, we can find an element $g$
    such that
    \[ \gamma(\tau) = g^{-1}(a_k^{-1}(\cdots(a_1^{-1}(z)))) \text{ and } z = a_1(\cdots(a_k(g(\gamma(\tau))))) .\]
    This proves that $T$ is open.

    Now let $\tau_0 \in \partial T$.
    Then there is a neighborhood of $t(\gamma(\tau_0))$ in $t^{-1}(t(\gamma(\tau_0)))$
    of invertible elements such that for every $g$ in this neighborhood we have
    \[ (g^{-1},\gamma(\tau_0)) \in \U \text{ and } (g,g^{-1}\gamma(\tau_0)) \in \U .\]
    Using bi-regularity of $G$ and picking $\tau \in T$ sufficiently close to $\tau_0$
    we can find a $g$ in this neighborhood
    such that $\gamma(\tau) = g^{-1}\gamma(\tau_0)$. If $a_1,...,a_k$ are such that
    \[ \gamma(\tau) = a_k^{-1}(\cdots(a_1^{-1}(z))) \text{ and } z = a_1(\cdots(a_k(\gamma(\tau)))) ,\]
    then
    \[ \gamma(\tau_0) = g(a_k^{-1}(\cdots(a_1^{-1}(z)))) \text{ and } z = a_1(\cdots(a_k(g^{-1}(\gamma(\tau_0))))) .\]
    This shows that $\tau_0 \in T$.
    Therefore, $T$ is also closed. We conclude that $T = [0,1]$.
    
    Since $1\in T$, we can find a sequence $a_1,...,a_k \in G$ such that
    \[ x_1^{-1} = a_k^{-1}(\cdots(a_1^{-1}(z))) \]
    and
    \[ z = a_1(\cdots(a_k(x_1^{-1}))) .\]
    This means we can expand the initial word into
    \[ (a_1,..., a_k,x_1^{-1},x_1,x_2,...,x_n,x_n^{-1},...,x_2^{-1},a_k^{-1},...,a_1^{-1},z) .\]
    We can expand this further into
    \[ (a_1,..., a_k,x_1^{-1},x_1,x_2,...,x_n,x_n^{-1},...,x_2^{-1},t(a_k^{-1}),a_k^{-1},...,a_1^{-1},z) \]
    and finally
    \[ (\underbrace{a_1,..., a_k,x_1^{-1},x_1,x_2,...,x_n,x_n^{-1},...,x_2^{-1},x_1^{-1},x_1,a_k^{-1},...,a_1^{-1}}_{\text{block}},z) ,\]
    which is of the required form. This proves the lemma.
\end{proof}

%% file: chapters/classification.tex
\chapter{Classification of local Lie groupoids with integrable algebroids}
\chaptermark{Classification of local Lie groupoids}
\label{chapter:classification}

Olver \cite{olver} has classified local Lie groups in terms of Lie groups.
In this chapter, we generalize this result to local Lie groupoids.
Our results are direct generalizations of results by Olver to the groupoid case
\cite[theorems 19 and 20]{olver}.
The result we obtain can only be used to classify those local Lie groupoids
that have integrable algebroids, a restriction that does not show up in Olver's
result because all Lie algebras are integrable.
Before we can explain the classification result proper, we have to study
coverings of local Lie groupoids.

\section{Coverings of local Lie groupoids}

Our first task is to understand coverings of local Lie groupoids.
Just like in the theory of global Lie groupoids, coverings play an important
role for local Lie groupoids.
Recall that for groupoids, the notion of covering concerns not the space
of arrows, but rather the source fibers (or target fibers).
We will assume that all the local groupoids in this section have connected source
and target fibers.

\begin{definition}
    Suppose that $G$ is a local Lie groupoid over $M$ (with source map $s$,
    target map $t$, unit map $u$).
    A \emph{covering} (resp.\ \emph{generalized covering}) of $G$
    is a tuple $(G',s',t',u')$ and a map $\varphi : G' \to G$, such that
    \begin{itemize}
        \item $u' : M\into G$ is an embedding,
        \item $s',t' : G' \to M$ are surjective submersions such that $s'\circ u' = t'\circ u' = \id$,
        \item $\varphi : G' \to G$ is a smooth map such that $s' = s\circ\varphi$, $t'=t\circ\varphi$, $u=\varphi\circ u'$,
        \item for each $m\in M$, the map $\varphi{\restriction_{(s')^{-1}(m)}} : (s')^{-1}(m) \to s^{-1}(\varphi(m))$ is a covering map (resp.\ local diffeomorphism).
    \end{itemize}
\end{definition}

Our goal for this section is the following result.

\begin{theorem}
    A generalized covering $\varphi : G' \to G$ of a bi-regular local Lie
    groupoid has the structure of a local Lie groupoid making
    $\varphi$ into a morphism of local Lie groupoids.
    The local Lie groupoid $G'$ is also bi-regular.
    \label{thm:coverings}
\end{theorem}

\begin{remark}
    Every local Lie groupoid has a source-simply connected cover.
    It is constructed in the same way as for global Lie groupoids (details can be found
    in \cite{lectures-integrability-lie}).
    In short, the source-simply connected cover is obtained as $p^{-1}(M)$,
    where
    \[ p : \Pi_1(\F_s) \to G, \]
    is the source map of the fundamental groupoid of the source foliation of $G$.
    The theorem above therefore shows that every local Lie groupoid is covered by a
    source-simply connected local Lie groupoid.
    As explained in \cite{the-article}, it is possible to give a (simpler)
    proof of the existence of source-simply connected covering groupoids in its
    own right. We refer the reader to that article for the details.
\end{remark}

The key to understanding \cref{thm:coverings} are Maurer-Cartan forms.
For a Lie group, the left- and right-invariant Maurer-Cartan forms are usually considered as Lie algebra valued
1-forms, but we will formulate them as bundle maps $TG \to \pi^*(T_{\{e\}}G)$
where $\pi : G \to \{e\}$ (this is the same thing, of course).
The following are the analogues of the left-invariant and right-invariant
Maurer-Cartan forms, respectively.

\begin{definition}
    An \emph{$s$-framing} is an isomorphism of vector bundles $\omega_s : T^sG
    \to t^*(T^s_MG)$ over the identity such that $\omega_s{\restriction_M} =
    \id$.
    A \emph{$t$-framing} is an isomorphism of vector bundles $\omega_t :
    T^tG \to s^*(T^t_MG)$ over the identity such that $\omega_t{\restriction_M}
    = \id$.
\end{definition}

If $G$ is a bi-regular local Lie groupoid, the left- and right-invariant Maurer Cartan forms
are $s$- and $t$-framings, with
\[ \omega_{MC}^R : T^s_gG \ni v \mapsto (d_{t(g)}\rho_{g})^{-1}(v) \]
and
\[ \omega_{MC}^L : T^t_gG \ni v \mapsto (d_{s(g)}\lambda_{g})^{-1}(v) ,\]
where $\rho_g$ is right multiplication by $g$, and $\lambda_g$ is left multiplication by $g$.
For a local Lie groupoid, the framings used will be these Maurer-Cartan forms.
Given framings, we can talk about invariant vector fields.

\begin{definition}
    Suppose we are given an $s$-framing $\omega_s$.  A section of $T^sG$ is
    called \emph{right-invariant} if its image under the $s$-framing is of the
    form $t^*\sigma$ for some $\sigma \in \Gamma(T^s_MG)$.

    Suppose we are given a $t$-framing.  A section of $T^tG$ is called
    \emph{left-invariant} if its image under the $t$-framing is of the form
    $s^*\sigma$ for some $\sigma \in \Gamma(T^t_MG)$.
\end{definition}

\begin{lemma}
    If $G$ is a local Lie groupoid, then
    \begin{itemize}
        \item the Lie bracket of right-invariant vector fields is right-invariant,
        \item the Lie bracket of left-invariant vector fields is left-invariant,
        \item the Lie bracket of a right-invariant and a left-invariant vector field is zero.
    \end{itemize}\label{lem:brackets-of-invariant-vf}
\end{lemma}
\begin{proof}
    We first prove that the bracket of right-invariant vector fields is right-invariant.
    Note that a vector field $X$ on $G$ is right-invariant if and only if
    its value at $g\in G$ is $(d\rho_g)(X_{t(g)})$.
    Let $X_1$ and $X_2$ be right-invariant vector fields on $G$.
    By the local associativity of $G$, the vector fields $X_1$ and $X_2$ are
    invariant under $\rho_g$ near $t(g)$.
    Therefore, the value of $[X_1,X_2]$ at $g$ is indeed equal to $d\rho_g([X_1,X_2](t(g)))$, and $[X_1,X_2]$ is right-invariant.
    Similarly, the bracket of left-invariant vector fields is left-invariant.

    We now show that right-invariant vector fields commute with left-invariant vector fields.
    Suppose that $X$ is a right-invariant vector field on $G$,
    and $Y$ is a left-invariant vector field on $G$.
    Let $g\in G$.
    Let $\tau\mapsto \gamma_X(\tau)$ be the integral curve of $X$ with $\gamma_X(0) = t(g)$,
    and let $\tau'\mapsto \gamma_Y(\tau')$ be the integral curve of $Y$ with $\gamma_Y(0) = s(g)$.
    Then note that $\tau \mapsto \gamma_X(\tau) \cdot g$ is also an integral curve of $X$
    (defined for $\tau$ small),
    by local associativity.
    Also $\tau \mapsto \gamma_X(\tau) \cdot g \cdot \gamma_Y(\tau')$ is an integral curve of $X$ (defined for $\tau$ and $\tau'$ small).
    Similarly, the curves $\tau' \mapsto g\cdot \gamma_Y(\tau')$ and
    $\tau' \mapsto \gamma_X(\tau) \cdot g \cdot \gamma_Y(\tau')$ are integral curves
    of $Y$ (for $\tau,\tau'$ small).
    It follows that the flows of $X$ and $Y$ commute for small times, because flowing $g$
    first by $X$ for time $\tau$ and then by $Y$ for time $\tau'$ gives the same result as
    first flowing by $Y$ and then by $X$
    (namely $\gamma_X(\tau)\cdot g \cdot \gamma_Y(\tau')$).
\end{proof}

\begin{remark}
For a 3-associative local Lie groupoid, right-invariant vector fields are
indeed invariant under right
multiplication (and similarly for left-invariant vector fields under left
multiplication).
This justifies the choice of terminology.
\end{remark}

\Cref{thm:coverings} will follow from the following characterization.
It is a direct generalization of theorem 18 in \cite{olver}.

\begin{proposition}
    Let $G$, $M$, $s$, $t$ and $u$ be as above.
    Then given $s$- and $t$-framings, there is a
    local Lie groupoid structure giving rise to these
    framings if and only if
    \begin{itemize}
        \item the Lie bracket of left-invariant vector fields is left-invariant,
        \item the Lie bracket of right-invariant vector fields is right-invariant,
        \item the Lie bracket of a left-invariant vector field and a right-invariant vector field is zero.
    \end{itemize}
    If we demand that the local Lie groupoid be strongly connected, the structure is unique
    up to restriction.
    \label{prop:characterization}
\end{proposition}

\begin{remark}
Theorem 18 in \cite{olver} claims that giving a local Lie group structure is equivalent to
prescribing right-invariant vector fields,
without specifying the left-invariant vector fields.
Explicit examples show, however, that this is not sufficient, and the left-invariant
vector fields need to be specified as well. More precisely, if only right-invariant
vector fields are specified, it may happen that no compatible left-invariant
vector fields exist. The mistake in the proof is that the
multiplication does not get defined near $G\times\{e\}$ (only near $\{e\}\times G$).
For an explicit example, one can look at $M = \R_{>0} \times \Sphere^1$, with coordinates $(x,\theta)$. The vector fields $X = x\frac{\partial}{\partial x}$ and $Y = x\frac{\partial}{\partial \theta}$ satisfy $[X,Y] = Y$.
If we try to prescribe $X$ and $Y$ as right-invariant vector fields on $M$,
there is no way to find corresponding left-invariant vector fields to make $M$ into a
local Lie group: if $Z$ is a left-invariant vector field that equals $\frac{\partial}{\partial x}$ at $(1,1)$, we must have $[Y,Z]=0$, but the flow of $Y$
maps the point $(1,1)$ to itself after time $2\pi$,
mapping $\frac{\partial}{\partial x}$ to $\frac{\partial}{\partial x} + 2\pi\frac{\partial}{\partial \theta}$.
This shows that $Z$ cannot be invariant under the flow of $Y$.
\end{remark}

\begin{proof}[Proof of \cref{prop:characterization}]
    The ``only if'' part of the proof is \cref{lem:brackets-of-invariant-vf}.
    We prove the ``if'' part.
    Suppose that we have framings that satisfy the above conditions.
    We will prescribe the multiplication and inversion maps.
    Let us start by defining the multiplication on a neighborhood of
    $G \timesst M$.
    Pick $g\in G$ and write $x = s(g)$. If $h\in t^{-1}(x)$ is sufficiently close
    to $x$,
    then
    \[ h = \phi^1_{\omega_t^{-1}(s^*\sigma)}(x) \]
    for some local section $\sigma$ of $T^t_MG$ near $x$
    (here $\phi^1$ denotes the time-1 flow of a vector field).
    Note that $\omega_t^{-1}(s^*\sigma)$ is left-invariant.
    We set
    \[ g\cdot h = \phi^1_{\omega_t^{-1}(s^*\sigma)}(g) \]
    if this flow exists.
    Note that this product $g\cdot h$ is well-defined for $h$ small,
    in the sense that it doesn't depend on the choice of $\sigma$.
    This defines the multiplication near $G \timesst M$.
    Similarly, one defines $g\cdot h = \phi^1_{\omega_s^{-1}(t^*\tau)}(h)$ if $g = \phi^1_{\omega_s^{-1}(t^*\tau)}(x)$
    is small, and $\tau$ is a local section of $T^s_MG$.

    One thing remains to check to ensure we have a well-defined multiplication:
    if $g = \phi^1_{\omega_s^{-1}(t^*\tau)}(x)$ and $h = \phi^1_{\omega_t^{-1}(s^*\sigma)}(x)$ for local
    sections $\tau$ and $\sigma$ as above, we need to ensure that the two definitions
    of $g\cdot h$ agree.
    The first definition defines $g\cdot h$ as
    \[ \phi^1_{\omega_t^{-1}(s^*\sigma)}(g) = \phi^1_{\omega_t^{-1}(s^*\sigma)}(\phi^1_{\omega_s^{-1}(t^*\tau)}(x)) \]
    and the second as
    \[ \phi^1_{\omega_s^{-1}(t^*\tau)}(h) = \phi^1_{\omega_s^{-1}(t^*\tau)}(\phi^1_{\omega_t^{-1}(s^*\sigma)}(x)) .\]
    Because left- and right-invariant vector fields commute, these definitions coincide
    for $g, h$ sufficiently close to $x$.
    Therefore, by restricting if necessary, we have
    defined a multiplication map on $G \timesst G$ near
    $G \timesst M \cup M \timesst G$.

    We leave it to the reader to check that this multiplication is locally associative
    (possibly after restricting
    further),
    and to define an inversion map near $M$.
    The uniqueness up to restriction follows immediately from
    lemma~\ref{lem:stronglyconnecteddeterminedbyrestriction}.
\end{proof}

\section{Classification result}

In \cite{olver}, Olver shows roughly speaking that every local Lie group
is covered by a cover of a globalizable local Lie group.
We prove the analogous result for groupoids, and strengthen it slightly.

\begin{maintheorem}
    Suppose $G$ is a bi-regular $s$-connected local Lie groupoid over $M$
    with integrable Lie algebroid $A$.
    Write $\tilde{G}$ for its source-simply connected cover,
    and write $\G(A)$ for the source-simply connected integration of $A$.
    Then we have the following commutative diagram.
    \begin{center}
    \begin{tikzpicture}
        \matrix(m)[matrix of math nodes, row sep=2.2em, column sep=2.6em]{
            & \tilde{G} & \\
            G & & U\subseteq \G(A) \\
            & \AsCo(G) & \\
        };
        \path[->] (m-1-2) edge node[auto,swap]{$p_1$} (m-2-1.north east);
        \path[->] (m-1-2) edge node[auto]{$p_2$} (m-2-3.north west);
        \path[->] (m-2-1.south east) edge (m-3-2);
        \path[->] (m-2-3.south west) edge (m-3-2);
    \end{tikzpicture}
    \end{center}
    Here, $U$ is an open neighborhood of $M$ in $\G(A)$, $p_1$ is the covering map,
    and the map $G \to \AsCo(G)$ is the natural map from a local Lie group to its
    associative completion.
    The map $p_2$ is a generalized covering of local Lie groupoids,
    which sends a point $g\in\tilde{G}$ to the class of the $A$-path associated
    to a $\tilde{G}$-path from $s(g)$ to $g$.
    \label{thm:classification-groupoids}
\end{maintheorem}

\begin{proof}
    Our first step is to show the existence of the upper half of the diagram.
    This is the statement that is proved in \cite[theorem 21]{olver}
    for the case of local groups, and the proof is analogous.
    We use Cartan's method of the graph.
    The argument is similar to the usual argument for integrability of Lie algebroid
    morphisms to Lie groupoid morphisms \cite[proposition 6.8]{moerdijk-mrcun}.

    Consider the fibered product $\tilde{G} \tensor[_t]{\times}{_t} \G(A)$.
    We equip it with a foliation $\F$, given by
    \[ \F_{(g_1,g_2)} = \left\{ (\xi\cdot g_1, \xi\cdot g_2) \mid \xi \in A_{t(g_1)} \right\} .\]
    Integrability of this distribution follows from the fact that $\tilde{G}$ and $\G(A)$
    have the same Lie algebroid.
    Fix a point $x\in M$.
    Write $N_x$ for the leaf of $\F$ through $(x,x)$.

    We claim that the projection $\pi_1 : N_x \to s^{-1}(x)$ is a covering map
    (we are taking the source fiber
    in $\tilde{G}$).
    Note that $\pi_1 : N_x \to s^{-1}(x)$ is a local diffeomorphism by bi-regularity of the local groupoid.
    Now if $(g_1,g_2) \in N_x$, and $U_1$ is a neighborhood of $(g_1,g_2)$ in $N_x$ that gets mapped
    diffeomorphically to a neighborhood $U_2$ of $g_1$ in $s^{-1}(x)$,
    then $U_2$ is uniformly covered. Indeed, if $(g_1,g_3)$ is another point on $N_x$,
    then by invariance of $\F$ under right multiplication in the $\G(A)$-direction,
    the translated submanifold $U_1 \cdot g_2^{-1} \cdot g_3$ is a neighborhood of $(g_1,g_3)$ in $N_x$
    that gets mapped diffeomorphically to $U_1$.
    This shows that the projection $\pi_1 : N_x \to s^{-1}(x)$ is a covering map.
    Because $s^{-1}(x)$ is simply connected, it is a diffeomorphism.
    The inverse is a diffeomorphism $\phi_x : s^{-1}(x) \to N_x$.

    Let $N = \bigcup_{x\in M} N_x$.
    Putting together the $\phi_x$ gives a map $\phi : \tilde{G} \to N$. This map is smooth,
    because it is the extension by holonomy of the map $x\mapsto (x,x)$.
    We define
    \[ p_2 = \pi_2 \circ \phi .\]
    Letting $U$ be the image of $p_2$, we obtain the required diagram
    \begin{equation}
    \begin{tikzpicture}[baseline=(current  bounding  box.center)]
        \matrix(m)[matrix of math nodes, row sep=2.2em, column sep=2.6em]{
            & \tilde{G} & \\
            G & & U\subseteq \G(A) \\
        };
        \path[->] (m-1-2) edge node[auto,swap]{$p_1$} (m-2-1.north east);
        \path[->] (m-1-2) edge node[auto]{$p_2$} (m-2-3.north west);
        \tag{$\dag$}
    \end{tikzpicture}
    \end{equation}
    
    We now establish the existence of the lower half of the diagram.
    We apply the functor $\AsCo$ to $p_1$, which results in
    \begin{center}
    \begin{tikzpicture}[baseline=(current  bounding  box.center)]
        \matrix(m)[matrix of math nodes, row sep=2.2em, column sep=2.6em]{
            & \tilde{G} & \\
            G & & \AsCo(\tilde{G}) \\
               & \AsCo(G) & \\
        };
        \path[->] (m-1-2) edge node[auto,swap]{$p_1$} (m-2-1.north east);
        \path[->] (m-1-2) edge (m-2-3.north west);
        \path[->] (m-2-3.south west) edge node[auto]{$\AsCo(p_1)$} (m-3-2);
        \path[->] (m-2-1.south east) edge (m-3-2);
    \end{tikzpicture}
    \end{center}
    We claim that $\AsCo(\tilde{G}) \cong \G(A)$, and that under this identification the natural
    map $\tilde{G} \to \AsCo(\tilde{G})$ corresponds to $p_2$.

    Note that $\AsCo(\tilde{G})$ is smooth (the associators of $\tilde{G}$ are discrete
    because they lie in the fiber of $\tilde{G} \to U$ over the identity).
    We have maps $\AsCo(\tilde{G}) \to \AsCo(U) \to \AsCo(\G(A)) \cong \G(A)$,
    and all three of these are Lie groupoids. At the level of Lie algebroids, these maps
    are isomorphisms, and because $\G(A)$ is source-simply connected, we must in fact have
    \[ \AsCo(\tilde{G}) \cong \AsCo(U) \cong \AsCo(\G(A)) \cong \G(A) .\]
    This is the identification we were looking for.
    Following along the maps, we see that an element $g\in\tilde{G}$ is mapped to $\G(A)$
    as follows:
    \[ g\in\tilde{G} \mapsto [g]\in\AsCo(\tilde{G}) \mapsto [p_2(g)] \in \AsCo(U) \mapsto [p_2(g)] \in \AsCo(\G(A)) \mapsto p_2(g) \in \G(A) .\]
    The resulting diagram
    \begin{center}
    \begin{tikzpicture}[baseline=(current  bounding  box.center)]
        \matrix(m)[matrix of math nodes, row sep=2.2em, column sep=2.6em]{
            & \tilde{G} & \\
            G & & \G(A) \\
               & \AsCo(G) & \\
        };
        \path[->] (m-1-2) edge node[auto,swap]{$p_1$} (m-2-1.north east);
        \path[->] (m-1-2) edge node[auto]{$p_2$} (m-2-3.north west);
        \path[->] (m-2-3.south west) edge (m-3-2);
        \path[->] (m-2-1.south east) edge (m-3-2);
    \end{tikzpicture}
    \end{center}
    is what we were looking for.
\end{proof}

\begin{corollary}
    If $G$ is a $s$-simply connected local Lie groupoid with integrable algebroid, then $\Assoc(G)$ is uniformly discrete (and, in particular, $\AsCo(G)$ is smooth).
\end{corollary}
\begin{proof}
    The above result shows that $G$ is a generalized cover of a globalizable local
    Lie group. The associators of $G$ are therefore contained in the fiber of this
    generalized covering over the identity, and this fiber is discrete.
\end{proof}

Note that this corollary is false without the assumption of $s$-simply connectedness.
For example, there are local Lie groups with non-discrete associators.
An example appears in the next chapter, on \cpageref{page:ladder}.

%% file: chapters/integrability.tex
\chapter{Associativity and integrability}
\label{chapter:integrability}

In this chapter, we discuss the relationship between the integrability of a Lie algebroid
and the associativity of a local groupoid integrating it.
After some basic observations in \ref{sec:associators}, we give an example in \ref{sec:example},
in order to give the reader more intuition.
In \ref{sec:review-simplicial}, \ref{sec:expansion-contraction} and \ref{sec:monodromy}, we briefly discuss some ideas about
simplicial sets and about monodromy groups.
Finally, in \ref{sec:assoc-is-mono} and \ref{sec:mono-is-assoc}, we prove that
associators and monodromy groups are intimately related.

\section{Observations about associators}
\label{sec:associators}

Let us study the set of associators of a local Lie groupoid more closely.
The product of two associators (if defined) is again an associator.
Note that in a globally associative local Lie group, all associators are trivial,
meaning that for all $x\in M$,
$\Assoc_{G}(x) = \{ x \}$.
As it turns out, restricting a local Lie groupoid does not change its associators,
in the following sense.

\begin{lemma}
    Let $G$ be a local Lie groupoid over $M$, with multiplication $m_1 : \U_1 \to G$
    and inversion map $i_1 : \V_1 \to G$.
    Write $G_1$ for $G$ with this local groupoid structure.
    Suppose that we restrict $G$ to get a local Lie groupoid over $M$
    with multiplication map $m_2 : \U_2 \to G$ and inversion map $i_2 : \U_2 \to G$.
    Write $G_2$ for $G$ with this local groupoid structure.
    Assume that $G_1$ and $G_2$ have products connected to the axes.
    Then for all $x\in M$ we have
    \[ \Assoc_x(G_1) = \Assoc_{x}(G_2) .\]
\end{lemma}
\begin{proof}
    Let $W$ be the set of well-formed words on $G$.
    We consider two equivalence relations on $W$.
    We will write ${\sim_1}$ for the equivalence relation on $W$ generated by contractions
    and expansions for the multiplication on $G_1$,
    and ${\sim_2}$ for the equivalence relation on $W$ generated by contractions and
    expansions for the multiplication on $G_2$.
    (These are precisely the equivalence relations that were considered in the construction
    of the associative completions of $G_1$ and $G_2$.)
    Proposition \ref{prop:equivalence-relation} tells us that
    \[ g\in \Assoc_{x}(G_1) \iff (g)\sim_1 (x) \quad\text{and}\quad
    g\in \Assoc_{x}(G_2) \iff (g)\sim_2 (x) .\]
    It therefore suffices to show that ${\sim_1}$ and ${\sim_2}$ are the same equivalence relation.

    Clearly,
    $w\sim_2 w' \Rightarrow w\sim_1 w'$.
    We claim that the converse implication also holds, so that the two equivalence relations are equal.
    It suffices to show that two words that are elementarily equivalent for ${\sim_1}$
    are also equivalent for ${\sim_2}$.
    Let $(w_1,...,w_k,w_{k+1},...,w_n) \in W$ such that $(w_k,w_{k+1}) \in \U_1$.
    Following along with the proof of lemma \ref{lem:contractionintoexpansion}
    \emph{for the local groupoid $G_2$},
    we find that we can expand (for the multiplication $m_2$) the subword $(...,w_k,w_{k+1},...)$ into
    \[ (...,a_l,...,a_1,a_1^{-1},...,a_l^{-1},m_1(w_k,w_{k+1}),...), \]
    where the inverses are inverses for $G_2$ (and thus also for $G_1$, but that
    is not important).
    We can then contract this word (for the multiplication $m_2$) into
    \[ (...,m_1(w_k,w_{k+1}),...), \]
    so that $(w_1,...,w_k,w_{k+1},...,w_n) \sim_2 (w_1,...,w_kw_{k+1},...,w_n)$.
    This shows that ${\sim_1}$ and ${\sim_2}$ are equal,
    proving the result.
\end{proof}

Shrinking a local Lie groupoid (i.e.\ replacing it with a neighborhood of $M$ in $G$) \emph{can}
change the associators drastically.
To illustrate, we will give an example of a local Lie group $G$
such that $\Assoc(G)$ is not discrete. \phantomsection\label{page:ladder}
Because every local Lie group has a neighborhood of the identity that is globalizable
(and therefore has trivial associators), this shows that the associators can change from
non-discrete to discrete by shrinking the local groupoid.

\begin{example}[Local Lie group with non-discrete associators]
Let $M$ be a singleton. Let $B \subset \R^2$ be the ladder-shaped set
\[ B = (\{0\}\times \R) \cup (\{1\}\times \R) \cup ([0,1]\times \Z) \]
and let $G$ be its thickening
\[ G = \left\{ p \in \R^2 \mid \exists q \in B \text{ such that } d(p,q) < \frac1{10} \right\} .\]
(Here, $d$ is the Euclidean distance.)
We will equip $G$ with the structure of a local Lie group in such a way that $\Assoc(G)$ is not discrete.

For each $n \in \Z_{>0}$, let $f_n : [\frac1{10},\frac9{10}] \to \R_{>0}$ be a smooth function that
is 1 outside of $[\frac13, \frac23]$ and such that the time-$\frac{8}{10}$ flow of the vector field
$f_n\frac{\partial}{\partial x}$ on $\R$ maps $\frac1{10}$ to $\frac9{10} + \frac1{100n}$.
In other words, the flow of $f_n\frac{\partial}{\partial x}$ is just a little bit faster than
that of $\frac{\partial}{\partial x}$.
For each $n \in \Z_{\leq 0}$, let $f_n$ be the function on $[\frac1{10},\frac9{10}]$ that takes value 1 everywhere.
Now consider the following vector fields on $G$:
\[ X(x,y) = \begin{cases} f_n(x)\frac{\partial}{\partial x} &
        \text{if } x\in\left[\frac1{10},\frac9{10}\right] \text{ and } y\in(n-\frac1{10},n+\frac1{10}) \\
        \frac{\partial}{\partial x} & \text{otherwise}, \end{cases} \]
\[ Y(x,y) = \frac{\partial}{\partial y} .\]
There is a unique local Lie group structure on $G$ over $M$ such that $X$ and $Y$
are bi-invariant (more precisely, unique up to restriction if we ask that the local groupoid has
with products connected to the axes).
We claim that for this local Lie group structure, $\Assoc(G)$ is not discrete.

Let $a = (\frac1{20},0)\in G$ and $b = (0,\frac1{20}) \in G$.
Then consider the product
\[ \underbrace{a^{-1}\cdots a^{-1}}_{20} \underbrace{b^{-1}\cdots b^{-1}}_{20n}
   \underbrace{b\cdots b}_{20n} \underbrace{a\cdots a}_{20} ,\]
where the numbers under the braces indicate the number of repetitions.
Clearly, this product results in the neutral element $(0,0)$ when evaluated from the inside out.
However, evaluating first the second half $c = b\cdots b a\cdots a$ from left to right,
and then evaluating the rest $a^{-1}\cdots a^{-1} b^{-1}\cdots b^{-1} c$ from right to left,
we get $\left(\frac1{100n},0\right)$.
This shows that $\left(\frac1{100n},0\right) \in \Assoc(G)$ for all positive $n$, so that the associators do not
form a discrete set.
\end{example}

\begin{figure}
    \centering
    \begin{tikzpicture}[scale=2]
        \clip (-2,-0.6) rectangle (2,1.5);

        \foreach \y in {-2,-1,0,1} {
            \draw (0.1,\y+0.1) -- (0.9,\y+0.1) -- (0.9,\y+0.9) -- (0.1,\y+0.9) -- cycle;
        }
        \draw (-0.1,-2) -- (-0.1,2);
        \draw (1.1,-2) -- (1.1,2);

        \draw (0,0) -- (0,1) -- (1.05,1) -- (1.05,0) -- (0.05,0);
        \draw[thick, postaction={decorate}, decoration={markings,mark=at position 0.5 with {\arrow{Latex}}}] (0,0) -- (0,1);
        \draw[thick, postaction={decorate}, decoration={markings,mark=at position 0.5 with {\arrow{Latex}}}] (0,1) -- (1.05,1);
        \draw[thick, postaction={decorate}, decoration={markings,mark=at position 0.5 with {\arrow{Latex}}}] (1.05,1) -- (1.05,0);
        \draw[thick, postaction={decorate}, decoration={markings,mark=at position 0.5 with {\arrow{Latex}}}] (1.05,0) -- (0.05,0);

        \draw[fill=black] (0,0) circle (0.015);
        \draw[fill=white] (0.05,0) circle (0.015);

        \draw[dashed,line width=0.5pt] (0.025,0) circle (0.2);

        \coordinate (A) at (2.5cm/2/2,-0.2cm/2/2);
        \draw[line width=0.5pt] (0.106873, 0.182474) -- (-0.929381, 0.647422);
        \draw[line width=0.5pt] (0.0755409, -0.193509) -- (-1.02338, -0.480527);

        \pgflowlevel{\pgftransformshift{\pgfpoint{-2.5cm}{0.2cm}}};
        \pgflowlevel{\pgftransformscale{2}};
        \pgftransformscale{1.5};
        \draw[dashed] (0.025,0) circle (0.2);
        \clip (0.025,0) circle (0.2);

        \foreach \y in {-2,-1,0,1} {
            \draw (0.1,\y+0.1) -- (0.9,\y+0.1) -- (0.9,\y+0.9) -- (0.1,\y+0.9) -- cycle;
        }
        \draw (-0.1,-2) -- (-0.1,2);
        \draw (1.1,-2) -- (1.1,2);

        \draw[thick] (0,0) -- (0,1) -- (1.05,1) -- (1.05,0) -- (0.05,0);

        \draw[fill=black] (0,0) circle (0.015);
        \draw[fill=white] (0.05,0) circle (0.015);

        \draw[line width=0.2pt,{Latex[length=0.3mm, width=0.4mm]}-{Latex[length=0.3mm, width=0.4mm]}] (0,-0.03) -- (0.05,-0.03);
    \end{tikzpicture}
    \caption{The vector field $X$ is obtained by slightly modifying
        $\frac{\partial}{\partial x}$.
        The origin is the black dot.
        If we flow the origin along $Y$
        ($=\frac{\partial}{\partial y}$) for time 1, then along $X$ for time 1,
        then along $-Y$ for time 1, and finally along $-X$ for time 1, we end
        up slightly away from the origin, at the white dot.
        This white dot is an associator.
        The gap is exaggerated in the picture for clarity.}
\end{figure}
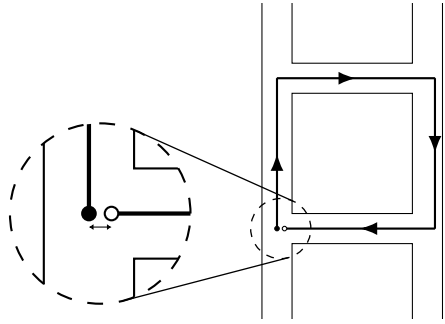

\section{Example}
\label{sec:example}

So far, we have not given an example of a local Lie groupoid with non-integrable algebroid.
We do this now.
The example will hopefully give the reader some intuition for associators.

Let $M = \Sphere^2\times \Sphere^2$, where each copy of $\Sphere^2$ is equipped with the usual round metric.
Let
\[ H = \{ ((y,y'),(x,x')) \in M\times M \mid x+y\neq 0\neq x'+y' \} \times \R .\]
We prescribe the source and target maps as
\[ s((y,y'),(x,x'),a) = (x,x') \in M \]
and
\[ t((y,y'),(x,x'),a) = (y,y') \in M .\]
We introduce a multiplication on $H$ by the formula
\[ ((z,z'),(y,y'),a_1) \cdot ((y,y'),(x,x'),a_2) = ((z,z'),(x,x'),a_1+a_2+A(\Delta xyz)+\lambda A(\Delta x'y'z')) ,\]
where $\lambda \in \R$ is a fixed parameter (recall that $A$ is the area),
and which is defined whenever
\[ A(\Delta xyz) \in (-\pi,\pi) \quad\text{and}\quad A(\Delta x'y'z') \in \left(-\frac{\pi}{\lvert \lambda\rvert},\frac{\pi}{\lvert\lambda\rvert}\right) .\]
(The condition on $A(\Delta x'y'z')$ is void if $\lambda = 0$.)
The algebroid of this local
Lie groupoid is integrable precisely if $\lambda$ is rational \cite[example 3.1]{lectures-integrability-lie}.

Let us check that $H$ is 3-associative.
Suppose that $((z,z'),(y,y'),a_1)$, $((y,y'),(x,x'),a_2)$ and $((x,x'),(w,w'),a_3)$ are in $H$.
Then
\begin{align*}
    & ( ((z,z'), (y,y'),a_1) \cdot ((y,y'),(x,x'),a_2) )\cdot ((x,x'),(w,w'),a_3) \\
    &= ((z,z'),(x,x'),a_1+a_2+a_3+A(\Delta xyz)+A(\Delta wxz)+\lambda A(\Delta x'y'z') + \lambda A(\Delta w'x'z'))
\end{align*}
and
\begin{align*}
    & ((z,z'), (y,y'),a_1) \cdot ( ((y,y'),(x,x'),a_2) \cdot ((x,x'),(w,w'),a_3) ) \\
    &= ((z,z'),(x,x'),a_1+a_2+a_3+A(\Delta wxy)+A(\Delta wyz)+\lambda A(\Delta w'x'y') + \lambda A(\Delta w'y'z')) .
\end{align*}
Now
\[ A(\Delta xyz) + A(\Delta wxz) = A(\Delta wxy) + A(\Delta wyz) \pmod{4\pi} \]
because both equal the area of the quadrangle $wxyz$ (which is defined up to $\pi$).
By the restriction on the areas of these triangles, equality must hold in $\R$ (not just in $\R/4\pi$).
Similarly, we will have
\[ \lambda A(\Delta x'y'z') + \lambda A(\Delta w'x'z') = \lambda A(\Delta w'x'y') + \lambda A(\Delta w'y'z') .\]
This proves 3-associativity.

This local Lie groupoid is not globally associative. The counterexample to 6-associativity
of $G''$ on \cpageref{example:gprimeprime} also works here,
by considering $G''$ as the local Lie subgroupoid $\{ (y,y'),(N,N) \} \times \R$ of $H$,
where $N \in \Sphere^2$ is the north pole.
Recall, however, that there was a neighborhood of $M \subset G''$ that was globally
associative.
This is not necessarily the case for $H$.

The following result will be a consequence of \cref{thm:assoc-mono}.

\begin{proposition}
    There is an open neighborhood of $M \subset H$ that is globally associative
    if and only if $\lambda \in \mathbb{Q}$ (i.e.\ iff the Lie algebroid is integrable).
    \label{lem:nonint}
\end{proposition}

\section{Review of simplicial complexes and sets}
\label{sec:review-simplicial}

We will describe the precise link between associators and monodromy groups later.
First, we quickly review the basics of simplicial complexes and simplicial sets.

An \emph{(abstract) simplicial complex} is a collection $S$ of finite non-empty sets,
such that if $A$ is an element of $S$, then so is every non-empty subset of $A$.
The elements of $S$ are called \emph{simplices}.
A simplex that contains $k+1$ elements is called a \emph{$k$-simplex}.
We will use the terms \emph{vertices}, \emph{edges} and \emph{faces} for the $0$-, $1$- and $2$-simplices
of a simplicial complex.
A face is said to be \emph{attached} to an edge if the edge is contained in the face.
In our arguments, $S$ will always be finite.
The \emph{dimension} of a simplicial complex is the largest $k$ for which it has a $k$-simplex.
The \emph{standard $n$-simplex} is the power set of $\{0,...,n\}$, with the empty set removed.
It is a simplicial complex of dimension $n$, and we will denote it by $[n]$.

An \emph{ordered simplicial complex} is a simplicial complex $S$,
together with a partial ordering of its vertices, such that the order is linear (i.e. total)
on each simplex.
The standard $n$-simplex comes with an obvious ordering of its vertices, making it into an
ordered simplicial complex.

Let $\Delta$ be the simplex category (its objects are linearly ordered sets of the form $\{0,...,n\}$,
and its morphisms are (non-strictly) order-preserving functions between such sets).
A \emph{simplicial set} is a contravariant functor from $\Delta$ to the category of small sets.
A morphism of simplicial sets is a natural transformation between them.
The \emph{degeneracy maps} of a simplicial set are the images of all of the maps
from $\{0,...,n+1\}$ to $\{0,...,n\}$ (for any $n$).
The \emph{face maps} of a simplicial set are the images of all of the maps
from $\{0,...,n\}$ to $\{0,...,n+1\}$ (for any $n$).

Every ordered simplicial complex $S$ can be turned into a simplicial set,
by prescribing it on objects as
\[ \{0,...,n\} \mapsto \{ \text{order-preserving maps $\{0,...,n\}\to\Verts(S)$ whose image is in $S$} \} ,\]
where $\Verts(S)$ is the set of vertices of $S$,
and on morphisms by composition.

Given a local Lie groupoid $G$ over $M$, we can associate to it a simplicial set $\Nerve_2G$
of dimension 2 by setting
\begin{align*}
    \{ \text{0-simplices} \} &= M \\
    \{ \text{1-simplices} \} &= G \\
    \{ \text{2-simplices} \} &= \U \subset G \timesst G .
\end{align*}
(Recall that $\U$ is the domain of the multiplication map.)
The higher-dimensional simplices are all degenerate.
The degeneracy maps insert identity elements, the face maps $G\to M$ are $s$ and $t$,
the face maps $\U \to G$ are $(g,h)\mapsto h$, $(g,h)\mapsto gh$, $(g,h)\mapsto g$.

\section{Expansion and contraction of simplicial complexes}
\label{sec:expansion-contraction}

Let $W_k$ be the ordered simplicial complex
\[ \{ \{0\}, ..., \{k\} \} \cup \{ \{0,1\}, \{1,2\},...,\{k-1,k\} \} .\]
It is shown in figure~\ref{fig:wk}.
A map $W_k \to \Nerve_2G$ is nothing but a well-formed word on $G$ of length $k$.
We now want to express equivalence of words (for the equivalence relation ${\sim}$ of before)
using simplicial terminology.

\begin{figure}
\begin{center}
    \begin{tikzpicture}
        \draw[thick,dotted] (2.4,0) -- (3.4,0);
        \fill (0,0) circle (0.05);
        \fill (1,0) circle (0.05);
        \fill (2,0) circle (0.05);
        \fill (3.8,0) circle (0.05);
        \fill (4.8,0) circle (0.05);
        \begin{scope}[thick,decoration={
            markings,
            mark=at position 0.55 with {\arrow{<}}}
            ] 
            \draw[postaction={decorate}] (0,0) -- (1,0);
            \draw[postaction={decorate}] (1,0) -- (2,0);
            \draw[] (2,0) -- (2.4,0);
            \draw[] (3.4,0) -- (3.8,0);
            \draw[postaction={decorate}] (3.8,0) -- (4.8,0);
        \end{scope}
        \draw[decoration={brace,raise=0.5cm},decorate] (0,0) -- (4.8,0)
        node[pos=0.5,yshift=0.6cm,anchor=south] {$k$ edges};
    \end{tikzpicture}
\end{center}
\caption{The simplicial complex $W_k$ is just a sequence of $k$ edges, head to tail.}
\label{fig:wk}
\end{figure}
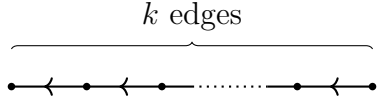

The simplicial complexes that we will work with will be particularly well-behaved:
they will be 2-dimensional, and each edge will have
at most two faces attached to it.
We will say that an edge is a \emph{boundary edge} if it has at most one face attached to it.

Suppose that $S$ is an ordered simplicial complex.
If $\{u,w\}$ is a boundary edge of $S$ with $u < w$,
we can obtain a new ordered simplicial complex $S'$ by
\begin{enumerate}
    \item adding a new vertex $v$ to the simplicial complex
        with $u<v<w$,
    \item adding edges $\{u,v\}$ and $\{v,w\}$, and a face $\{u,v,w\}$.
\end{enumerate}
We will say that $S'$ can be obtained from $S$ by \emph{expansion}.

Suppose that $S$ is an ordered simplicial complex.
If $\{u,v\}$ and $\{v,w\}$ are boundary edges of $S$, with $u < v < w$,
and $\{u,w\}$ is not an edge of $S$,
we can obtain a new ordered simplicial complex $S'$ by
\begin{enumerate}
    \item adding the edge $\{u,w\}$,
    \item adding the face $\{u,v,w\}$.
\end{enumerate}
We will say that $S'$ can be obtained from $S$ by \emph{contraction}.

\begin{definition}
    We say that an ordered simplicial complex $S$ is a \emph{good complex}
    if there is an integer $k\geq 1$ such that $S$ can be obtained
    from $W_k$ by repeated expansion or contraction.
\end{definition}

Note that a good complex comes with a choice of two vertices:
there is one vertex that is minimal among boundary vertices (we will call it the \emph{source} of the good complex)
and one vertex that is maximal among boundary vertices (we will call it the \emph{target} of the good complex).

\begin{definition}
    If $S$ is a good complex, a \emph{boundary path} of $S$ is an ordered subcomplex $S'$ of $S$
    that is contained in its boundary,
    and for which there is an isomorphism of ordered complexes $W_k \to S'$
    that maps the source of $W_k$ to the source of $S$, and the target of $W_k$ to the target of $S'$.
\end{definition}

\begin{remark}
    A good complex can have many boundary paths. For example,
    the complex shown in figure \ref{fig:good-complex-not-disk} has eight of them.
    The good complexes that will be most relevant for us are those that are
    homeomorphic to a disk. These have precisely two boundary paths.

    Every good complex is a bunch of disks and lines attached source-to-target (like in
    \ref{fig:good-complex-not-disk}, where the complex consists of two disks, a line,
    and another disk).
    If the number of such disks is $d$, there are $2^d$ boundary paths.
    Whenever we ask for a good complex to be homeomorphic to a disk, our goal
    is to be able to speak unambiguously about its two boundary paths.
\end{remark}

In other words, a boundary path of $S$ is just a path from source to target along the boundary of $S$
such that the vertices are increasing along the path.

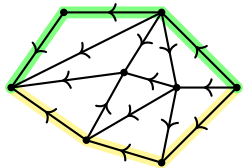
\begin{figure}
\begin{center}
    \begin{tikzpicture}
        \path (3,1) coordinate (a)
              (2,0) coordinate (b)
              (2,2) coordinate (c)
              (1.5,1.2) coordinate (d)
              (2.2,1) coordinate (h)
              (1,0.3) coordinate (e)
              (0.7,2) coordinate (f)
              (0,1) coordinate (g);
        \draw[line width=4.5, yellow!50, line cap=round] (a) -- (b) -- (e) -- (g);
        \draw[line width=4.5, green!50, line cap=round] (a) -- (c) -- (f) -- (g);
        \fill (a) circle (0.05);
        \fill (b) circle (0.05);
        \fill (c) circle (0.05);
        \fill (d) circle (0.05);
        \fill (h) circle (0.05);
        \fill (e) circle (0.05);
        \fill (f) circle (0.05);
        \fill (g) circle (0.05);
        \begin{scope}[thick,line cap=round,decoration={
            markings,
            mark=at position 0.55 with {\arrow{>}}}
            ] 
            \draw[postaction={decorate}] (a)--(b); 
            \draw[postaction={decorate}] (b)--(e);
            \draw[postaction={decorate}] (e)--(g);
            \draw[postaction={decorate}] (a)--(h); 
            \draw[postaction={decorate}] (h)--(b);
            \draw[postaction={decorate}] (h)--(e); 
            \draw[postaction={decorate}] (e)--(d); 
            \draw[postaction={decorate}] (d)--(g);
            \draw[postaction={decorate}] (h)--(d); 
            \draw[postaction={decorate}] (a)--(c); 
            \draw[postaction={decorate}] (c)--(h);
            \draw[postaction={decorate}] (c)--(d); 
            \draw[postaction={decorate}] (c)--(g); 
            \draw[postaction={decorate}] (c)--(f); 
            \draw[postaction={decorate}] (f)--(g);
        \end{scope}
    \end{tikzpicture}
\end{center}
\caption{A good complex. This one can be obtained from the yellow-marked $W_3$ on the bottom
    by (for example) an expansion, a contraction, an expansion,
    a contraction, an expansion, two contractions and an expansion.
    The source is the rightmost vertex, the target is the leftmost vertex.
    This complex has two boundary paths (one on the top, one on the bottom).
    They are highlighted in green and yellow.}
\end{figure}

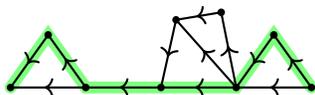
\begin{figure}
\begin{center}
    \begin{tikzpicture}
        \path (0,0) coordinate (a)
              (1,0) coordinate (b)
              (2,0) coordinate (c)
              (3,0) coordinate (d)
              (4,0) coordinate (e)
              (0.5,0.7) coordinate (f)
              (2.2,0.9) coordinate (g)
              (2.8,1) coordinate (h)
              (3.5,0.7) coordinate (i);
        \draw[line width=4.5, green!50, line cap=round] (e) -- (i) -- (d) -- (c) -- (b) -- (f) -- (a);
        \fill (a) circle (0.05);
        \fill (b) circle (0.05);
        \fill (c) circle (0.05);
        \fill (d) circle (0.05);
        \fill (e) circle (0.05);
        \fill (f) circle (0.05);
        \fill (g) circle (0.05);
        \fill (h) circle (0.05);
        \fill (i) circle (0.05);
        \begin{scope}[thick,line cap=round,decoration={
            markings,
            mark=at position 0.55 with {\arrow{>}}}
            ] 
            \draw[postaction={decorate}] (e)--(d);
            \draw[postaction={decorate}] (d)--(c);
            \draw[postaction={decorate}] (c)--(b);
            \draw[postaction={decorate}] (b)--(a);
            \draw[postaction={decorate}] (e)--(i);
            \draw[postaction={decorate}] (i)--(d);
            \draw[postaction={decorate}] (d)--(h);
            \draw[postaction={decorate}] (d)--(g);
            \draw[postaction={decorate}] (h)--(g);
            \draw[postaction={decorate}] (g)--(c);
            \draw[postaction={decorate}] (b)--(f);
            \draw[postaction={decorate}] (f)--(a);
        \end{scope}
    \end{tikzpicture}
\end{center}
\caption{A good complex that is not homeomorphic to a disk.
This one has eight boundary paths. One of them is highlighted.}
\label{fig:good-complex-not-disk}
\end{figure}

\begin{definition}
    If $S$ is a good complex (which we can consider as a simplicial set)
    and $\phi : S \to \Nerve_2G$ is a simplicial map,
    then the words corresponding to the restriction of $\phi$ to the boundary paths of $S$
    are called the \emph{boundary words} of $\phi$.
\end{definition}

The following result now holds more or less by construction.

\begin{proposition}
    Two well-formed words on $G$ are equivalent (for the equivalence relation ${\sim}$ defined earlier)
    if and only if there is a good complex $S$ and a simplicial map $\phi : S \to \Nerve_2G$
    such that both words are boundary words of $\phi$.
    \label{prop:equivalence-geometrically}
\end{proposition}

\begin{proof}
    If the two words $w$ and $w'$ are equivalent, then there is a sequence of expansions and contractions
    of words turning $w$ into $w'$.
    Let $w = w_0, w_1, ..., w_n = w'$ be the shortest sequence of expansions and contractions
    from $w$ to $w'$.
    We can use this sequence to build a good complex $S$ and a map $\phi : S \to \Nerve_2G$,
    as we explain now.

    Start with a map $\phi_0 : W_k \to \Nerve_2G$ representing $w_0$ (so $w_0$ has $k$ letters).
    If $w_1$ is an expansion of $w_0$, we can expand $W_k$ in the corresponding edge
    to get a map $\phi_1 : S_1 \to \Nerve_2G$ with boundary words $w_0$ and $w_1$.
    If $w_1$ is a contraction of $w_1$, we can contract $W_k$ in the corresponding edge
    to get a map $\phi_1 : S_1 \to \Nerve_2G$ with boundary words $w_0$ and $w_1$.
    (Note that the edge that we need to add to the complex for the contraction is not already present.
    This edge can only be present if we have an expansion followed by a contraction
    undoing that expansion, but we chose the shortest sequence of words linking $w$ to $w'$.)

    Conversely, if we have a good complex $S$ and a map $\phi : S \to \Nerve_2G$ with boundary words $w_1$
    and $w_2$, we can use this information to exhibit a sequence of expansions and contractions of words
    turning $w_1$ into $w_2$.
\end{proof}

\section{Monodromy groups}
\label{sec:monodromy}

The integrability of a Lie algebroid is controlled by its \emph{monodromy groups} \cite{crainic-fernandes-lie}.
The monodromy group $\Mon_x$ at $x$ is a subgroup of $\G(\g_x)$ contained in the center $Z(\G(\g_x))$.
Its construction uses the language of $A$-paths, which we will assume the reader
is familiar with.
A detailed exposition can be found in \cite{crainic-fernandes-lie}
and \cite[sections 2.2, 3.3 and 3.4]{lectures-integrability-lie}.
Let us briefly recall the construction of $\Mon_x$.
Let $A$ be a Lie algebroid over $M$.

\begin{definition}
    Suppose $a$ and $a'$ are two $A$-paths over the same base path.
    We say that $a$ and $a'$ are \emph{$A$-homotopic along a trivial sphere}
    if there is an $A$-homotopy between $a$ and $a'$ whose base homotopy
    determines a trivial element in $\pi_2(\Orbit_x)$ (where $\Orbit_x$ is the orbit
    of $x$).
\end{definition}

Suppose $a_0$ is an $A$-path with base path
$\gamma_0$ in the orbit $\Orbit_x$.
If $S : [0,1]^2 \to \Orbit_x$ is such that
\begin{itemize}
    \item $S(0,t) = \gamma_0(t)$,
    \item $S(s,0) = \gamma_0(0)$ and $S(s,1) = \gamma_0(1)$ for all $s$,
\end{itemize}
then we can lift $dS : TI\times TI \to T\Orbit$ to a morphism of algebroids
$a \,dt + b\,ds : TI\times TI \to A$ such that $a(0,t)=a_0(t)$ and such that
$b(s,0)$ and $b(s,1)$ vanish.
We obtain an $A$-path $a_1(t) = a(1,t)$.
Its class modulo $A$-homotopies along trivial spheres
only depends on
the homotopy class (rel boundary) of $S$ \cite[proof op prop. 3.21]{lectures-integrability-lie}.
Let us write $\tilde{\partial}(a_0, S)$ for the class of $a_1$ modulo
$A$-homotopy along trivial spheres.

Now if in the above situation the path $t\mapsto S(1,t)$ is contant,
say at $x\in M$,
then $\tilde{\partial}(a_0, S)$ is (the class of) a $\g_x$-path.
By integration to a $\G(\g_x)$-path,
it determines an element of $\G(\g_x)$, which we call
\[ \partial(a_0,S) = \text{endpoint of integration of $\tilde{\partial}(a_0,S)$} .\]
We will also write $0_x$ for the trivial $A$-path at $x$.
The monodromy group $\Mon_x$ is defined as the image of
\[ \partial : \pi_2(\Orbit_x) \to \G(\g_x) : [\alpha] \mapsto \partial(0_x, \alpha) ,\]
where we consider an element $[\alpha] \in \pi_2(\Orbit_x,x)$ as represented by a map $[0,1]^2\to\Orbit_x$.

\section{Associators are monodromy elements}
\label{sec:assoc-is-mono}

In order to give a precise connection between monodromy groups and associators,
we will need a condition on $G$ that ensures that elements can only be multiplied
if they are sufficiently small (i.e.\ we may need to restrict $G$'s multiplication).

We will need some auxiliary data.
The goal of this data is to associate to every sufficiently small element $g\in G$
an $A$-path $P(g)$ representing it (up to homotopy
along trivial spheres).
Moreover, we will do this in such a way that
if $g_1,g_2 \in G$ satisfy $s(g_1)=s(g_2)$ and $t(g_1)=t(g_2)$, then
$P(g_1)$ and $P(g_2)$ lie over the same base path in $M$.

Choose a function $f:[0,1]\to[0,1]$ that is 0 near 0 and 1 near 1. We will use it
to reparameterize curves to ensure smoothness of concatenations (but it has no further
meaning).
Choose a metric $\langle\cdot,\cdot\rangle$ on $M$.
Pick an $A$-connection $\nabla$ on $A$.
This connection allows us to define an exponential map $\exp_\nabla : A \to G$,
defined in a neighborhood of the zero section $M$ (\cite[section 4.4]{lectures-integrability-lie}).
If $g\in G$ is sufficiently close to $M$, then it is in the image of $\exp_{\nabla}$,
say $g=\exp(\xi)$.
Write $x=s(g)$. Then we have a $G$-path $\tau \mapsto \exp(f(\tau)g)$.
Associated to this $G$-path is an $A$-path (by differentiating and applying the
Maurer-Cartan form), which we call $\tilde{P}(g)$.
Moreover, if $g\in G$ is sufficiently close to $M$, the base path of $\tilde{P}(g)$
is contained completely in a uniformly normal subset of $M$.
(An open subset of a Riemannian manifold is \emph{uniformly normal} if there exists
some $\delta > 0$ such that the subset is contained in a geodesic ball of radius
$\delta$ around each of its points. See \cite[lemma 5.12]{lee-riemannian}.
Important for us will be the fact that between any two points in a uniformly normal
neighborhood, there is a unique minimizing geodesic, and it lies in this neighborhood.)
Now we modify the $A$-path $\tilde{P}(g)$ into an $A$-path $P(g)$ that lies
over the geodesic $\gamma_{s(g),t(g)}$ from $s(g)$ to $t(g)$.
We do this by prescribing
\[ P(g) = \partial(\tilde{P}(g), [\alpha]) ,\]
where $[\alpha] : [0,1] \times [0,1] \to M$ is the map such that
$s\mapsto \alpha(t,s)$ is the geodesic from $(\rho\circ\tilde{P}(g))(t)$
to $\gamma_{s(g),t(g)}(t)$.
This way, we have associated to each $g\in G$ sufficiently close to $M$
an $A$-path $P(g)$ that lies over the geodesic from $s(g)$ to $t(g)$.
Write $P(g_k,...,g_1)$ for the concatenation of $P(g_1)$, ..., $P(g_k)$.
If $w = (w_k,...,w_1)$ is a well-formed word on $G$,
write $P(w) = P(w_k,...,w_1)$.

Write $\Delta[k]$ for the standard $k$-simplex:
\[ \Delta[k] = \left\{ (\lambda_0,...,\lambda_k) \in \R^{k+1} \mid \lambda_i \geq 0, \sum_i \lambda_i = 1 \right\} .\]
The boundary of $\Delta[2]$ consists of three line segments. If $(g,h) \in \U$,
write $\beta_{(g,h)}$ for the map $\partial\Delta[2] \to M$ that sends
the first edge of $\Delta[2]$ to the base path of $P(g)$,
the second edge to the base path of $P(gh)$, and the third edge to the base path of $P(h)$.

Let $S''$ be a neighborhood of the diagonal in $M\times M$ such that for all $(x,y)\in S''$
there is a unique shortest geodesic from $x$ to $y$.
Let
\[ S' = \{ (x,y,z) \in M\times M\times M \mid (x,y), (y,z), (x,z) \in S'' \} .\]
If $(x,y,z) \in S'$, we will write $\beta_{(x,y,z)} : \partial \Delta[2] \to M$ for the
map that sends the edges of $\Delta[2]$ to the three geodesics between $x$, $y$ and $z$.
Consider the space
\[ S = \left\{ (x,y,z,[\alpha]) \in S'\times \left( \frac{\text{maps $\Delta[2] \to M$}}{\text{homotopy rel $\partial$}} \right) \mid
    \alpha{\restriction_{\partial \Delta[2]}} = \beta_{(x,y,z)}
    \right\} .\]

It is a smooth manifold, and the map $\pi : S \to M^3 : (x,y,z,[\alpha]) \mapsto (x,y,z)$
is a local diffeomorphism.
Note that
\[ T = \{ (x,y,z) \in M\times M\times M \mid x=y \text{ or } y=z \text{ or } x=z \} \]
naturally sits inside $S$ by
\[ T\ni (x,y,z)\mapsto (x,y,z,[\text{maps whose image is contained in the image of $\beta_{(x,y,z)}$}]) .\]
Choose a map
$\varphi_2 : U \to S$
from a neighborhood $U$ of $T$ in $M^3$
such that
\begin{itemize}
    \item $\pi \circ \varphi_2 = \id_{U}$,
    \item $\varphi_2{\restriction_T} = \id_T$.
\end{itemize}

\begin{definition}
    Suppose that $G$ is a local Lie groupoid over $M$, and that we have chosen
    auxiliary information $(f,\langle\cdot,\cdot\rangle,\nabla,\varphi_2)$ as above.
    We say that $G$ is \emph{shrunk} if
    \begin{enumerate}
        \item the path $P(g)$ is defined for every $g\in G$,
        \item the isotropy groups $G_x$ are simply connected
            (this implies that we have well-defined maps $G_x \to \G(\g_x)$
            as explained below),
        \item for each $x\in M$, the map $G_x \to \G(\g_x)$
            is injective,
        \item if $(g,h) \in \U$ then $\{s(h),t(h)=s(g),t(g)\} \in U$,
            so that we may write $\varphi_2(g,h) := \varphi_2(s(h),t(h),t(g))$,
        \item for every $(g,h) \in \U$, the class of $P(gh)$
            (as $A$-path modulo $A$-homotopy along trivial spheres)
            equals
            \[ \tilde{\partial}(P(g,h), \varphi_2(g,h)) .\]
            \label{condition-shrunk-essence}
    \end{enumerate}
\end{definition}

The map $G_x \to \G(\g_x)$ works as follows. Let $g\in G_x$.
Take a path in $G_x$ from $x$ to $g$. Differentiate this path
to get a $\g_x$-path. This path integrates to a $\G(\g_x)$-path.
The endpoint of this path is the image of $g$ under the map $G_x \to \G(\g_x)$.
If $G_x$ is simply-connected, this map is well-defined, because any two paths
in $G_x$ from $x$ to $g$ are homotopic (and this homotopy induces an $A$-homotopy
at the level of $A$-paths, which implies that the integrations to $\G(\g_x)$-paths
have the same endpoint).

Every transitive local Lie groupoid has an open neighborhood of the identities
that is shrunk (after restricting).
The first three conditions can be satisfied by shrinking $G$.
The fourth can be satisfied by restricting the multiplication.
The last condition can be satisfied by shrinking and restricting:
it is automatically satisfied for a global Lie groupoid,
and over every contractible neighborhood $U$, the algebroid
is integrable (and over $U$ and near $M$, the local groupoid is
isomorphic to a global integration of $A{\restriction_U}$,
so that the condition holds for these small groupoid elements).

\begin{proposition}
    Suppose that $G$ is a shrunk local Lie groupoid over $M$.
    If $w = (w_1,...,w_k)$ and $w' = (w'_1,...,w'_{k'})$ are equivalent
    words, then the $A$-paths $P(w_1,...,w_k)$ and $P(w'_1,...,w'_{k'})$ are $A$-homotopic.
\end{proposition}
\begin{proof}
    It suffices to prove this if $w$ and $w'$ are elementarily equivalent,
    so we will assume that $w'$ is obtained from $w$ by multiplying the letters
    $w_i$ and $w_{i+1}$.
    The $A$-path
    \[ \tilde{\partial}(P(w_i,w_{i+1}),\varphi_2(w_i,w_{i+1})) \]
    is homotopic to $P(w_i,w_{i+1})$ by construction.
    Because $G$ is shrunk, it is also $A$-homotopic to
    the $A$-path $P(w_iw_{i+1})$.
    We conclude that $P(w_i,w_{i+1})$ and
    $P(w_iw_{i+1})$ are $A$-homotopic, which suffices to prove the result.
\end{proof}

The monodromy group $\Mon_x$ consists precisely of the endpoints
of integrations of $\g_x$-paths that are $A$-homotopic to the trivial path
$0_x$.
This gives us the following corollary.

\begin{corollary}
    Suppose that $G$ is a shrunk local Lie groupoid over $M$.
    Let $x\in M$.
    Consider $G_x$ as a subset of $\G(\g_x)$ using the natural map
    $G_x \to \G(\g_x)$.
    Then $\Assoc_x(G) \subset \Mon_x$.
    \label{cor:assoc-is-mon}
\end{corollary}

\section{Monodromy elements are associators}
\label{sec:mono-is-assoc}

Suppose that $G$ is shrunk.
If we have a simplicial map $[1] \to \Nerve_2G$,
we can associate to it a curve in $M$ as before:
we take the base path of the curve $P(g)$ where $g$ is the image of the 1-cell.
If we have a simplicial map $[2] \to \Nerve_2G$,
we get a homotopy class of maps $\Delta[2] \to M$ rel boundary,
namely $\varphi_2(g,h)$ as before, where $(g,h) \in \U$ is the image of the
2-cell.
More generally, if we have a good complex $S$ and a simplicial map $\phi : S \to \Nerve_2G$,
we get an induced homotopy class of maps $\realization{S} \to M$ rel boundary,
where $\realization{S}$ is the geometric realization of $S$,
by mapping each 1-cell and 2-cell as above, which we call $[\phi]$.

\begin{lemma}
    Let $G$ be a shrunk local Lie groupoid.
    Suppose $S$ is a good complex homeomorphic to a disk,
    and $\phi : S\to \Nerve_2G$ is a simplicial map.
    Let $w$ and $w'$ be the boundary words of $\phi$.
    Then
    \[ \tilde{\partial}(P(w),[\phi]) \text{ and } P(w') \]
    are $A$-homotopic along a trivial sphere.
    \label{lem:monodromy-word}
\end{lemma}
\begin{proof}
    If $S$ has only one face,
    this follows immediately from condition (\ref{condition-shrunk-essence})
    in the definition of a shrunk local groupoid.
    By induction on the number of faces of $S$, the statement follows.
\end{proof}

\begin{proposition}
    Suppose that $G$ is a shrunk local Lie groupoid over $M$.
    Let $x\in M$.
    Consider $G_x$ as a subset of $\G(\g_x)$ using the natural map
    $G_x \to \G(\g_x)$.
    Suppose that $S$ is a good complex homeomorphic to a disk,
    and that
    \[ \phi : S \to \Nerve_2G \]
    is a simplicial map sending the boundary to $x$.
    (In particular, $\phi$ induces a class $[\phi] \in \pi_2(M,x)$.)
    Suppose that one of the boundary words of $\phi$ is $(x)$,
    and that the other boundary word has all letters in $G_x$.
    Then $\partial([\phi]) \in \Assoc_x(G)$.
    \label{prop:mon-is-assoc}
\end{proposition}
\begin{proof}
    Let $w = (w_1,...,w_k)$ be the other boundary word of $\phi$.
    By lemma \ref{lem:monodromy-word}, $\partial([\phi])$ is the endpoint
    of the integration of $P(w) = P(w_1,...,w_k)$ to a $\G(\g_x)$-path.
    Now working in $\G(\g_x)$, we see that
    \[ \partial([\phi]) = \prod_{i=1}^k w_i .\]
    By lemma \ref{lem:smalllemma} below, applied to $G_x \subset \G(\g_x)$, we have
    \[ (\partial([\phi])) \sim (w_1,...,w_k) \quad\text{in $W(G_x)$.}\]
    In particular, this equivalence holds in $W(G)$.
    However, the map $\phi$ shows that $(w_1,...,w_k) \sim (x)$
    in $W(G)$
    (by proposition \ref{prop:equivalence-geometrically}).
    This implies that $(\partial([\phi])) \sim (x)$ in $W(G)$,
    and so $\partial([\phi]) \in \Assoc_x(G)$.
\end{proof}

\begin{lemma}
    Let $\G$ be a source-simply connected Lie groupoid,
    and let $U \subset \G$ be an open neighborhood of the identities,
    considered as a local Lie groupoid.
    Suppose that $g_1, \ldots, g_n \in U$ are such that their product
    $g = g_1\cdots g_n$ lies in $U$.
    Then in $W(U)$ we have
    \label{lem:smalllemma}
    \[ (g_1,\ldots,g_n) \sim (g) .\]%
\end{lemma}
\begin{proof}
    Apply the functor $\AsCo$ to the inclusion $U \into \G$.
    We get $\AsCo(U) \to \AsCo(\G) \cong \G$, which is a morphism of Lie groupoids.
    It is an isomorphism
    at the level of Lie algebroids,
    and $\G$ is source-simply connected, so this map is an isomorphism.
    In particular, it is an injection.
    Because the words $(g_1, ..., g_n), (g) \in W(U)$ are mapped to the same element of $\G$,
    they must represent the same element of $\AsCo(U)$. This means that they are equivalent.
\end{proof}

Note that this lemma is false if $\G$ is not $s$-simply connected.
For example, if we take $U = (-0.5,0.5)/\Z \subset \R/\Z = \G$,
then $0.25+0.25+0.25+0.25=0$ in $\G$, but $(0.25,0.25,0.25,0.25) \not\sim (0)$ as words
in $W(U)$.

We now state and prove the main theorem relating associators to monodromy.

\begin{maintheorem}
    Suppose that $G$ is a shrunk local Lie groupoid over $M$. Let $x\in M$.
    Consider $G_x$ as a subset of $\G(\g_x)$ using the natural map
    $G_x \to \G(\g_x)$.
    \label{thm:assoc-mono} 
    Then
    \[ \Assoc_x(G) = \Mon_x \cap G_x .\]
\end{maintheorem}

\begin{proof}
    The inclusion $\subset$ was corollary \ref{cor:assoc-is-mon}.
    We now prove the inclusion $\supset$.
    By proposition \ref{prop:mon-is-assoc}, it suffices to show that for every
    $[\alpha] \in \pi_2(M,x)$,
    there is a good complex $S$, homeomorphic to a disk,
    and a simplicial map $\phi : S \to \Nerve_2G$, such that
    \begin{enumerate}
        \item the induced map $\realization{S} \to M$ maps the boundary of
            $\realization{S}$ to $x$,
        \item one of the boundary words of $\phi$ is $(x)$,
        \item $[\phi] = [\alpha]$ as elements of $\pi_2(M,x)$.
    \end{enumerate}

    Write $\Delta = \{ (x,y) \in \R^2 \mid x\geq 0, y\geq 0, x+y\leq 1\}$.
    Represent $[\alpha]$ as a map $\alpha : \Delta \to M$, mapping the boundary of $\Delta$ to $x$.
    Fix a splitting $\sigma : TM \to A$ of the anchor.
    If $e$ is a sufficiently short oriented line segment in $\Delta$,
    from a point $v_1$ to a point $v_2$,
    we will associate to it an element $g_e$ of $G$, as follows.
    There is a tangent vector $V_e$ in $T_{\alpha(v_1)}M$ such that
    $t(\exp(\sigma(V_e))) = \alpha(v_2)$ (if $e$ is sufficiently short).
    Take $g_e = \exp(\sigma(V_e))$.
    The base path of $P(g_e)$ is then a geodesic from $\alpha(v_1)$ to $\alpha(v_2)$.

    If $k$ is a positive integer, we can subdivide $\Delta$ into $k$\footnote[2]{this number should be squared} triangles,
    as shown in figure~\ref{fig:triangulation}.
    By choosing $k$ large enough, we can ensure that
    \begin{enumerate}
        \item every edge $e$ of the triangulation is short enough to ensure that
            $g_e$ is defined,
        \item if $e_3,e_2,e_1$ are the edges of a face, the products
            $(g_{e_3}g_{e_2})g_{e_1}$ and $g_{e_3}(g_{e_2}g_{e_1})$ are both defined,
        \item if $e_3,e_2,e_1$ are the edges of a face, and $f_m,...,f_1$ are edges
            of the triangulation forming a path
            of total Euclidean length at most 10 (as measured in $\R^2$)
            with $t(e_3) = s(e_1) = t(f_m)$,
            then the product
            \[\tag{$\dag$}\label{eq:laceproduct} (g^{-1}_{f_1} \cdots (g^{-1}_{f_m} (g_{e_3} g_{e_2} g_{e_1}) g_{f_m}) \cdots g_{f_1}) \]
            is defined (the brackets that are not specified may be put in in any way).
    \end{enumerate}
    If $k$ is large enough to guarantee these properties, we get an induced
    element of $\pi_2(M,x)$, as follows:
    we can map the 1-skeleton of the triangulated $\Delta$ to $M$ by mapping
    each vertex of the triangulation to its image under
    $\alpha$, and mapping each edge $e$ to the base path of $P(g_e)$.
    For each face of the triangulation, there is a preferred homotopy class of map
    into $M$ (rel boundary), namely the one determined by $\varphi_2$.
    All in all, we get a well-defined homotopy class of map from $\Delta$ to $M$,
    rel 1-skeleton. In particular, this defines a homotopy class of map from $\Delta$ to $M$,
    rel boundary, and hence an element of $\pi_2(M,x)$.
    After choosing $k$ larger if necessary, we may assume that this class is precisely $[\alpha]$.

    Let us write $\eps_{3k}, ..., \eps_{1}$ for the $3k$ edges at the circumference of $\Delta$,
    in counterclockwise order starting at the origin $(0,0)$, and with the corresponding orientation
    (so $\eps_k,...,\eps_1$ are oriented to the right, $\eps_{2k},...,\eps_{k+1}$ to the upper left,
    and $\eps_{3k},...,\eps_{2k+1}$ down).
    We will work with ordered sequences of oriented edges of the triangulation.
    Such a sequence of oriented edges $(e_n, ..., e_1)$ is \emph{well-formed}
    if $e_{i}$ ends where $e_{i+1}$ starts (for all $i$).
    A \emph{block} will be a well-formed sequence (possibly empty) of the form
    \[ (e_1^{-1}, ..., e_n^{-1}, e_n, ..., e_1) ,\]
    where $e_i^{-1}$ is the edge $e_i$ with its opposite orientation.
    This name is chosen to be the same as in the proof of theorem \ref{thm:smoothness-of-asco}.
    If $F$ is one of the faces of the triangulation, a \emph{lace around $F$} will be a
    sequence of the form
    \[ (e_1^{-1}, ..., e_n^{-1}, c, b, a, e_n, ..., e_1) ,\]
    where $a, b, c$ are the three sides of $F$, oriented counterclockwise,
    and $e_1$ starts at $(0,0)$. (So it loops around $F$ once, and encloses nothing else.)
    The words $(e_n,...,e_1)$ and $(e_1^{-1},...,e_n^{-1})$ are the \emph{ends} of the lace.

    Now let $(e_n, ..., e_1)$ be a sequence of oriented edges of the triangulation
    that simultaneously satisfies the conditions
    \begin{enumerate}[(a)]
        \item the sequence can be obtained from $(\eps_{3k},...,\eps_2,\eps_{1})$ by repeated insertion
            of blocks,
        \item the sequence is a concatenation of laces $\ell_N, ..., \ell_1$,
            one around each face of the triangulation,
            and each end of a lace has at most $2k$ edges (which certainly ensures that
            the total euclidean length of the edges of that end is less than 10).
    \end{enumerate}
    Such a sequence always exists.
    Figure \ref{fig:sequence-in-triangulation} shows an example for $k=2$.
    The video at
    \url{http://math.uiuc.edu/~michiel2/edge-sequence/}
    illustrates an appropriate sequence of edges for $k=5$,
    which can immediately be generalized to
    work for any value of $k$.
    Note that for each of the laces, the product of the elements associated to the edges
    is defined if we use the order as in (\ref{eq:laceproduct}).
    We will call this the \emph{product} of the lace.

    Now this sequence allows us to build up the required good complex $S$ and
    the simplicial map $\phi : S \to \Nerve_2G$.
    Start with the complex $W_1$, mapped to the element $x$.
    Expand $3k-1$ times, to get a good complex with two boundary paths: one is the original $W_1$,
    the other is a copy of $W_{3k}$. On the complex we have so far,
    $\phi$ will still map everything to $x$.

    Recall that $(e_n,...,e_1)$ can be obtained from $(\eps_{3k},...,\eps_1)$ by insertion of blocks.
    This means that the word $(g_{e_n},...,g_{e_1})$ can be obtained from $(x,...,x)$ by expansion.
    Expanding our good complex further, we therefore get a complex and a map $\phi$
    with boundary words $(x)$ and $(g_{e_n},...,g_{e_1})$.

    Now we group the $(e_n,...,e_1)$ by laces.
    Because the corresponding products are defined, we may now continue building our good complex using
    contractions.
    This results in a complex and map $\phi$ with boundary words
    $(x)$ and $(p_N,...,p_1)$,
    where $p_i$ is the product of the $i$'th lace.

    By construction, this simplicial map $\phi$ induces $[\alpha]$, proving the result.
\end{proof}

\begin{figure}
    \centering
    \begin{tikzpicture}[scale=0.6]
        \foreach \x in {0,...,7}
        {\pgfmathtruncatemacro{\maxy}{7-\x}
          \foreach \y in {0,...,\maxy}
            \draw (\x,\y) -- (\x+1,\y) -- (\x,\y+1) -- (\x,\y);}
    \end{tikzpicture}
    \caption{The triangulation of $\Delta$, shown here for $k=8$.}
    \label{fig:triangulation}
\end{figure}
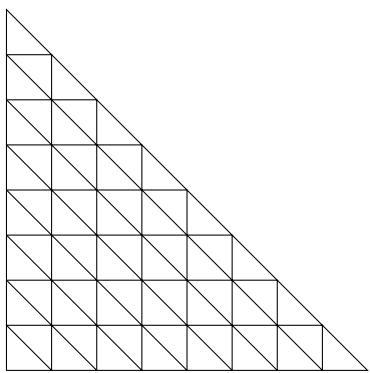

\begin{figure}
    \begin{center}
        \begin{tikzpicture}[scale=3]
            \pgfmathsetlengthmacro{\tick}{0.3mm}
            \draw[rounded corners=1] (0,-2 * \tick) -- (1cm,-2*\tick) -- (2cm+2*\tick,-2*\tick) --
            (1cm+\tick,1cm-\tick) -- (1cm+\tick,-\tick) -- (0,-\tick) -- (0,0) -- (1cm,0) -- (1cm,1cm) --
            (0,2cm) -- (0,1cm) -- (1cm-\tick, 1cm) -- (1cm-\tick,\tick) -- (0,\tick) -- (0,2*\tick) --
            (1cm-2*\tick, 2*\tick) -- (1cm-2*\tick, 1cm-\tick) -- (\tick, 1cm-\tick) --
            (1cm-3*\tick,3*\tick) -- (0,3*\tick) -- (0,4*\tick) -- (1cm-5.4142*\tick, 4*\tick) --
            (0,1cm-1.4142*\tick) -- (0,5*\tick);
            \draw[fill=black] (0,-2*\tick) circle[radius=0.5*\tick]; 
            \draw[fill=white] (0,5*\tick) circle[radius=0.5*\tick]; 
            \node at (1.33,0.27) {$F_1$};
            \node at (0.33,1.3) {$F_2$};
            \node at (0.67,0.67) {$F_3$};
            \node at (0.28,0.33) {$F_4$};
        \end{tikzpicture}
    \end{center}
    \caption{A sequence of edges that satisfies the conditions in the proof of
        theorem \ref{thm:assoc-mono}, for $k=2$.
        The sequence is shown here as a path along the edges
        of the triangulation that starts at the black dot (the first edge is $\eps_1$)
        and ends at the white dot (the last edge is $\eps_{3k}$).
        The sequence is of the form $(\text{lace around $F_1$}, \text{lace around $F_2$},
        \text{lace around $F_3$},\text{lace around $F_4$})$, and simultaneously of the form
        $(\eps_1, \eps_2, \eps_3, \text{block}, \eps_4, \eps_5, \text{block}, \text{block}, \eps_6)$.}
    \label{fig:sequence-in-triangulation}
\end{figure}
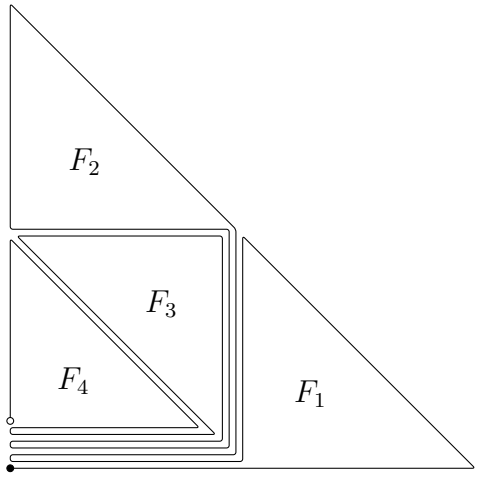

%% file: chapters/appendix.tex
\chapter{Appendix: Computer code}

The following code renders a video that illustrates
the triangulation of figure \ref{fig:triangulation}
and a sequence of edges that satisfies the conditions in the proof of theorem
\ref{thm:assoc-mono}.
To execute this code, Python 3 is required, with the numpy and matplotlib libraries
installed.
\vspace*{0.5em}

\begin{python}
import matplotlib.pyplot as plt
import matplotlib.animation as animation
import numpy
from math import ceil

k = 5
fps = 30
speed = 3.0

# lace around a lower triangle
def lower_lace(i, j):
    for x in range(1,i+1):
        yield (x,0)
    for y in range(0,j+1):
        yield (i+1,y)
    yield (i,j+1)
    yield (i,j)
    yield (i+1,j)
    for y in range(j,0,-1):
        yield (i+1, y-1)
    for x in range(i+1,0,-1):
        yield (x-1,0)

# lace around an upper triangle
def upper_lace(i,j):
    for x in range(1,i+1):
        yield (x,0)
    for y in range(0,j+1):
        yield (i+1,y)
    yield (i+1, j+1)
    yield (i,j+1)
    yield (i+1,j)
    for y in range(j,0,-1):
        yield (i+1, y-1)
    for x in range(i+1,0,-1):
        yield (x-1,0)

# laces for all triangles in a column
def column(i):
    yield from lower_lace(i,k-i-1)
    for j in range(k-i-2, -1, -1):
        yield from upper_lace(i,j)
        yield from lower_lace(i,j)

# the entire sequence
def everything():
    yield (0,0)
    for i in range(k-1,-1,-1):
        yield from column(i)

path = list(everything())
xs = [p[0] for p in path]
ys = [p[1] for p in path]

def animate(n, point):
    point.set_data([
        numpy.interp(speed*n/fps, range(len(xs)), xs),
        numpy.interp(speed*n/fps, range(len(ys)), ys)])

if __name__ == "__main__":
    fig = plt.figure()
    fig.set_facecolor('white')
    plt.plot(xs, ys, color='#cccccc', linewidth=2.0)
    plt.plot([i for i in range(k+1) for j in range(k+1-i)],
        [j for i in range(k+1) for j in range(k+1-i)],
        'o', color='#cccccc')
    plt.axis([-1, k+1, -1, k+1])
    plt.axis('off')
    ax = plt.gca()
    plt.tight_layout(pad=0.1)
    point, = ax.plot([0],[0],'o', color='#000000',
        markersize=10)
    ani = animation.FuncAnimation(fig, animate,
        list(range(ceil(fps*len(path)/speed))),
        fargs=[point], interval=(1000/fps))

    # save animation to a file
    FFMpegWriter = animation.writers['ffmpeg']
    metadata = dict(title='Edge sequence')
    writer = FFMpegWriter(fps=fps, metadata=metadata)
    ani.save('edge-sequence.mp4', writer=writer)
    # show animation on screen
    plt.show()
\end{python}